%% file: random-graph-coloring-CI.tex
\documentclass[letterpaper, 11pt]{amsart}
\usepackage[foot]{amsaddr}
\usepackage{macros}
\usepackage{fullpage}
\usepackage{notations}

\newboolean{doubleblind}
\setboolean{doubleblind}{False}

\title{Rapid Mixing for Sampling Colorings of Random Graphs up to Tree Uniqueness Threshold}
\ifdoubleblind
    \author{Author(s)}
\else
    \author{Zhidan Li, Kuan Yang}
    \address[Zhidan Li, Kuan Yang]{School of Computer Science, Shanghai Jiao Tong University, Shanghai, China. \textnormal{Email: \url{yueEnTeRnAL@sjtu.edu.cn}, \url{kuan.yang@sjtu.edu.cn}}.}
\fi

\begin{document}

\begin{abstract}
    We study the mixing time of the single-site Glauber dynamics for sampling proper $q$-colorings on sparse Erd\H{o}s-R\'{e}nyi random graphs $\mathcal{G}(n, d/n)$.
    We prove that, whenever $q \ge (1 + o_d(1))d$, Glauber dynamics mixes in polynomial time with high probability over $\mathcal{G}(n, d/n)$ for all sufficiently large integers $n$.
    This improves the previous bound $q = (\alpha^* + o_d(1))d$ where $\alpha^* = 1.7632...$ is the unique solution of $x = \exp(1/x)$ on $[1, \infty)$ by Efthymiou, Hayes, \v{S}tefankovi\v{c}, and Vigoda (SODA, 2018) and asymptotically matches the uniqueness threshold on infinite regular trees.

    Our proof has three main ingredients. First we establish logarithmic bounds on the spectral and coupling independence constants of the Gibbs distribution by extending the tree analysis of Chen, Liu, Mani, and Moitra (J. ACM, 2026).
    Second, we extend the entropy-factorization framework for two-spin systems on random graphs, developed by Bez\'{a}kov\'{a}, Galanis, Goldberg, and \v{S}tefankovi\v{c} (TALG, 2024), to proper colorings by introducing a buffer around high-degree vertices. Finally, we compare the resulting heat-bath block dynamics with the single-site Glauber dynamics.
\end{abstract}

\maketitle

% \tableofcontents
\newpage

\input{section/introduction.tex}
\input{section/notation.tex}
\input{section/random-graph-property.tex}
\input{section/coupling-independence.tex}
\input{section/rapid-mixing.tex}

\bibliographystyle{alpha}
\bibliography{refs}

\appendix
\input{appendix/CLMM-clarification.tex}
\input{appendix/Poisson-estimation.tex}

\end{document}

%% file: section/introduction.tex
\section{Introduction} \label{sec:introduction}

The model of proper colorings, which is often viewed as the anti-ferromagnetic Potts model at zero temperature, plays a fundamental role in combinatorics, probability theory, statistical physics and theoretical computer science.
Given a graph $G = (V, E)$ and a color number $q \ge 2$,  a \emph{proper $q$-coloring} is an assignment of colors from $[q]\defeq\set{1,\ldots,q}$ to the vertices of $G$ such that adjacent vertices receive different colors.
Let $\Omega_{G,q}$ denote the set of proper $q$-colorings of $G$, and let $\mu_{G,q}$ be the uniform distribution on $\Omega_{G,q}$.
Efficient sampling from $\mu_{G, q}$ is a basic and important algorithmic problem, closely related to approximate counting and the study of Gibbs measures. 
For every even $q \ge 4$ and $\Delta > q$, Galanis, \v{S}tefankovi\v{c}, and Vigoda \cite{GSV15} showed that, unless $\NP = \RP$, there is no $\FPRAS$ for counting proper $q$-colorings even on triangle-free $\Delta$-regular graphs.

In this paper, we study sampling proper colorings of \emph{sparse \Erdos--\Renyi random graphs}.
For a fixed real number $d>1$, the random graph $\G(n,d/n)$ has vertex set $[n]$ and includes each pair in $\binom{[n]}{2}$ independently with probability $d/n$.
Its average degree converges to $d$ with high probability, whereas its maximum degree is $\Theta(\ln n/\ln\ln n)$ with high probability.
Thus, the maximum degree and the typical local branching rate are on very different scales.
This distinction makes it natural to seek sampling algorithms whose color requirement depends on the average degree $d$ rather than on the maximum degree.

The statistical physics of random graph colorings provides several benchmarks for this problem.
For large $q$, the cavity method predicts a \emph{clustering} transition at average degree of order $q\ln q$, followed by condensation and, eventually, a loss of colorability near $2q\ln q$ \cite{ZK07,BS20}.
These transitions describe the global geometry of the space of proper colorings.
By contrast, \emph{Gibbs uniqueness} concerns the sensitivity of local marginals to distant boundary conditions and provides a natural benchmark for the behavior of local Markov chains.
For the infinite $\Delta$-regular tree, the Gibbs uniqueness threshold for proper colorings is $q\ge\Delta+1$ for every integer $\Delta\ge2$ \cite{Jonasson02}.
Its branching factor is $\Delta-1$, which corresponds to the local branching rate $d$ of $\G(n,d/n)$.
Related local-stability calculations for sparse \Erdos--\Renyi random graphs also suggest the bound $q=d+1$ as a natural uniqueness threshold \cite{ZK07}.
These observations motivate the question of whether a local-update Markov chain, and in particular the natural \emph{single-site Glauber dynamics}, can mix rapidly when the number of colors and the average degree are in the tree-uniqueness regime.

Over the past several decades, substantial progress has been made on almost uniform sampling of proper colorings of sparse random graphs with a number of colors depending on the average degree rather than the maximum degree.
Mossel and Sly \cite{MS10} established rapid mixing of a Gibbs sampler when $q=q(d)$ is sufficiently large.
Efthymiou \cite{Efthymiou14} subsequently introduced a weighted block decomposition that reduced the color requirement to a quantity linear in $d$.
Combining block decomposition with local uniformity properties of the Gibbs distribution, Efthymiou, Hayes, \v{S}tefankovi\v{c}, and Vigoda \cite{EHSV18} proved that the single-site Glauber dynamics mixes in polynomial time with high probability whenever $q\ge (\alpha^*+o_d(1))d$, where $\alpha^*=1.7632\ldots$ is the unique solution of $x=\e^{1/x}$ on $[1,\infty)$.

However, a different line of work has reached a smaller color requirement if we relax the requirements of accuracy.
Efthymiou \cite{Efthymiou16} developed a polynomial-time non-MCMC algorithm which, for $q\ge (1+o_d(1))d$, produces a proper coloring whose distribution is within $n^{-\Omega(1)}$ total variation distance of $\mu_{G,q}$, with high probability over the random graph.
Related deterministic work \cite{Efthymiou20} provides a $\bigl(1\pm n^{-\Omega(1)}\bigr)$-approximation to the free energy of the coloring model in the same regime.
However, the sampling accuracy in \cite{Efthymiou16} (and the counting accuracy in  \cite{Efthymiou20}) is tied to the instance size, and these results do not establish polynomial-time mixing of the single-site Glauber dynamics or yield an $\FPAUS$ with an arbitrarily prescribed accuracy parameter.
Consequently, a fundamental question remained open:
\begin{center}
    \emph{Does there exist an $\FPAUS$ for proper $q$-colorings on $\G(n,d/n)$ when $q\ge (1+o_d(1))d$?\\
    In particular, does the single-site Glauber dynamics mix in polynomial time?}
\end{center}

We answer this question affirmatively, improving the previous rapid-mixing condition $q\ge (\alpha^*+o_d(1))d$ to $q\ge (1+o_d(1))d$ and thereby asymptotically reaching the uniqueness threshold of the corresponding regular tree.
\begin{theorem} \label{thm:rapid-mixing}
    Fix $\eta \in (0, 1)$.
    For sufficiently large constants $d > 1$ and every integer $q \ge (1 + \eta)d$, with probability $1 - o_n(1)$ for $G \sim \G(n, d/n)$, the single-site Glauber dynamics for proper $q$-colorings of $G$ has mixing time 
    \[
        \tau_{\mix}(\PGD,\eps) \le n^{C}\bigl(1+\ln(1/\eps)\bigr)     \qquad\text{for every }\eps\in(0,1),
    \]
    where $C = C(\eta, d, q) > 0$ is a constant independent of $n$ and $\eps$.
\end{theorem}

\Cref{thm:rapid-mixing} closes the asymptotic gap between the previously known rapid-mixing bound for Glauber dynamics on sparse random graphs and the tree uniqueness threshold.
In particular, it shows that the presence of vertices whose degrees grow with $n$ does not preclude polynomial-time mixing when the number of colors exceeds the average degree by an arbitrarily small fixed fraction and $d$ is sufficiently large.
%Since the Glauber dynamics has the uniform coloring measure as its stationary distribution, the theorem also yields an $\FPAUS$ with high probability in this regime.

\subsection{Proof overview}
The proof of \Cref{thm:rapid-mixing} has three main ingredients: a logarithmic bound on coupling and spectral independence, an entropy-factorization argument using a buffer around high-degree vertices, and a comparison between heat-bath block dynamics and the original single-site Glauber dynamics.

\subsubsection*{First step: spectral and coupling independence}
Our starting point is \emph{spectral independence}, a framework for translating bounds on conditional influences into rapid mixing.
\begin{definition}[Influence matrix and spectral independence \cite{CGSV21}]
    Given a distribution $\mu$ over $\Omega \subseteq [q]^V$, for every boundary $\Lambda \subseteq V$ and feasible boundary condition $\tau \in [q]^{\Lambda}$ (\IE, $\mu(\cdot \midl \tau)$ is well-defined), its influence matrix is indexed by pairs $u \gets i$ with $u \in V\setminus\Lambda$ and $\mu(u \gets i \midl \tau)>0$.
    For two such pairs $u \gets i$ and $v \gets j$ with $u\ne v$, define
    \begin{align*}
        \Psi_{\mu}^{\tau}(u \gets i, v \gets j) \defeq \mu(v \gets j \midl u \gets i, \tau) - \mu(v \gets j \midl \tau)
    \end{align*}
    with zero on all entries with $u = v$.\footnote{An alternative definition for $\Psi_{\mu}^{\tau}$ is introduced in \cite{FGYZ22}.}
    We say $\mu$ is \emph{$\gamma$-spectrally independent} if for every feasible $\Lambda \subseteq V$ and $\tau \in [q]^{\Lambda}$, $\lambda_{\max}(\Psi_{\mu}^{\tau}) \le \gamma$.
\end{definition}

Spectral independence was first introduced in \cite{ALOG24}, building on the framework of high-dimensional expanders \cite{AL20}, to establish rapid mixing for the hard-core model up to the uniqueness threshold.
It was subsequently developed for proper colorings on bounded-degree triangle-free graphs \cite{FGYZ22,CGSV21} and applied to a variety of spin systems, yielding rapid or optimal mixing bounds and further local-to-global principles \cite{CLV23,CLV26,JPV22,BCCPSV22,CFYZ22,CE22}.
For sparse random graphs, \cite{BGGS24} established spectral independence for certain two-spin systems, including the hard-core model, using a branching-value analysis, while \cite{EF23} sharpened the bounds and obtained rapid mixing of Glauber dynamics.
It still remains open how to establish the spectral independence for multi-spin systems on random graphs.

We address these challenges through \emph{coupling independence}, introduced in \cite{CZ23}.
\begin{definition}[Coupling independence \cite{CZ23}] \label{def:coupling-independence}
    A distribution $\mu$ over $\Omega \subseteq [q]^V$ is said to be \emph{$C$-coupling independent} if for every boundary $\Lambda \subseteq V$ and every feasible partial pinnings $\sigma, \tau$ on $\Lambda$ which differ only at exactly one coordinate, it holds that
    \begin{align*}
        \W_1(\mu(\cdot \midl \sigma), \mu(\cdot \midl \tau)) \le C
    \end{align*}
    where
    \begin{align*}
        \W_1(\mu, \nu) \defeq \inf_{\gamma\; \text{is a coupling of}\; (\mu, \nu)} \E[(X, Y) \sim \gamma]{\sum_{u \in V} \id{X(u) \neq Y(u)}}
    \end{align*}
    is the \emph{Wasserstein-$1$ distance} of $\mu$ and $\nu$.
    Denote by $\CI(\mu)$ the coupling independence of $\mu$, \IE, the minimal $C$ such that $\mu$ is $C$-coupling independent.
\end{definition}

Coupling independence is closely related to spectral independence and can be more convenient to establish by recursive coupling arguments.
Recent work has used this approach to study the Glauber dynamics for proper colorings on bounded-degree graphs and graphs with large girth \cite{CLMM26,CF25}.

To show \Cref{thm:rapid-mixing}, we establish the following technical theorem, which is the first result extending coupling independence of proper coloring models to sparse \Erdos-\Renyi random graphs.
\begin{theorem}[Coupling independence] \label{thm:coupling-independence}
    Fix $\eta \in (0, 1)$.
    For every sufficiently large positive real $d > 1$ and every integer $q \ge (1 + \eta)d$, with probability at least $1 - o_n(1)$ over the \Erdos-\Renyi random graph $G \sim \G(n, d/n)$ for all sufficiently large integers $n$, it holds that
    \begin{align*}
        \CI(\mu_{G, q}) \le C \ln n
    \end{align*}
    where $C = C(\eta, d) > 0$ is a constant depending only on $\eta$ and $d$.
\end{theorem}

As a consequence, we obtain a spectral-independence bound of the same order.
\begin{corollary}[Spectral independence] \label{cor:spectral-independence}
    Fix $\eta \in (0, 1)$.
    For every sufficiently large positive real $d > 1$ and every integer $q \ge (1 + \eta)d$, with probability at least $1 - o_n(1)$ over the \Erdos-\Renyi random graph $G \sim \G(n, d/n)$ for all sufficiently large integers $n$, $\mu_{G, q}$ is $O_{\eta, d}(\ln n)$-spectrally independent.
\end{corollary}

\begin{remark}
    The logarithmic orders in \Cref{thm:coupling-independence,cor:spectral-independence} are sharp for every fixed admissible pair $(d,q)$.
    In \Cref{prop:independence-lower-bounds}, we show that, with high probability, suitable feasible pinnings of an isolated tree component give $\Omega_{d,q}(\ln n)$ lower bounds for both independence parameters.
\end{remark}

The proof of \Cref{thm:coupling-independence} combines the tree-based coupling method of \cite{CLMM26} with the local structure of sparse random graphs.
Although neighborhoods of logarithmic radius are trees or unicyclic with high probability, high-degree vertices prevent a direct application of bounded-degree correlation-decay estimates.
We overcome this difficulty by introducing a message-passing process on locally tree-like graphs whose nodes are marked as good or bad.
Good nodes lie sufficiently far from high-degree vertices and enjoy strong spatial mixing and total influence decay inherited from bounded-degree trees; bad nodes are handled by a more conservative message recursion.
A comparison with a suitably marked Galton--Watson tree then bounds these messages and yields $O(\ln n)$ coupling independence uniformly over all feasible boundary conditions.

\subsubsection*{Second step: high-degree buffer and entropic factorization}
We next adapt the entropy-factorization framework of \cite{BGGS24} from two-spin systems to proper colorings.

Turning logarithmic spectral independence into a polynomial mixing bound requires special care because the coloring model has hard constraints and the random graph contains vertices of high degree.
We choose a fixed degree cutoff and introduce a one-step buffer consisting of all high-degree vertices and their neighbors.
Letting $U$ denote the remaining low-degree vertices, we study the marginal coloring distribution on $U$.
The buffer ensures that every proper coloring of the induced subgraph on $U$ extends to a proper coloring of the original graph; together with connected conditional supports, marginal boundedness, and spectral independence, this makes it possible to apply a uniform block entropy-factorization theorem in \cite{CLV26,BGGS24}.

To pass from block entropy factorization to a single-site inequality, we analyze the geometry of a randomly chosen block together with the buffered high-degree region.
A percolation argument adapted to the buffer shows that, except for a sufficiently small probability over the random block, every connected component of the induced subgraph has size $O(\ln n)$.
On such components, an elementary entropy tensorization bound incurs only a polynomial loss.
Combining this local bound with uniform block entropy factorization yields an approximate single-site entropy tensorization inequality with a polynomial factor, and hence polynomial relaxation time for the marginal Glauber dynamics on $U$.

\subsubsection*{Final step: Markov chain comparison}
The mixing bound for the marginal distribution does not immediately imply rapid mixing of the original Glauber dynamics, since the high-degree buffer has been marginalized out.
To recover the full coloring distribution, we introduce a heat-bath block dynamics in which each update resamples one vertex of $U$ together with the buffer.
An orthogonal decomposition of $L^2(\mu_{G,q})$ relates the spectral gap of this block dynamics to that of the marginal Glauber dynamics.
We then use the logarithmic size of the relevant connected components to bound the relaxation times of the conditional single-site chains within each block.
Finally, a comparison of Dirichlet forms transfers the polynomial relaxation-time bound to the original single-site Glauber dynamics and completes the proof of \Cref{thm:rapid-mixing}.

\subsection{Discussion}
We conclude the introduction with two directions suggested by our results.

\subsubsection*{Coloring random graphs with fewer colors}
%As mentioned above, the clustering threshold of $\mu_{G, q}$ is $d = q\ln{q}$ predicted by the cavity method in \cite{ZK07} and $G$ is colorable when $d < 2q\ln{q}$ \cite{AN05} with high probability.
%These phenomena indicate a possible algorithmic implication for estimating proper $q$-colorings with fewer colors, for example, $q = d - O(1)$ or $q = \Theta(d/\ln{d})$.
%Understanding how the mixing behavior of natural Markov chains relates to the clustering and condensation transitions remains an important open problem.
For large $q$, the clustering and colorability thresholds of sparse random graph colorings occur at average degrees of order $q\ln q$ \cite{ZK07,AN05}.
These phenomena suggest that the geometry of the coloring space may remain interesting well below the tree uniqueness threshold.
A natural direction is to determine whether efficient sampling is possible with fewer colors, for example when $q=d-O(1)$ or even $q=\Theta(d/\ln d)$.

%\subsubsection*{Random sampling vs. deterministic approximate counting}
%\Cref{thm:rapid-mixing} provides an $\FPAUS$ for $\mu_{G, q}$ with high probability when $q = (1 + o_d(1))d$.
%We remark that this $\FPAUS$ also works for list-coloring instances under the same hypothesis.
%and by self-reducibility, it provides an $\FPRAS$ for $\abs{\Omega_{G, q}}$.
%Recent work \cite{LY26} based on strong spatial mixing and a recursive framework by linear-fractional programming designs an $\FPTAS$ for proper $q$-colorings on random graphs when $q \ge (2 + o_d(1))d$.
%It still remains open whether $\mu_{G, q}$ exhibits strong spatial mixing when $q \ge (1 + o_d(1))d$ and thus there is an $\FPTAS$ for $\abs{\Omega_{G, q}}$.

\subsubsection*{Random sampling vs. deterministic approximate counting}
\Cref{thm:rapid-mixing} gives an $\FPAUS$ for $\mu_{G,q}$ with high probability whenever $q\ge(1+\eta)d$ for any fixed $\eta>0$ and sufficiently large $d$.
On the random-graph event used in the proof, a polynomial-time initial coloring is obtained by first three-coloring every tree or unicyclic component of $G[V^{\HD}]$ and then greedily coloring the remaining vertices, whose degrees are at most $D<q$.
The coupling-independence analysis also applies to the feasible list-coloring instances satisfying the list-size slack condition introduced in \Cref{sec:notation}.
On the deterministic side, recent work \cite{LY26} combines strong spatial mixing with a recursive linear-fractional-programming framework to obtain an $\FPTAS$ for counting proper colorings on sparse random graphs when $q\ge(2+o_d(1))d$.
Note that \Cref{prop:independence-lower-bounds}  implies that the Gibbs distribution does not exhibit strong spatial mixing at the threshold $q=(1+o_d(1))d$.
It remains open whether we can obtain deterministic approximation algorithms for $\abs{\Omega_{G,q}}$ without SSM.

\subsection{Organization of this paper}
\Cref{sec:notation} introduces the notation and the family of feasible list-coloring instances used throughout the paper.
In \Cref{sec:random-graph-property}, we establish the structural properties of sparse \Erdos--\Renyi random graphs needed in the analysis.
\Cref{sec:coupling-independence} proves \Cref{thm:coupling-independence} and \Cref{cor:spectral-independence} by combining the coupling method of \cite{CLMM26} with recursive messages on marked Galton--Watson trees. It also establishes matching logarithmic lower bounds for both independence parameters.
Finally, in \Cref{sec:rapid-mixing}, we complete the proof of  \Cref{thm:rapid-mixing} via entropy factorization and Markov chain comparison.

\subsection*{AI disclosure}
ChatGPT Pro 5.6 Sol and 6 Astra help to generate the proof of coupling independence after the authors provide the rough idea and related references.
The authors manage to accomplish the proof idea of the rapid mixing from spectral independence with GPT's help to handle the high-degree buffers.
Heavy calculations are provided by agents and verified by authors.
The authors rewrite the proof and take full responsibility for the paper.

%% file: section/notation.tex
\section{Notations} \label{sec:notation}

Before all, we fix some notations here to help understand our paper.
Let $\mathbb R$ be all real numbers, $\mathbb R_{> 0}$ be all positive real numbers and $\mathbb N$ be all natural numbers.
For a natural number $n \in \mathbb N$, denote by $[n]$ the set $\set{1, \ldots, n}$.
We use $\e$ to denote the natural logarithm base and $\ln$ to denote the natural logarithm function.
Let $H_N$ be the $N$-th harmonic number, \IE, $H_N = \sum_{j = 1}^N \frac{1}{j}$.
For an event $\mathcal{E}$, we also view it as a subset of the total universe and hence use $\overline{\mathcal{E}}$ to denote its complement.
The indicator function of $\mathcal{E}$ is denoted by $\id{\mathcal{E}}$.

A graph $G$ consists of a vertex set $V(G)$ and an edge set $E(G) \subseteq \binom{V(G)}{2}$.
When the context is clear, we simply use $V$ and $E$ to denote $V(G)$ and $E(G)$ respectively.
For a subset $S \subseteq V$, we use $N_G(S)$ to denote the neighborhood of $S$ in $G$ and set $N_G[S] \defeq S \cup N_G(S)$.
We use $\deg_G(\cdot)$ to denote the degree function of $G$.
Let $\dist_G(\cdot, \cdot)$ be the distance metric in $G$.
Let $B_G(v, \ell)$ be the $\ell$-radius ball centered at $v$ in $G$ and $S_G(v, \ell)$ be the $\ell$-radius sphere at $v$ in $G$, \IE,
\begin{align*}
    B_G(v, \ell) \defeq \set{u \in V(G) \mid \dist_G(u, v) \le \ell}, \quad
    S_G(v, \ell) \defeq \set{u \in V(G) \mid \dist_G(u, v) = \ell}.
\end{align*}
Furthermore, for a vertex subset $S \subseteq V$, let $B_G(S, \ell) = \cup_{v \in S} B_G(v, \ell)$ be the collection of vertices within distance $\ell$ from $S$.
With abuse of notations, we sometimes use the vertex subset $S$ to denote the induced subgraph $G[S]$ when $G$ is clear in context.

Given a finite universe $U$ and an assignment $\sigma : U \to [q]$, we use $\sigma(S)$ to denote the partial assignment of $\sigma$ on $S$.
For a subset $\Lambda \subseteq U$, the notation $\sigma_{\Lambda}$ is abused to denote a partial assignment on $\Lambda$.
When context is clear, we drop braces and use the element itself to denote the singleton set containing it.
Meanwhile, we sometimes use $\sigma(v)$ to denote the color of $\sigma$ on $v$ as well and use the term $v \gets c$ to denote the event that the vertex $v$ is assigned value $c$.

\subsubsection*{Proper colorings and list-coloring instances}
For a graph $G$ and a color number $q \ge 3$, a feasible list-coloring instance $\Phi = (G_\Phi, L_{\Phi} = \set{L_\Phi(v)}_{v \in V_\Phi})$ consists of an induced subgraph $G_\Phi=G[V_\Phi]$ and color lists $L_\Phi(v) \subseteq [q]$ for every $v \in V_\Phi$ such that $\Phi$ admits at least one proper list-coloring and
\begin{align*}
    \forall v \in V_\Phi, \quad \abs{L_\Phi(v)} - \deg_{G_\Phi}(v) \ge q - \deg_G(v).
\end{align*}
Let $\mathcal L(G, q)$ collect all such feasible list-coloring instances, and for each $\Phi\in\mathcal L(G,q)$ let $\mu_\Phi$ be the uniform distribution on its nonempty set of proper list-colorings.
Feasibility is a separate requirement: the list-size slack inequality alone need not imply the existence of a proper list-coloring.
Every feasible pinning of an instance in $\mathcal L(G,q)$ induces another instance in $\mathcal L(G,q)$ by deleting the pinned vertices and removing their colors from the lists of their unpinned neighbors.

\subsubsection*{Probability theory}
We use standard notations for probability theory.
For simplicity, when the context is clear, we abuse the usage of a random variable and its law.
For an event $\mathcal E$, we use the term `$\cdot \;|\; \mathcal E$' to denote random variables conditional on $\mathcal E$.
Let $\mu$ be a probability distribution on $\Omega \subseteq [q]^U$ where $q$ is the number of colors and $U$ is a finite universe.
A partial assignment $\tau$ on a boundary $\Lambda \subseteq U$ is feasible if it can be extended to an assignment on $U$ with positive mass in $\mu$.
We use $\mu^{\tau}$ to denote the probability distribution induced by $\mu$ conditional on pinning $\tau$ on $\Lambda$ and $\mu_{S}^{\tau}$ to denote the marginal of $\mu^{\tau}$ projected onto $S$.
For two random variables $X, Y$, we say $X$ is dominated by $Y$ (denoted by $X \preceq Y$) if there exists a coupling $\gamma$ of $(X, Y)$ such that
\begin{align*}
    \Pr[(X, Y) \sim \gamma]{X \le Y} = 1.
\end{align*}
For $p \ge 1$, define the $\ell_p$-norm of a random variable $X$ as
\begin{align*}
    \norm{X}_p \defeq \E{\abs{X}^p}^{1/p}.
\end{align*}

%% file: section/random-graph-property.tex
\section{Random Graph Properties} \label{sec:random-graph-property}

Before proving our main results \Cref{thm:rapid-mixing,thm:coupling-independence}, we investigate some structural properties of random graphs.
These properties provide some intuitions on sparse random graphs $\G(n, d/n)$, which play crucial roles in our analysis.

\subsection{Locality of random graphs}
In this part, we consider the locality of random graphs.
Let $\LocalTreeOrUnicyclic_{\ell}$ be the event that for every vertex $v \in V$, the ball $B_G(v, \ell)$ is tree or unicyclic.
Set
\begin{align*}
    L \defeq \left\lfloor \frac{\ln n}{1000\ln{(2\e d)}} \right\rfloor.
\end{align*}
% The following lemma shows that the event $\LocalTreeOrUnicyclic_{\ell}$ holds with high probability for some $\ell = O(L)$, whose proof is provided in \Cref{subsec:random-graph-locality}.
The following lemma shows that the event $\LocalTreeOrUnicyclic_{\ell}$ holds with high probability for some $\ell = O(L)$.
\begin{lemma} \label{lem:locally-tree-or-unicyclic}
    For sufficiently large positive integer $n$, it holds that
    \begin{align*}
        \Pr[G \sim \G(n, d/n)]{G \in \LocalTreeOrUnicyclic_{5L}} \ge 1 - o_n(1).
    \end{align*}
\end{lemma}

We firstly provide the following probability of the occurrence of a subgraph containing two small cycles.
\begin{lemma} \label{lem:two-surplus-tree-probability}
    Let $\TwoSurplus_m$ be the event that $G$ contains a connected subgraph on at most $m$ vertices with at least two surplus edges.
    Then for $d > 1$, it holds that
    \begin{align*}
        \Pr[G \sim \G(n, d/n)]{G \in \TwoSurplus_m} \le \frac{\e d^2}{n(\e d - 1)} m^2 (\e d)^m.
    \end{align*}
\end{lemma}
\begin{proof}
    Assume an arbitrary order on vertices and an arbitrary order on all possible edges.
    Fix $k = 1, \ldots, m$.
    For every connected subgraph on $k$ vertices with at least $k + 1$ edges, we construct a witness as: pick a minimal spanning tree firstly and then choose two minimal surplus edges.
    To count the number of such subgraphs, it suffices to count the number of all possible witnesses.
    By Cayley's formula and the bound $\binom{n}{k} \le (\e n/k)^k$, since there are at most $k^4$ possible choices of two surplus edges in a tree of size $k$, the number of possible witnesses on $k$ vertices in $G$ is at most
    \begin{align*}
        \binom{n}{k} k^{k - 2} k^4 \le \frac{\e^k n^k}{k^k} \cdot k^{k + 2} = \e^k n^k k^2.
    \end{align*}
    Therefore
    \begin{align*}
        \Pr[G \sim \G(n, d/n)]{G \in \TwoSurplus_m} &\le \sum_{k = 1}^m \e^k n^k k^2 \cdot \ab(\frac{d}{n})^{k + 1} \\
        &= \sum_{k = 1}^m \frac{d}{n} (\e d)^k k^2 \\
        &\le \frac{\e d^2}{n(\e d - 1)} m^2 (\e d)^m
    \end{align*}
    which concludes the lemma.
\end{proof}

\begin{proof}[Proof of \Cref{lem:locally-tree-or-unicyclic}]
    Suppose that there exists a vertex $v \in V$ such that the $5L$-radius ball $B_G(v, 5L)$ contains two cycles.
    Take a breadth-first spanning tree of the ball and two edges $e_1 = (u, w), e_2 = (x, y)$ (with possibly repeated occurrence of a vertex) outside this spanning tree.
    Consider the subgraph generated by joining tree paths from the endpoints $u, w, x$ and $y$ to $v$ in addition to $e_1$ and $e_2$, \IE,
    \begin{align*}
        v \stackrel{\text{tree path}}{\longrightarrow} u \stackrel{e_1}{\longrightarrow} w \stackrel{\text{tree path}}{\longrightarrow} v \stackrel{\text{tree path}}{\longrightarrow} x \stackrel{e_2}{\longrightarrow} y \stackrel{\text{tree path}}{\longrightarrow} v.
    \end{align*}
    This subgraph is on at most $20L + 1$ vertices and contains two cycles.
    Then it holds that
    \begin{align*}
        \Pr[G \sim \G(n, d/n)]{G \in \overline{\LocalTreeOrUnicyclic_{5L}}} \le \Pr[G \sim \G(n, d/n)]{G \in \TwoSurplus_{20L + 1}}.
    \end{align*}
    Apply \Cref{lem:two-surplus-tree-probability} with $m = 20L + 1$, and we obtain that
    \begin{align*}
        \Pr[G \sim \G(n, d/n)]{G \in \LocalTreeOrUnicyclic_{5L}} \ge 1 - \frac{\e^2 d^3}{\e d - 1} \ab(1 + \frac{\ln n}{50 \ln{(2\e d)}})^2 n^{-49/50} \ge 1 - O_d\ab(n^{-49/50} (\ln n)^2)
    \end{align*}
    which is $1 - o_n(1)$ for every sufficiently large integer $n$.
\end{proof}

\subsection{Small occupancy of high-degree vertices in long paths}
The next property is that, for every long path in $\G(n, d/n)$, although there might be some vertices with high degree in the path, they account for a small proportion.

For $\eta \in (0, 1)$ and $d > 1$, pick the cut-off $D$ as
\begin{align*}
    D \defeq \floor{(1 + \eta/3)d}.
\end{align*}
The high-degree vertex set is defined as
\begin{align*}
    V^{\HD} \defeq \set{v \in V \mid \deg_G(v) > D}.
\end{align*}
Let $\MajorGood_{\ell, r}$ be the event that for every simple path in $\G(n, d/n)$ on $\ell$ vertices (of length $\ell - 1$, equivalently), there are at least $\ell/2$ vertices outside $B_G(V^\HD, r)$.
\begin{lemma} \label{lem:high-degree-small-occupancy}
    On the event $\LocalTreeOrUnicyclic_{5L}$, for sufficiently large $d > 1$ and $r = O_\eta((\ln d)^2)$, it holds that
    \begin{align*}
        \Pr[G \sim \G(n, d/n)]{G \in \MajorGood_{\floor{L/4}, r} \mid \LocalTreeOrUnicyclic_{5L}} \ge 1 - o_n(1).
    \end{align*}
\end{lemma}
% The proof of \Cref{lem:high-degree-small-occupancy} is in \Cref{subsec:high-degree-small-occupancy}.

To prove \Cref{lem:high-degree-small-occupancy}, we first need the following bound.
\begin{lemma} \label{lem:joint-degree-bound}
    Fix an arbitrary graph $J$ on a subset of $V = [n]$.
    For every subset $S \subseteq V$ and $z > 0$, it holds that
    \begin{align*}
        \Pr{S \subseteq V^\HD \mid J \subseteq G} \le \exp\ab(-zD\abs{S} + 2z\abs{E(J)} + \frac{\abs{S} d}{2}\ab(\e^{2z} - 1)).
    \end{align*}
\end{lemma}
\begin{proof}
    We write
    \[
        \sum_{v \in S} \deg_G(v) = N(S, J) + \sum_{e = (u, v) \notin E(J)} \abs{\set{u, v} \cap S} \cdot \id{e \in E(G)}
    \]
    where $N(S, J)$ is the total number of degrees contributing to $\sum_{v \in S} \deg_G(v)$ in $J$.
    It holds that $N(S, J) \le 2\abs{E(J)}$ and
    \[
        \sum_{e = (u, v) \notin E(J)} \abs{\set{u, v} \cap S} \le n\abs{S}.
    \]
    By the inequality $\e^{z x} - 1 \le \frac{x}{2}(\e^{2z} - 1)$ for every $x \in [0, 2]$ and $z > 0$,
    \[
    \begin{aligned}
        \E{\exp\ab(z \sum_{e = (u, v) \notin E(J)} \abs{\set{u, v} \cap S} \cdot \id{e \in E(G)})} &= \prod_{e = (u, v) \notin E(J)} \ab(1 + \frac{d}{n}\ab(\e^{z \abs{\set{u, v} \cap S}} - 1)) \\
        &\le \exp\ab(\frac{d}{n} \sum_{e = (u, v) \notin E(J)} \ab(\e^{z\abs{\set{u, v} \cap S}} - 1)) \\
        &\le \exp\ab(\frac{\abs{S} d}{2} \ab(\e^{2z} - 1)).
    \end{aligned}
    \]
    When $S \subseteq V^{\HD}$, it must hold that $\sum_{v \in S} \deg_G(v) \ge D\abs{S}$.
    Therefore by the Markov inequality,
    \[
    \begin{aligned}
        \Pr{S \subseteq V^\HD \mid J \subseteq G} &\le \Pr{\sum_{e = (u, v) \notin E(J)} \abs{\set{u, v} \cap S} \id{e \in E(G)} \ge D\abs{S} - 2\abs{E(J)} \mid J \subseteq G} \\
        &\le \exp\ab(-z D\abs{S} + 2z\abs{E(J)} + \frac{\abs{S}d}{2}\ab(\e^{2z} - 1)).
    \end{aligned}
    \]
    Then we conclude the lemma.
\end{proof}

\begin{proof}[Proof of \Cref{lem:high-degree-small-occupancy}]
    Set $L_0 = \floor{L/4}$.
    Suppose that there is a path $P = (u = v_0, \ldots, v_{L_0 - 1})$ containing fewer than $L_0/2$ vertices outside $B_G(V^{\HD}, r)$.
    Note that $B_G(P, r)$ is contained in $B_G(u, L_0 + r - 1)$.
    Therefore $B_G(u, L_0 + r - 1)$ is tree or unicyclic conditional on the event $\LocalTreeOrUnicyclic_{5L}$ for all sufficiently large $n$.
    For every $v \in P \cap B_G(V^{\HD}, r)$, pick a vertex $w = w(v) \in V^{\HD}$ minimizing the distance $\dist_G(v, w)$.
    Let
    \begin{align*}
        S \defeq \set{w = w(v) \;\colon\; v \in P \cap B_G(V^{\HD}, r)}.
    \end{align*}
    Using a fact that in a tree or unicyclic graph, there are at most $4r + 2$ vertices in a simple path at distance $\le r$ from a fixed vertex, it holds that
    \begin{align*}
        \frac{L_0}{2(4r + 2)} \le \abs{S} \le L_0.
    \end{align*}
    To apply \Cref{lem:joint-degree-bound}, the graph $J$ is constructed as the union of $P$ and all shortest paths from $v$ to $w(v)$ for $v \in P \cap B_G(V^{\HD}, r)$.
    It is clear that $J$ is connected and part of $B_G(u, L_0 + r - 1)$, meaning that $J$ is a tree or a unicyclic graph.
    Thus
    \begin{align*}
        \abs{V(J)} \le (r + 1) L_0, \quad \abs{E(J)} \le \abs{V(J)}, \quad \abs{E(J)} \le (r + 1)(8r + 4) \abs{S}.
    \end{align*}
    Apply \Cref{lem:joint-degree-bound} on $J$, $S$ and $z = \frac{1}{2} \ln\frac{D - 8(r + 1)(2r + 1)}{d}$.
    It holds that for sufficiently large $d > 1$, $D - 8(r + 1)(2r + 1) \ge (1 + \eta/4)d$ and thus
    \begin{align*}
        \Pr{S \subseteq V^{\HD} \mid J \subseteq G} &\le \exp\ab(-\frac{\abs{S} d}{2} \ab((1 + \eta/4) \ln (1 + \eta/4) - (1 + \eta/4) + 1)) \\
        &\le \exp\ab(-\frac{C_\eta d L_0}{8r + 4})
    \end{align*}
    where $C_\eta = \frac{1}{2}((1 + \eta/4) \ln(1 + \eta/4) - \eta/4) > 0$.

    To conclude the lemma, we count the number of all possible $J$.
    Assume that $J$ is on $m$ vertices.
    Since $J$ is tree or unicyclic, the expected number of all possible $J$ is at most
    \begin{align*}
        \binom{n}{m} m^{m - 2} \frac{d^{m - 1}}{n^{m - 1}} + \binom{n}{m} m^{m - 2} \frac{m^2}{2} \frac{d^m}{n^m} \le \ab(\frac{n}{d m^2} + \frac{1}{2}) (\e d)^m.
    \end{align*}
    For a fixed $J$ on $m$ vertices, there are at most $2^{m}$ possible choices of $S$.
    Therefore, the failure probability is upper bounded by
    \begin{align*}
        \Pr[G \sim \G(n, d/n)]{G \in \overline{\MajorGood_{L_0, r}} \cap \LocalTreeOrUnicyclic_{5L}} &\le \sum_{m = 1}^{(r + 1)L_0} \ab(\frac{n}{d m^2} + \frac{1}{2}) (\e d)^m \cdot 2^m \cdot \exp\ab(-\frac{C_\eta d L_0}{8r + 4}) \\
        &\le n\ab(\frac{1}{d} + \frac{1}{2}) \frac{2\e d}{2\e d - 1} (2\e d)^{(r + 1)L_0} \exp\ab(-\frac{C_\eta d L_0}{8r + 4}) \\
        &\le n\ab(\frac{1}{d} + \frac{1}{2}) \frac{2\e d}{2\e d - 1} \exp\ab(-\ab(\frac{C_\eta d}{8r + 4} - (r + 1)\ln(2\e d))L_0).
    \end{align*}
    Increasing $d$ such that $\frac{C_\eta d}{8r + 4} - (r + 1)\ln(2\e d) \ge 20000 \ln(2\e d)$, this joint probability is at most $n^{-2 + o(1)}$.
    By \Cref{lem:locally-tree-or-unicyclic}, $\Pr{\LocalTreeOrUnicyclic_{5L}}=1-o_n(1)$, so
    \[
        \Pr{\overline{\MajorGood_{L_0,r}}\mid\LocalTreeOrUnicyclic_{5L}}
        =\frac{\Pr{\overline{\MajorGood_{L_0,r}}\cap\LocalTreeOrUnicyclic_{5L}}}{\Pr{\LocalTreeOrUnicyclic_{5L}}}
        =o_n(1),
    \]
    which concludes the lemma.
\end{proof}

%% file: section/coupling-independence.tex
\section{Coupling Independence for Random Graph Colorings} \label{sec:coupling-independence}

In this section, we present the proof of \Cref{thm:coupling-independence}.
First, we introduce the result of strong spatial mixing and total influence decay for colorings on trees in \cite{CLMM26}.
Then we show how to analyze the coupling discrepancy on random graphs using the introduced result.
Roughly speaking, a message passing process on local trees with good/bad marks would give an upper bound to the coupling independence under any arbitrary pinning.
Lastly, we estimate relevant messages, which finally concludes \Cref{thm:coupling-independence}.

\subsection{Strong spatial mixing and total influence decay on trees}
Before we analyze the coupling independence on $G \sim \mathcal G(n, d/n)$, let us recall the strong spatial mixing and total influence decay for colorings on trees derived by \cite{CLMM26}.
We consider the following form of strong spatial mixing and total influence decay.
\begin{definition}[Strong spatial mixing \cite{CLMM26}] \label{def:strong-spatial-mixing}
    The Gibbs distribution $\mu$ of a feasible list-coloring instance $\Phi = \left(G = (V, E), L = \set{L(v)}_{v \in V}\right)$ exhibits \emph{strong spatial mixing} with exponential decay rate $1 - \delta$ and constant $C_{\SSM} > 0$ if the following holds.
    For every integer $\ell \ge 2$, every feasible partial pinning $\pi$ on $\Lambda \subseteq V$, every vertex $v \in V \setminus \Lambda$, every color $c$ feasible at $v$ under $\pi$, and every two partial colorings $\sigma, \tau$ on $S_G(v, \ell) \setminus \Lambda$ feasible together with $\pi$, both conditional probabilities below are positive and
    \begin{align*}
        \abs{\frac{\mu(v \gets c \;|\; \pi, \sigma)}{\mu(v \gets c \;|\; \pi, \tau)} - 1} \le C_{\SSM} (1 - \delta)^\ell.
    \end{align*}
\end{definition}

\begin{definition}[Total influence decay \cite{CLMM26}] \label{def:total-influence-decay}
    The Gibbs distribution $\mu$ of a feasible list-coloring instance $\Phi = \left(G = (V, E), L = \set{L(v)}_{v \in V}\right)$ exhibits \emph{total influence decay} with exponential decay rate $1 - \delta$ and constant $C_{\TID} > 0$ if the following holds.
    For every integer $\ell \ge 1$, every feasible partial pinning $\pi$ on $\Lambda \subseteq V$, every vertex $v \in V \setminus \Lambda$, every two colors $a, b$ feasible at $v$ under $\pi$,
    \begin{align*}
        \sum_{u \in S_G(v, \ell) \setminus \Lambda} \TV{\mu(u \gets \cdot \;|\; \pi, v \gets a)}{\mu(u \gets \cdot \;|\; \pi, v \gets b)} \le C_{\TID}(1 - \delta)^\ell.
    \end{align*}
\end{definition}

The following proposition describes the strong spatial mixing and total influence decay for proper colorings on trees.
\begin{proposition}[{\cite[Proposition 8.8]{CLMM26}}] \label{prop:SSM-TID-on-tree}
    If $q = (1 + \eta) \Delta$ for $\eta \in (0, 1)$ and $\Delta \ge 7/\eta^2$, for all trees of maximum degree $\Delta$, the uniform distribution of proper $q$-colorings exhibits:
    \begin{itemize}
        \item strong spatial mixing with exponential decay rate $\e^{-\eta/4}$ and constant
        \[
            C_{\SSM} = \frac{2q}{(\eta \Delta)^{1/2}} \e^{1/\eta + \eta/4} = O_{\eta}\ab(\poly{\Delta}).
        \]
        \item total influence decay with exponential decay rate $\e^{-\eta/4}$ and constant
        \[
            C_{\TID} = \frac{q^3}{(\eta\Delta)^{3/2}(q - 1)} \e^{3/\eta} = O_{\eta}\ab(\poly{\Delta}).
        \]
    \end{itemize}
\end{proposition}
The same rates of strong spatial mixing and total influence decay hold for every feasible list-coloring instance on a tree of maximum degree at most $\Delta$ satisfying
\[
    \abs{L(v)} - \deg_T(v) \ge q - \Delta \qquad \text{for every free vertex }v \in V(T).
\]
This follows from the pinned-leaf reduction recorded in \Cref{lem:CLMM-list-coloring-extension}.

%\begin{lemma}\label{lem:list-tree-SSM-TID}
%    Fix \(\eta\in(0,1)\). For all sufficiently large \(d\), the following holds. Let \(\Phi\in\mathcal L(G,q)\), and suppose    that \(B_{G_\Phi}(v,r)\) is a tree containing no vertex of \(V^{\HD}\). Under every feasible pinning \(\pi\), the induced list-coloring instance on this ball satisfies SSM and TID with decay rate at most \(\exp(-\eta/16)\) and constant \(O_{\eta,q}(\poly{d})\).
%\end{lemma}

%\begin{proof}
%    For every free vertex \(x\) in the ball, pinning preserves the list-degree gap, and therefore
%    \[
%        |L_\Phi^\pi(x)|-\deg_{\Phi^\pi}(x)
%        \ge q-\deg_G(x)
%        \ge q-D
%        \ge \frac{2\eta}{3}d.
%    \]
%    Moreover, \(|L_\Phi^\pi(x)|\le q\) and \(\deg_{\Phi^\pi}(x)\le D\). 
%    Applying the list-coloring Jacobian estimate underlying~\cite[Theorems 8.3 and 8.5]{CLMM26}, with
%    \[
%        \gamma_x
%        \defeq |L_\Phi^\pi(x)|-\deg_{\Phi^\pi}(x)
%        \ge q-D,
%    \]
%    gives contraction at most
%    \[
%        \exp\!\left(
%            -\frac{2(q-D)}{q}
%            +\frac{3}{q-D-1}
%        \right)
%        \le \exp(-\eta/16)
%    \]
%    for sufficiently large \(d\). The corresponding SSM and TID constants are \(O_{\eta,q}(\poly{d})\), uniformly over all feasible pinnings.
%\end{proof}

An application of \Cref{prop:SSM-TID-on-tree} helps us analyze the coupling discrepancy on local parts of a global graph.
\begin{corollary} \label{cor:local-TID-on-tree}
    For every fixed parameter $\eta \in (0, 1)$, there exist $\kappa = \kappa(\eta) > 0$ and $\rho = \rho(\eta) \in (0, 1)$ such that the following holds for all sufficiently large $d > 1$. There exist integers
    \begin{align*}
        R = O_\eta(\ln d), \quad r = O_\eta\ab((\ln d)^2) > R + 2
    \end{align*}
    satisfying
    \begin{align} \label{eq:local-decay-parameter}
        4R \cdot \e^{-\kappa \cdot R} \cdot \rho^{-R} H_{D^R} \le 1\,.
    \end{align}
    Furthermore, for every integer $q \ge (1 + \eta)d$, the following conclusion holds.
    Let $\Phi \in \mathcal L(G, q)$ and suppose that the ball $B_G(v, r)$ is a tree and contains no vertex in $V^{\HD}$.
    For every feasible partial pinning $\pi$ on $\Lambda \subseteq V_\Phi \setminus \set{v}$ in $\Phi$ and two colors $a, b$ feasible at $v$ under $\pi$, it holds that
    \begin{align*}
        \sum_{\text{free vertex}~u \in S_{G_\Phi}(v, R)} \TV{\mu_\Phi(u \gets \cdot \;|\; \pi, v \gets a)}{\mu_\Phi(u \gets \cdot \;|\; \pi, v \gets b)} \le \e^{-\kappa \cdot R}.
    \end{align*}
\end{corollary}
\begin{proof}
    We first prove a weaker version restricted to trees by verifying the parameter choices in \Cref{prop:SSM-TID-on-tree}.
    Then we show how to remove the dependence on $q$ in the weak results.
    Finally, we embed the refined SSM and TID results on small trees into the global graph.

    Define
    \begin{align*}
        \Delta \defeq \max\;\set{D, \ceil{2q/3}}, \quad \delta = q/\Delta - 1.
    \end{align*}
    Since $q \ge (1 + \eta)d$, for sufficiently large $d > 1$, it holds that
    \begin{align*}
        \frac{\eta}{4} \le \min\;\set{\frac{2\eta}{3 + \eta}, \frac{q - 3}{2q + 3}} \le \delta \le \frac{1}{2}, \quad \Delta \ge 7/\delta^2.
    \end{align*}
%    By \Cref{prop:SSM-TID-on-tree} with choice of $\Delta$ and $q = (1 + \delta)\Delta$ at the local tree $T = G_\Phi[B_{G_\Phi}(v, r)]$, it holds that the SSM and TID hold with exponential decay $\e^{-\delta/4} \le \e^{-\eta/16}$ and constants
    We focus on the local tree at $v$ of radius $r$:
    \[
        T \defeq G_{\Phi}[B_{G_{\Phi}}(v, r)].
    \]
    After imposing the pinning $\pi$ and restricting the instance to $T$, let $T^\pi$ denote the forest induced by the free vertices and let $L^\pi(x)$ be the resulting effective list of each free vertex $x$.
    Observe that pinning and restriction to the local tree cannot decrease the gap between color lists and degree, \IE, for every free vertex $x$,
    \begin{align*}
        \abs{L^\pi(x)} - \deg_{T^\pi}(x) &\ge \abs{L_\Phi(x)} - \deg_{G_\Phi}(x) \\
        &\ge q - \deg_G(x) \\
        &\ge q - D \ge q - \Delta.
    \end{align*}
    Here the third inequality follows because the $r$-radius ball contains no vertex in $V^{\HD}$, while the last inequality follows from $\Delta \ge D$.
    
    Since $T^\pi$ has maximum degree at most $\Delta$ and $q = (1 + \delta)\Delta$, every component of the induced list-coloring instance satisfies the hypotheses of \Cref{lem:CLMM-list-coloring-extension}. Hence the SSM and TID bounds hold with exponential decay $\e^{-\delta/4} \le \e^{-\eta/16}$ and constants
    \begin{align*}
        C_{\SSM} = \frac{2q}{(\delta \Delta)^{1/2}} \e^{1/\delta + \delta/4}, \quad C_{\TID} = \frac{q^3 \e^{3/\delta}}{(\delta \Delta)^{3/2}(q - 1)}.
    \end{align*}
    A further upper bound for $C_{\SSM}$ and $C_{\TID}$ is $B_\eta \sqrt{q}$ where
    \begin{align*}
        B_\eta \defeq 4(6/\eta)^{3/2} \e^{12/\eta}.
    \end{align*}
    Therefore, for $3 \le k \le r$, every two additional feasible boundary pinnings $\sigma, \tau$ on $S_T(v,k)\setminus\Lambda$, and every color $c$ feasible at $v$ under $\pi$,
    \begin{align*}
        \abs{\frac{\mu(v \gets c \;|\; \pi, \sigma)}{\mu(v \gets c \;|\; \pi, \tau)} - 1} \le B_\eta \sqrt{q} \cdot \e^{-k \cdot \eta/16}
    \end{align*}
    and for two feasible colors $a, b$ at $v$ and $2 \le k \le r$,
    \begin{align*}
        \sum_{\text{free vertex}~u \in S_{G_\Phi}(v, k)} \TV{\mu(u \gets \cdot \;|\; \pi, v \gets a)}{\mu(u \gets \cdot \;|\; \pi, v \gets b)} \le B_\eta \sqrt{q} \cdot \e^{-k \cdot \eta/16},
    \end{align*}
    where $\mu$ is the Gibbs distribution of proper list-colorings on $T$ with all common pins and color lists.

    To remove the dependence on $q$, we delete the vertex $v$ and for every free child $w$, let $T_w$ be the component of $T - v$ containing $w$.
    Let $p_w$ be the marginal at $w$ in $T_w$ under the common pinning $\pi$.
    It holds that
    \begin{align*}
        \forall~\text{feasible color}~c, \quad p_w(c) \le \frac{1}{q - D + 1}
    \end{align*}
    since $w$ has at most $D - 1$ neighbors in $T_w$.
    Conditional on the assignment $a$ or $b$ at $v$, we only need to consider the maximal coupling of $p_w(\cdot\;|\; X_w \neq a)$ and $p_w(\cdot\;|\; X_w \neq b)$.
    The total variation distance is
    \begin{align*}
        \frac{\max\;\set{p_w(a), p_w(b)}}{1 - \min\;\set{p_w(a), p_w(b)}} \le \frac{1/(q - D + 1)}{1 - 1/(q - D + 1)} = \frac{1}{q - D}.
    \end{align*}
    Then we couple free vertices in $S_{G_\Phi}(v, k)$ conditional on the assignment $a$ or $b$ at $v$ as follows.
    For every free child $w$ of $v$, we maximally couple assignments at $w$ as $p_w(\cdot \;|\; X_w \neq a)$ and $p_w(\cdot \;|\; X_w \neq b)$.
    When they agree, we use the identical coupling in the remaining vertices in $T_w$.
    Otherwise, we apply \Cref{prop:SSM-TID-on-tree} at $T_w$.
    Therefore, the sum of marginal total variation distances over the free vertices of $S_{T_w}(w,k-1)$ is at most $\frac{1}{q - D} \cdot B_\eta \sqrt{q} \cdot \e^{-(k - 1)\eta/16}$.
    Summing over all children of $v$, the refined bound of total influence decay can be bounded by
    \begin{align*}
        \sum_{\text{free vertex } u \in S_{G_\Phi}(v, k)} \TV{\mu(u \gets \cdot \midl \pi, v \gets a)}{\mu(u \gets \cdot \midl \pi, v \gets b)} \le \frac{D}{q - D} B_\eta \sqrt{q} \cdot \e^{-(k - 1)\eta/16}.
    \end{align*}
    Since \(q\ge(1+\eta)d\), \(D\le(1+\eta/3)d\), and \(x\mapsto \sqrt{x}/(x-D)\) is decreasing for \(x>D\), we obtain that
    \begin{align*}
        \frac{D\sqrt q}{q-D} \le \frac{3\sqrt D}{\eta}.
    \end{align*}
    Therefore,
    \begin{align*}
        \sum_{\text{free }u\in S_{G_\Phi}(v,k)} \TV{\mu(u\gets\cdot\mid\pi,v\gets a)}   {\mu(u\gets\cdot\mid\pi,v\gets b)} \le \frac{3B_\eta\e^{\eta/16}}{\eta}\sqrt D\, \e^{-k\eta/16}.
    \end{align*}

    To refine the SSM, fix $3\le k\le r$ and two additional feasible pinnings $\sigma, \tau$ on $S_T(v,k)\setminus\Lambda$.
    For every free child $w$, let $p_w^\sigma$ and $p_w^\tau$ be the marginals at $w$ in $T_w$ under the common pinning $\pi$ and the additional pinnings $\sigma, \tau$ respectively.
    Since $q - D \ge 3$, it is clear to see $p_w^\sigma$ and $p_w^\tau$ share the same support.
    Apply \Cref{prop:SSM-TID-on-tree} at $T_w$ and we obtain that for all feasible colors $c$ at $w$,
    \begin{align*}
        \abs{p_w^\sigma(c) - p_w^\tau(c)} \le \frac{1}{q - D + 1} \cdot B_{\eta} \sqrt{q} \cdot \e^{-(k - 1) \cdot \eta/16}.
    \end{align*}
    Set
    \begin{align*}
        W_c^\sigma = \prod_{w~\text{free child of}~v} (1 - p_w^\sigma(c)), \quad
        W_c^\tau = \prod_{w~\text{free child of}~v} (1 - p_w^\tau(c)).
    \end{align*}
    By the mean value theorem,
    \begin{align*}
        \abs{\ln\frac{W_c^\sigma}{W_c^\tau}} \le \sum_{w~\text{free child of}~v} \frac{\abs{p_w^\sigma(c) - p_w^\tau(c)}}{1 - 1/(q - D + 1)} \le \frac{D}{q - D} \cdot B_\eta \sqrt{q} \cdot \e^{-(k - 1)\cdot \eta/16}.
    \end{align*}
    After normalization, we can obtain that
    \begin{align*}
        \abs{\ln\frac{\mu(v \gets c \;|\; \pi, \sigma)}{\mu(v \gets c \;|\; \pi, \tau)}} \le \frac{2D}{q - D} \cdot B_\eta \sqrt{q} \cdot \e^{-(k - 1)\cdot \eta/16}.
    \end{align*}
    Note that $(1 - 1/(q - D + 1))^D \le W_c^\sigma, W_c^\tau \le 1$.
    Then it holds that
    \begin{align*}
        \abs{\ln \frac{\mu(v \gets c \;|\; \pi, \sigma)}{\mu(v \gets c \;|\; \pi, \tau)}} \le -2D \ln\ab(1 - \frac{1}{q - D + 1}) \le \frac{2D}{q - D} \le \frac{6}{\eta}.
    \end{align*}
    Therefore, by the inequality $\abs{\e^x-1} \le \e^{\abs{x}}\abs{x}$, we have
    \begin{align*}
        \abs{\frac{\mu(v \gets c \;|\; \pi, \sigma)}{\mu(v \gets c \;|\; \pi, \tau)} - 1} \le \frac{6\sqrt{D} \e^{\eta/16 + 6/\eta}}{\eta} B_\eta \cdot \e^{-k \cdot \eta/16}. 
    \end{align*}
    Hence, we have the refined SSM and TID exponential decay rate $\e^{-\eta/16}$ and constant at most $C \defeq 512 \eta^{-3} \e^{20/\eta} \cdot D^2$.

    Now we show how to embed the SSM and TID results on trees into the global graph.
    Let $S_r \defeq S_{G_\Phi}(v, r) \setminus \Lambda$ and $S_R \defeq S_{G_\Phi}(v, R) \setminus \Lambda$ be the free spheres at radii $r$ and $R$ respectively.
    Let $\nu$ be the law of $X_{S_r}$ under the common pinning $\pi$.
    Then it holds that for every two feasible boundaries $\sigma, \tau$ on $S_r$,
    \begin{align*}
        \abs{\mu_\Phi(v \gets c \;|\; \pi, \sigma) - \mu_\Phi(v \gets c \;|\; \pi, \tau)} \le \mu_\Phi(v \gets c \;|\; \pi, \tau) \cdot C \e^{-r \cdot \eta/16}.
    \end{align*}
    Averaging over $\tau$ with the law $\nu$, it holds that
    \begin{align*}
        \left|\mu_\Phi(v \gets c \;|\; \pi, \sigma) - \sum_{\tau} \nu(\tau) \mu_\Phi(v \gets c \;|\; \pi, \tau)\right| \le C \e^{-r \cdot \eta/16} \sum_{\tau} \nu(\tau) \mu_\Phi(v \gets c \;|\; \pi, \tau). 
    \end{align*}
    Then by Bayes' rule, it holds that
    \begin{align*}
        \TV{\nu(\cdot \;|\; v \gets a)}{\nu} \le \frac{C}{2} \e^{-r \cdot \eta/16}
    \end{align*}
    meaning that $\TV{\nu(\cdot \;|\; v \gets a)}{\nu(\cdot \;|\; v \gets b)} \le C\e^{-r \cdot \eta/16}$.

    Therefore, we maximally couple $\sigma \sim \nu(\cdot \;|\; v \gets a)$ and $\tau \sim \nu(\cdot \;|\; v \gets b)$ and apply the refined TID results when $\sigma = \tau$.
    By the crude bound that the disagreements are at most $\abs{S_R}$ when $\sigma \neq \tau$, it holds that
    \begin{align*}
        \sum_{\text{free vertex}~u \in S_{G_\Phi}(v, R)} \TV{\mu_\Phi(u \gets \cdot \;|\; \pi, v \gets a)}{\mu_\Phi(u \gets \cdot \;|\; \pi, v \gets b)} \le \abs{S_R} C \e^{-r \cdot \eta/16} + C \e^{-R \cdot \eta/16}.
    \end{align*}

    We pick $\kappa = \eta/32$, $R$ as the smallest integer at least
    \begin{align*}
        \max\; \set{4, \frac{32}{\eta} \ln{(2C)}, \frac{8}{\kappa} \ln D}
    \end{align*}
    and $r$ as the smallest integer at least
    \begin{align*}
        \max\; \set{R + 3, \frac{16 R \ln D + 16 \kappa R + 16 \ln(2C)}{\eta}}.
    \end{align*}
    It is clear to see $R = O_\eta(\ln d)$ and $r = O_\eta((\ln d)^2)$ and the desired total influence decay holds.
    Moreover, pick $\rho = \e^{-\kappa/2}$.
    For sufficiently large $d > 1$, it can be easily verified that $4R \cdot \e^{-\kappa \cdot R} \cdot \rho^{-R} H_{D^R} \le 1$.
    Thus we conclude the corollary.
\end{proof}

We pick parameters $\kappa, \rho, R$ and $r$ with respect to $\eta$ and $d$ in the remaining part of \Cref{sec:coupling-independence} as in \Cref{cor:local-TID-on-tree}.

The second tool we introduce from \cite{CLMM26} is its \emph{one-disagreement coupling}.
The description below assumes that the intermediate pinnings are feasible; we verify this explicitly in our application in \Cref{lem:Wasserstein-distance-recursive-bound}.
We recall it here: to couple the two Gibbs distributions $\mu(\cdot \;|\; v \gets a)$ and $\mu(\cdot \;|\; v \gets b)$, we pick a free vertex $u \in S_G(v, \ell)$ and consider the maximal coupling of the marginals $\mu(u \gets \cdot \;|\; v \gets a), \mu(u \gets \cdot \;|\; v \gets b)$ at $u$.
If $u$ is successfully coupled to a common color $c$, we proceed to couple $\mu(\cdot \;|\; v \gets a, u \gets c)$ and $\mu(\cdot \;|\; v \gets b, u \gets c)$.
Otherwise, when $u$ is assigned the disagreement $c$ and $d$ respectively, we couple $\mu(\cdot \;|\; v \gets a, u \gets c)$ and $\mu(\cdot \;|\; v \gets b, u \gets d)$ using the path coupling:
\begin{itemize}
    \item recursively couple $\mu(\cdot \;|\; v \gets a, u \gets c)$ and $\mu(\cdot \;|\; v \gets a, u \gets d)$;
    \item recursively couple $\mu(\cdot \;|\; v \gets a, u \gets d)$ and $\mu(\cdot \;|\; v \gets b, u \gets d)$;
    \item concatenate the two couplings by path coupling.
\end{itemize}
Let $\W(\Phi, v)$ be the maximum Wasserstein-$1$ distance between the two Gibbs distributions obtained by pinning $v$ to two different feasible colors in $\Phi$, with Hamming distance measured on all of $V_\Phi$, including $v$.
We set $\W(\Phi, v) = 0$ if $v$ has fewer than two feasible colors.
Moreover, define
\begin{align*}
    \W(G, q) \defeq \max_{\Phi \in \mathcal L(G, q), v \in V_\Phi} \W(\Phi, v).
\end{align*}
The following recursive upper bound for $\W(\Phi, v)$ comes directly from the coupling.
\begin{lemma} \label{lem:Wasserstein-distance-recursive-bound}
    Under the hypothesis of \Cref{cor:local-TID-on-tree}, it holds that
    \begin{align*}
        \W(\Phi, v) \le (D^R + 1) + \e^{-\kappa R} H_{D^R} \max_{\Phi', u}\; \W(\Phi', u)
    \end{align*}
    where $\Phi'$ is the child instance of $\Phi$ obtained by pinning $v$ and possibly some vertices at the sphere $S_{G_\Phi}(v, R)$, and $u$ is an unpinned vertex in $S_{G_\Phi}(v, R)$.
    The maximum is zero if there are no such pairs $(\Phi',u)$.
\end{lemma}
\begin{proof}
    Write $S = S_{G_\Phi}(v, R)$ and $m = \abs{S} \le D^R$.
    The sphere $S$ separates $B_{G_\Phi}(v, R - 1)$ from the remaining vertices outside $S$.
    Moreover, every vertex $x$ in this ball satisfies
    \[
        \abs{L_\Phi(x)} - \deg_{G_\Phi}(x) \ge q - D \ge 3.
    \]
    Consequently, any feasible assignment outside $B_{G_\Phi}(v, R - 1)$ can be combined with any color $a \in L_\Phi(v)$ at $v$.
    Indeed, keep that outside assignment, set $v \gets a$, and greedily recolor the other vertices of the tree $B_{G_\Phi}(v, R - 1)$ from the root outwards.
    The root has no neighbor outside the ball because $R \ge 4$, and the list-degree slack ensures an available color at every subsequent step, even after excluding all colors on already fixed neighbors.
    Fix $\sigma$ as a feasible partial assignment on $P \subseteq S$ and a feasible prescription $v \gets c$.
    Denote by $\Phi^{\sigma, c}$ the list-coloring instance from $\Phi$ under the pinning $\sigma \land \set{v \gets c}$ and let $U^{\sigma, c} \defeq V_\Phi \setminus (P \cup \set{v})$ be the free vertices of $\Phi^{\sigma, c}$.
    At first it is clear to see $\Phi^{\sigma, c} \in \mathcal L(G, q)$ for every $P \subseteq S$, every feasible partial assignment $\sigma$ on $P$ and every feasible color $c$ at $v$.

    Set the quantity
    \begin{align*}
        \W_M \defeq \max_{\substack{P \subseteq S,\; \text{feasible}\; \sigma\; \text{on}\; P, \\ \text{feasible color}\; v \gets c,\; u \in S \setminus P}} \W(\Phi^{\sigma, c}, u).
    \end{align*}
    As above, an empty maximum is zero, and we use $H_0=0$.
    It suffices to show $\W(\Phi, v) \le (D^R + 1) + \e^{-\kappa R} H_m \cdot \W_M$.
    For every $\ell = 0, \ldots, m$, define
    \begin{align*}
        \D_\ell \defeq \max_{\substack{P \subseteq S : \abs{S \setminus P} = \ell \\ \text{feasible}\; \sigma\; \text{on}\; P,\; v \gets a\; \text{and}\; v \gets b}} \W_1(\mu_\Phi(\cdot \;|\; \sigma, v \gets a), \mu_\Phi(\cdot \;|\; \sigma, v \gets b)).
    \end{align*}
    Then $\W(\Phi, v) = \D_m$.

    Now we show by induction that
    \begin{align*}
        \D_\ell \le \begin{cases}
            D^R + 1, & \ell = 0 \\
            \D_{\ell - 1} + \frac{\e^{-\kappa R}}{\ell} \W_M, & \ell = 1, \ldots, m
        \end{cases}
    \end{align*}
    and thus it concludes $\W(\Phi, v) = \D_m \le (D^R + 1) + \e^{-\kappa R} H_m \cdot \W_M$.
    
    \paragraph{\textbf{Base case $\ell = 0$.}}
    Conditional on the full assignment on $S$, the distributions inside and outside $B_{G_\Phi}(v, R - 1)$ are independent.
    We therefore couple the outside assignments identically, so disagreements can only occur in $B_{G_\Phi}(v, R - 1)$.
    Hence $\D_0 \le \sum_{j = 0}^{R - 1} D^j \le D^R + 1$.
    
    \paragraph{\textbf{The intermediate case $\ell = 1, \ldots, m$.}}
    Fix a sphere subset $P \subseteq S$ with $\abs{S \setminus P} = \ell$, a feasible pinning $\sigma$ on $P$ and an arbitrary free vertex $u \in S \setminus P$.
    If $(\sigma, v \gets b, u \gets d)$ is feasible, the extension argument above shows that $(\sigma, v \gets a, u \gets d)$ is feasible for any feasible color $a$ at $v$.
    Thus all intermediate conditional distributions in the one-disagreement coupling are well-defined.
    For all colors $a, b, c, d$ for which $(\sigma, v \gets a, u \gets c)$ and $(\sigma, v \gets b, u \gets d)$ are feasible, the coupling gives
    \begin{align*}
        \W_1(\mu_\Phi(\cdot \;|\; \sigma, v \gets a, u \gets c), \mu_\Phi(\cdot \;|\; \sigma, v \gets b, u \gets d)) &\le \D_{\ell - 1} + \id{c \neq d} \W(\Phi^{\sigma, a}, u) \\
        &\le \D_{\ell - 1} + \id{c \neq d} \W_M.
    \end{align*}
    Therefore for every free vertex $u \in S \setminus P$,
    \begin{align*}
        &\W_1(\mu_\Phi(\cdot \;|\; \sigma, v \gets a), \mu_\Phi(\cdot \;|\; \sigma, v \gets b)) \\
        &\qquad\le \D_{\ell - 1} + \TV{\mu_\Phi(u \gets \cdot \;|\; \sigma, v \gets a)}{\mu_\Phi(u \gets \cdot \;|\; \sigma, v \gets b)} \cdot \W_M.
    \end{align*}
    By \Cref{cor:local-TID-on-tree},
    \begin{align*}
        \sum_{u \in S \setminus P} \TV{\mu_\Phi(u \gets \cdot \;|\; \sigma, v \gets a)}{\mu_\Phi(u \gets \cdot \;|\; \sigma, v \gets b)} \le \e^{-\kappa R}.
    \end{align*}
    Pick the vertex $u$ to minimize $\TV{\mu_\Phi(u \gets \cdot \;|\; \sigma, v \gets a)}{\mu_\Phi(u \gets \cdot \;|\; \sigma, v \gets b)}$, and maximize over $P,\sigma,a,b$ to obtain
    \begin{align*}
        \D_\ell \le \D_{\ell - 1} + \frac{\e^{-\kappa R}}{\ell} \W_M.
    \end{align*}

    Combining all things together, we conclude that $\W(\Phi, v) \le (D^R + 1) + \e^{-\kappa R} H_m \cdot \W_M$, which leads to the lemma by monotonicity of harmonic numbers.
\end{proof}

\subsection{Message passing on marked trees}
In this part, we show a message passing process on a local tree with nodes marked as good or bad.
The messages would help us analyze the disagreements through the coupling on random graphs.

For a finite tree $T$ rooted at $\varrho$, each node $\omega$ in $T$ is labeled as good or bad and equipped with cost $\cost(\omega) \ge 0$.
Let $\child_T(\omega)$ be the child nodes of $\omega$ in $T$ and $\parent_T(\omega)$ be the parent node of $\omega$ in $T$ with convention $\parent_T(\varrho) = \varnothing$.
Let $\SubTree_T(\omega)$ be the subtree of $T$ rooted at $\omega$.
For $\alpha \in (0, 1)$ recursively define the message
\begin{align} \label{eq:message-passing-recursion}
    \phi(\alpha; T, \varrho) = \cost(\varrho) + \begin{cases}
        \alpha \cdot \max_{\omega \in \child_T(\varrho)} \phi(\alpha; \SubTree_T(\omega), \omega), & \varrho\; \text{is good} \\
        \sum_{\omega \in \child_T(\varrho)} \phi(\alpha; \SubTree_T(\omega), \omega), & \varrho\; \text{is bad}
    \end{cases}
\end{align}
where the summation and the maximum over an empty set are defined as zero.
For simplicity, for an unrooted tree $T$, we write $\phi(\alpha; T, \varrho)$ as the message after rooting $T$ at $\varrho$.
The following rerooting lemma records the change of messages.
\begin{lemma} \label{lem:rerooting}
    For a node $\omega \neq \varrho$ in $T$ at distance $\ell \ge 1$ from $\varrho$, it holds that
    \begin{align*}
        \phi(\alpha; T - \varrho, \omega) \le \ell \alpha^{-\ell} \phi(\alpha; T, \varrho)
    \end{align*}
    where $\phi(\alpha; T - \varrho, \omega)$ only needs to pass messages on the component of $T - \varrho$ containing $\omega$ after switching the root to $\omega$, with inherited marks and costs.
\end{lemma}
\begin{proof}
    Assume that the unique path from $\varrho$ to $\omega$ is
    \begin{align*}
        \varrho = \gamma_0, \gamma_1, \ldots, \gamma_\ell = \omega.
    \end{align*}
    Remove all edges of this path.
    Let $T_i$ be the remaining component containing $\gamma_i$ for $i = 0, \ldots, \ell$.

    We firstly observe that
    \begin{align*}
        \phi(\alpha; T_i, \gamma_i) \le \phi(\alpha; \SubTree_T(\gamma_i), \gamma_i).
    \end{align*}
    In fact, $T_i$ is a subtree of $\SubTree_T(\gamma_i)$ and the upper bound is direct.
    Furthermore,
    \begin{align*}
        \phi(\alpha; \SubTree_T(\gamma_i), \gamma_i) \le \alpha^{-1} \phi(\alpha; \SubTree_T(\gamma_{i - 1}), \gamma_{i - 1})
    \end{align*}
    by recursion of the message $\phi$.
    Therefore for every $i = 0, \ldots, \ell$,
    \begin{align*}
        \phi(\alpha; T_i, \gamma_i) \le \alpha^{-i} \phi(\alpha; \SubTree_T(\gamma_0), \gamma_0) = \alpha^{-i} \phi(\alpha; T, \varrho).
    \end{align*}

    Now for $i = 1, \ldots, \ell$, let $B_i$ be the tree consisting of $T_1, \ldots, T_i$, joined by edges in path $\gamma_1, \ldots, \gamma_i$.
    Root $B_i$ at $\gamma_i$.
    Then it holds that
    \begin{align*}
        \phi(\alpha; B_1, \gamma_1) = \phi(\alpha; T_1, \gamma_1), \quad \phi(\alpha; B_\ell, \gamma_\ell) = \phi(\alpha; T - \varrho, \omega).
    \end{align*}

    For $i = 2, \ldots, \ell$, if $\gamma_i$ is bad, then we directly have
    \begin{align*}
        \phi(\alpha; B_i, \gamma_i) = \phi(\alpha; B_{i - 1}, \gamma_{i - 1}) + \phi(\alpha; T_i, \gamma_i).
    \end{align*}
    Otherwise, since $\alpha < 1$, it directly holds that
    \begin{align*}
        \phi(\alpha; B_i, \gamma_i) \le \phi(\alpha; B_{i - 1}, \gamma_{i - 1}) + \phi(\alpha; T_i, \gamma_i).
    \end{align*}
    Summing over $i = 2, \ldots, \ell$, it holds that
    \begin{align*}
        \phi(\alpha; B_\ell, \gamma_\ell) &\le \sum_{i = 1}^\ell \phi(\alpha; T_i, \gamma_i) \\
        &\le \phi(\alpha; T, \varrho) \sum_{i = 1}^\ell \alpha^{-i} \\
        &\le \ell \alpha^{-\ell} \cdot \phi(\alpha; T, \varrho)
    \end{align*}
    which concludes this rerooting lemma.
\end{proof}

\subsection{Coupling independence by local messages}
The next step is to bound the coupling independence by local messages.

For a subgraph $J$ of $G$ such that it is tree or unicyclic, we construct a tree $\T = \T(J)$ as:
\begin{itemize}
    \item if $J$ is a tree itself, let $\T$ be a copy of $J$ and $\node(\cdot) : V(J) \to V(\T)$ be the mapping such that it maps every vertex $u$ in $J$ to its corresponding node in $\T$;
    \item otherwise $J$ is unicyclic.
    Let $C$ be the unique cycle in $J$.
    We contract all vertices in $C$ into one cycle node and retain other vertices as single nodes.
    For every $(u, v) \in E(J)$ with $\node(u) \neq \node(v)$, there is an edge between $\node(u)$ and $\node(v)$.
\end{itemize}
Furthermore, we mark nodes in $\T$ as bad or good in the following way.
\begin{itemize}
    \item At first, we mark $\omega \in \T$ as bad if it is a cycle node or its corresponding vertex is in $V^{\HD}$.
    \item For every currently unmarked node $\omega$, if there exists a node marked as bad in the first step within distance $r$ from $\omega$, mark $\omega$ as bad.
    \item Finally mark all remaining nodes as good.
\end{itemize}

\begin{lemma} \label{lem:CI-by-local-message}
    Let $J$ be an induced subgraph of $G$ which is tree or unicyclic.
    Furthermore, fix the boundary $\Sigma \subseteq V(J)$ as:
    \[
        \Sigma \defeq \set{u \in V(J) \;:\; \dist_G(u, V(G) \setminus V(J)) \le r + R + 3}.
    \]
    We consider $\T = \T(J)$ with marks on nodes.
    Define the cost on $\T$ as:
    \begin{itemize}
        \item $\cost(\omega) = 2(D^R + 1) \cdot \abs{C \setminus \Sigma} + \W(G, q) \cdot \abs{C \cap \Sigma}$ if $\omega$ corresponds to the unique cycle $C$;
        \item $\cost(\omega) = 2(D^R + 1)$ if $\omega$ corresponds to a vertex not in $\Sigma$;
        \item $\cost(\omega) = \W(G, q)$ if $\omega$ corresponds to a vertex in $\Sigma$.
    \end{itemize}
    Then it holds that
    \begin{align*}
        \W(\Phi, v) \le \phi(\alpha; \T, \node(v))
    \end{align*}
    for every $\alpha \in [\rho, 1)$, feasible list-coloring instance $\Phi \in \mathcal L(G, q)$ and $v \in V_\Phi \cap V(J)$.
\end{lemma}
\begin{proof}
    Fix a feasible list-coloring instance $\Phi \in \mathcal L(G, q)$.
    Let $\T[\Phi]$ be the forest generated from $\T$ by deleting all nodes with no vertices in $V_\Phi$.
    Let $\T[\Phi](v)$ be the tree in $\T[\Phi]$ containing the corresponding node of $v$.
    % Define the quantity
    % \[
    %     \psi(\rho; V_\Phi, v) \defeq \phi(\rho; \T[\Phi](v), \node(v)).
    % \]
    We take $d$ sufficiently large that $q-D\ge 2$. We prove by induction on $\abs{V_\Phi}$ that
    \begin{equation}\label{eq:strong-CI-by-local-message}
        \W(\Phi,v) \le    \phi\bigl(\rho;\T[\Phi](v),\node(v)\bigr).
    \end{equation}
    
    \paragraph{\textbf{The boundary case: $v\in\Sigma$.}}
    By the definition of the costs, including the case where $\node(v)$ is a cycle node,
    \[
        \W(\Phi, v) \le \W(G, q) \le \cost(\node(v)) \le \phi\ab(\rho; \T[\Phi](v), \node(v)).
    \]
    
    \paragraph{\textbf{The first interior case: $v\notin\Sigma$ and $\node(v)$ is bad.}}
    Let $B$ be the maximal connected component of bad nodes containing $\node(v)$ in $\T[\Phi]$, and define
    \[
        U \defeq \set{x \in V(J) \cap V_\Phi \midc \node(x) \in B}, \qquad W \defeq V_\Phi\setminus U.
    \]

    Suppose first that $U \cap \Sigma \neq \emptyset$, and choose $x \in U \cap \Sigma$.
    Every node on the path from $\node(v)$ to $\node(x)$ is bad.
    Consequently, the sum rule in the message recursion gives
    \[
        \phi\ab(\rho; \T[\Phi](v), \node(v)) \ge \cost(\node(x)) \ge \W(G, q) \ge \W(\Phi, v).
    \]
    Thus we may assume that $U \cap \Sigma = \emptyset$.
    Every vertex of $U$ has all its neighbors in $J$.
    Every component of $\T[\Phi] - B$ adjacent to $B$ has at most one edge of $G_\Phi$ to $U$.
    Indeed, $\T[\Phi]$ is a forest and a retained cycle node is bad; an edge incident to that cycle node may have its original cycle endpoint outside $V_\Phi$ and hence be absent from $G_\Phi$.
    Index just the components with such an edge as $\T_1, \ldots, \T_\ell$, and write the edges of $G_\Phi$ between $U$ and $W$ as
    \[
        (u_1, v_1), \ldots, (u_\ell, v_\ell), \qquad u_j \in U, \quad v_j \in W, \quad \node(v_j) \in V(\T_j).
    \]
    Moreover, the vertices $v_1, \ldots, v_\ell$ are distinct, their corresponding nodes are good, and $\deg_G(v_j) \le D$.
    Expanding the sum rule throughout $B$ and discarding the nonnegative messages from any other adjacent components, we obtain
    \begin{equation} \label{eq:bad-component-message-expansion}
        \phi\ab(\rho; \T[\Phi](v), \node(v)) \ge \sum_{\omega \in B} \cost(\omega) + \sum_{j = 1}^{\ell} \phi\ab(\rho; \T_j, \node(v_j)).
    \end{equation}
    Since $U \cap \Sigma = \emptyset$, the first sum is at least $\abs{U}$.
    Fix two feasible colors at $v$, and couple the corresponding marginals on $U$ arbitrarily.
    For a pair of resulting assignments $\sigma, \tau$ on $U$, we now couple the conditional distributions on the \emph{entire} remaining vertex set $W$.
    For $j = 0, \ldots, \ell$, let $\Phi^{(j)}$ be the list-coloring instance on $G[W]$ with lists
    \[
        L^{(j)}(x) = \begin{cases}
            L_\Phi(v_i) \setminus \set{\tau(u_i)}, & x = v_i,\ i \le j, \\
            L_\Phi(v_i) \setminus \set{\sigma(u_i)}, & x = v_i,\ i > j, \\
            L_\Phi(x), & x \notin \set{v_1, \ldots, v_\ell}.
        \end{cases}
    \]
    Write $\nu^{(j)} = \mu_{\Phi^{(j)}}$.
    Then $\nu^{(0)}$ and $\nu^{(\ell)}$ are precisely the conditional distributions on $W$ given $\sigma$ and $\tau$, respectively.
    
    We verify that these intermediate instances are feasible.
    The instance $\Phi^{(0)}$ is feasible by construction.
    At step $j$, only the list at $v_j$ changes.
    If a coloring of $\Phi^{(j - 1)}$ is not valid for $\Phi^{(j)}$, recolor $v_j$ to a color in
    \[
        L_\Phi(v_j) \setminus \ab(\set{\tau(u_j)} \cup \set{\xi(y) : y \in N_{G[W]}(v_j)}),
    \]
    where $\xi$ denotes the current coloring.
    Such a color exists because
    \[
        \abs{L_\Phi(v_j)} \ge \deg_{G_\Phi}(v_j) + q - \deg_G(v_j) = \deg_{G[W]}(v_j) + 1 + q - \deg_G(v_j) \ge \deg_{G[W]}(v_j) + 3.
    \]
    Hence every $\Phi^{(j)}$ is feasible.
    Furthermore, for every $x \in W$,
    \[
        \abs{L^{(j)}(x)} - \deg_{G[W]}(x) \ge \abs{L_\Phi(x)} - \deg_{G_\Phi}(x) \ge q - \deg_G(x),
    \]
    so $\Phi^{(j)} \in \L(G, q)$.
    
    For $j = 1, \ldots, \ell$, define $\widehat{\Phi}_j$ on $G[W]$ by using the common lists of $\Phi^{(j - 1)}$ and $\Phi^{(j)}$ at all vertices other than $v_j$, and restoring the list $L_\Phi(v_j)$ at $v_j$.
    This is a feasible instance in $\L(G, q)$, and
    \[
        \nu^{(j - 1)} = \mu_{\widehat{\Phi}_j} \ab(\cdot \mid v_j \neq \sigma(u_j)), \qquad \nu^{(j)} = \mu_{\widehat{\Phi}_j} \ab(\cdot \mid v_j \neq \tau(u_j)).
    \]
    Both distributions are mixtures of the same family
    \[
        \set{\mu_{\widehat{\Phi}_j}(\cdot \mid v_j = c) \midc c \text{ is feasible at } v_j}.
    \]
    Coupling the mixing colors and then optimally coupling the corresponding conditional distributions yields
    \[
        \W_1\ab(\nu^{(j - 1)}, \nu^{(j)}) \le \W(\widehat{\Phi}_j, v_j).
    \]
    Since $V_{\widehat{\Phi}_j} = W$ and $v \in U$, these instances have fewer free vertices than $\Phi$.
    Moreover,
    \[
        \T[\widehat{\Phi}_j] = \T[\Phi] - B, \qquad \T[\widehat{\Phi}_j](v_j) = \T_j.
    \]
    The induction hypothesis and the triangle inequality give
    \[
        \W_1\ab(\nu^{(0)}, \nu^{(\ell)}) \le \sum_{j = 1}^{\ell} \W(\widehat{\Phi}_j, v_j) \le \sum_{j = 1}^{\ell} \phi\ab(\rho; \T_j, \node(v_j)).
    \]
    Combining this conditional coupling on $W$ with the coupling on $U$, and then maximizing over the two feasible colors at $v$, we obtain
    \[
        \W(\Phi, v) \le \abs{U} + \sum_{j = 1}^{\ell} \phi\ab(\rho; \T_j, \node(v_j)) \le \phi\ab(\rho; \T[\Phi](v), \node(v))
    \]
    where the last inequality follows from \eqref{eq:bad-component-message-expansion}.
    
    \paragraph{\textbf{The second interior case: $v \notin \Sigma$ and $\node(v)$ is marked as good.}}
    By the marking process, $B_G(v, r) \subseteq J$ and $B_G(v, r) \cap V^{\HD} = \emptyset$.
    \Cref{lem:Wasserstein-distance-recursive-bound} provides that
    \[
        \W(\Phi, v) \le (D^R + 1) + \e^{-\kappa R} H_{D^R} \cdot \max_{\Phi', u}\; \W(\Phi', u).
    \]
    Here each child instance $\Phi'$ is defined on a subset $V_{\Phi'} \subseteq V_\Phi \setminus \set{v}$ and has a vertex $u \in V_{\Phi'} \cap B_{G_\Phi}(v, R)$.
    Furthermore, the $R$-length path from $v$ to $u$ completely lies in the intersection of $J$ and $G_\Phi$ and therefore there is a path from $\node(v)$ to $\node(u)$ of length at most $R$ in $\T[\Phi]$.
    By the induction hypothesis and \Cref{lem:rerooting},
    \begin{align*}
        \W(\Phi', u) &\le \phi(\rho; \T[\Phi'](u), \node(u)) \\
        &\le \phi(\rho; \T[\Phi](v) - \node(v), \node(u)) \\
        &\le \ell \rho^{-\ell} \phi(\rho; \T[\Phi](v), \node(v)) \\
        &\le R \rho^{-R} \phi(\rho; \T[\Phi](v), \node(v))
    \end{align*}
    where the last inequality holds for $\ell \le R$ and $\rho < 1$.
    Recall \Cref{cor:local-TID-on-tree}, and we obtain that
    \begin{align*}
        \W(\Phi, v) &\le (D^R + 1) + \e^{-\kappa R} H_{D^R} \cdot R \rho^{-R} \cdot \phi(\rho; \T[\Phi](v), \node(v)) \\
        &\le (D^R + 1) + \frac{1}{4} \phi(\rho; \T[\Phi](v), \node(v)) \\
        &\le \phi(\rho; \T[\Phi](v), \node(v))
    \end{align*}
    where the last inequality holds for
    \[
        (D^R + 1) \le \frac{1}{2} \cost(\node(v)) \le \frac{1}{2} \phi(\rho; \T[\Phi](v), \node(v)).
    \]
    
    Putting all things together, we conclude the lemma with monotonicity of $\phi$ on $\alpha \ge \rho$.
\end{proof}

\subsubsection{Further decomposition of messages}
To bound $\phi(\alpha; \T, \node(v))$ in \Cref{lem:CI-by-local-message}, it is convenient to introduce the following decomposition.

Let $Y(\alpha; \T, \node(v)), Z(\alpha; \T, \node(v))$ be the messages on $\T$ as \eqref{eq:message-passing-recursion} with retained marks and the costs defined as
\begin{align*}
    \cost_Y(\omega) &= \begin{cases}
        2(D^R + 1) \cdot \abs{C}, & \omega\; \text{corresponds to the unique cycle}\; C \\
        2(D^R + 1), & \omega\; \text{corresponds to a single vertex}\; v
    \end{cases}\;, \\
    \cost_Z(\omega) &= \begin{cases}
        \abs{C \cap \Sigma}, & \omega\; \text{corresponds to the unique cycle}\; C \\
        \id{v \in \Sigma}, & \omega\; \text{corresponds to a single vertex}\; v
    \end{cases}\;.
\end{align*}
\begin{proposition} \label{prop:decomposed-message-property}
    The following properties hold.
    \begin{itemize}
        \item $\phi(\alpha; \T, \node(v)) \le Y(\alpha; \T, \node(v)) + \W(G, q) Z(\alpha; \T, \node(v))$.
        \item $Z(\alpha; \T, \node(v)) \le \frac{Y(\alpha; \T, \node(v))}{2(D^R + 1)}$.
        \item Assume that every path from $\node(v)$ to a node of positive $\cost_Z$ contains at least $\ell$ good nodes other than its endpoint.
        Then for every $0 < \alpha_0 \le \alpha_1 < 1$,
        \begin{align*}
            Z(\alpha_0; \T, \node(v)) \le \ab(\frac{\alpha_0}{\alpha_1})^\ell Z(\alpha_1; \T, \node(v)).
        \end{align*}
    \end{itemize}
\end{proposition}
\begin{proof}
    Set $A\defeq2(D^R+1)$ and $W\defeq\W(G,q)$.
    The update in \eqref{eq:message-passing-recursion} is monotone,
    positively homogeneous, and subadditive, where subadditivity at a
    good node follows from
    \[
        \max_i(x_i+y_i)\le\max_i x_i+\max_i y_i.
    \]
    At every node, the three costs satisfy
    \[
        \cost\le\cost_Y + \,W\cost_Z,
        \qquad
        \cost_Z\le A^{-1}\cost_Y.
    \]
    For a cycle node these inequalities reduce to
    $\abs{C\cap\Sigma}\le\abs C$; for a single-vertex node they are
    immediate.  Induction on the height of the rooted subtree, using
    the three properties of the update, proves the first two claims.

    For the last claim, put $\theta\defeq\alpha_0/\alpha_1$.
    For a node $\omega$, let $g(\omega)$ be the minimum number of good
    nonterminal nodes on a path from $\omega$ to a node of positive
    $\cost_Z$ in its descendant subtree.  If no such node exists, its
    $Z$-message is zero.  We claim, again by induction on subtree height,
    that
    \begin{equation}\label{eq:decomposed-Z-alpha-comparison}
        Z(\alpha_0;\T_\omega,\omega)
        \le
        \theta^{g(\omega)}Z(\alpha_1;\T_\omega,\omega).
    \end{equation}
    If $\cost_Z(\omega)>0$, then $g(\omega)=0$, and the claim follows
    from monotonicity in $\alpha$.  Otherwise, writing $g_i$ for the
    corresponding values at the children and applying the induction
    hypothesis gives
    \begin{align*}
        Z(\alpha_0;\T_\omega,\omega)
        &\le \theta^{\min_i g_i}
        Z(\alpha_1;\T_\omega,\omega), && \omega\text{ bad},\\
        Z(\alpha_0;\T_\omega,\omega)
        &\le \theta^{1+\min_i g_i}
        Z(\alpha_1;\T_\omega,\omega), && \omega\text{ good}.
    \end{align*}
    These exponents are precisely $g(\omega)$, proving
    \eqref{eq:decomposed-Z-alpha-comparison}.  The assumption gives
    $g(\node(v))\ge\ell$, and the third claim follows since
    $0<\theta\le1$.
\end{proof}

\subsection{Estimation of local messages}
On the event $\LocalTreeOrUnicyclic_{5L}$, for every vertex $v$, we pick $J_v$ as:
\begin{align*}
    J_v = \begin{cases}
        B_G(v, 2L), & \text{if}\; B_G(v, 2L)\; \text{is a tree} \\
        B_G(v, 4L), & \text{if}\; B_G(v, 2L)\; \text{is unicyclic}
    \end{cases}\;.
\end{align*}
Let $\T_v$ be the tree induced by $J_v$ as in \Cref{lem:CI-by-local-message}.
Now we give the following upper bound to $Y(\alpha; \T_v, \node(v))$.
\begin{lemma}\label{lem:message-estimation}
    Fix $\alpha\in(0,1)$.
    For $d$ sufficiently large, depending on $\eta$ and
    $\alpha$, there exists a constant
    $C=C(\eta,d,\alpha)>0$ such that, with probability
    at least $1-o_n(1)$, the event
    $\LocalTreeOrUnicyclic_{5L}$ holds and
    \[
        Y(\alpha;\T_v,\node(v))\le C\ln n
        \qquad\text{for every }v\in V.
    \]
\end{lemma}

To prove \Cref{lem:message-estimation}, we consider the message passing on a Galton--Watson tree.
Let $(\GWT, \varrho)$ be a Galton--Watson tree rooted at $\varrho$ such that for every node $\omega \in \GWT$, it independently has $K_{\omega} \sim \Poisson(d + 1)$ children.
Equip $\GWT$ with the depth $\depth(\omega) \defeq \dist_{\GWT}(\varrho, \omega)$.
Let $\mathscr H$ be the collection of nodes of more than $D - 3$ children, \IE,
\[
    \mathscr H \defeq \set{\omega \in \GWT \mid K_\omega > D - 3}
\]
and choose a parameter $h$ as
\[
    h \defeq \left\lceil\max\;\set{3r, \frac{4\ln 4}{\ln(1/\alpha)}}\right\rceil.
\]
We mark nodes in $\GWT$ as following two ways:
\begin{align*}
    \mathscr{B} \defeq \set{\omega \in \GWT \mid \dist_{\GWT}(\omega, \mathscr H) \le r}, \quad
    \mathscr{B}^+ \defeq \mathscr{B} \cup \set{\omega \in \GWT \mid \depth(\omega) < 2h}.
\end{align*}
We define the message on $\GWT$ as the limit of messages on finite depth.
Formally, for $\ell \ge 0$ and a subset of nodes $S$ representing bad nodes, for every $\depth(\omega) \le \ell$,
\begin{align*}
    \phi_{\ell, S}(\alpha; \omega) = 2(D^R + 1) + \begin{cases}
        \alpha \cdot \max_{\gamma \in \child_{\GWT}(\omega)} \phi_{\ell, S}(\alpha; \gamma) & \omega \notin S \\
        \sum_{\gamma \in \child_{\GWT}(\omega)} \phi_{\ell, S}(\alpha; \gamma) & \omega \in S
    \end{cases}
\end{align*}
with boundary condition $\phi_{\ell, S}(\alpha; \gamma) = 0$ whenever $\depth(\gamma) > \ell$, and with the maximum and sum over an empty set defined as zero.
Define
\begin{align*}
    \X \defeq \lim_{\ell \to \infty} \phi_{\ell, \mathscr{B}}(\alpha; \varrho), \quad
    \Y \defeq \lim_{\ell \to \infty} \phi_{\ell, \mathscr{B}^+}(\alpha; \varrho).
\end{align*}
We first clarify here that $\X$ and $\Y$ are well-defined and finite almost surely for all sufficiently large $d > 1$.
\begin{proposition} \label{prop:GWT-message-well-definedness-finiteness}
    It holds that $\X$ and $\Y$ exist.
    Furthermore, for sufficiently large $d > 1$, there exists some $t_0 > 0$ such that
    \begin{align*}
        \E{\e^{t_0\X}} < \infty, \quad \E{\e^{t_0\Y}} < \infty.
    \end{align*}
\end{proposition}
The proof of \Cref{prop:GWT-message-well-definedness-finiteness} is provided in \Cref{sec:Poisson-estimation}.
Here we show how to upper bound local messages by \Cref{prop:GWT-message-well-definedness-finiteness}.
\begin{proof}[Proof of \Cref{lem:message-estimation}]
    Fix $\alpha \in (0, 1)$ and take $d$ sufficiently large for \Cref{prop:GWT-message-well-definedness-finiteness} to apply.
    Write
    \[
        a_0 \defeq 2(D^R + 1), \qquad \lambda \defeq d + 1, \qquad L_0 \defeq 4L + 2r + 3, \qquad \mathcal{E} \defeq \LocalTreeOrUnicyclic_{5L}.
    \]
    Since $r$ is fixed as $n \to \infty$, we have $L_0 \le 5L$ for all sufficiently large $n$.
    Consequently, on $\mathcal E$, every ball $B_G(v, L_0)$ is a tree or unicyclic.
    Therefore the prescribed choice of $J_v$ is well-defined.
    On this event write
    \[
        Y_v \defeq Y(\alpha; \T_v, \node(v))
    \]
    and for notational convenience, set $Y_v = 0$ otherwise.
    
    By \Cref{prop:GWT-message-well-definedness-finiteness}, there exists $t_0 > 0$ such that $\E{\e^{t_0\X}}$ and $\E{\e^{t_0\Y}}$ are finite.
    Choose $t \in (0, t_0]$ sufficiently small that
    \[
        M_X \defeq \E{\e^{t\X}} < \infty, \qquad M_Y \defeq \E{\e^{t\Y}} < \infty, \qquad ta_0 + \lambda(M_Y - 1) \le 1.
    \]
    Such a choice is possible because, by dominated convergence, $M_Y \to 1$ as $t \downarrow 0$.

    To compare local messages on the original graph with $\X$ and $\Y$, we first describe the following exploration coupling.
    Reveal edges adaptively without conditioning on the absence of future collisions.
    Conditional on the revealed history, the edges from the current vertex to undiscovered vertices that have not previously been queried are independent Bernoulli variables with parameter $d/n$.
    Thus their number has distribution $\Binomial(N,d/n)$ for some $N \le n$, and
    \[
        \Binomial(N, d/n) \preceq \Poisson\ab(-N\ln(1 - d/n)) \preceq \Poisson(\lambda)
    \]
    for all sufficiently large $n$.

    At successive exploration steps, use fresh independent randomness to realize this domination.
    Complete all unused offspring slots independently.
    The resulting auxiliary offspring variables are independent $\Poisson(\lambda)$ variables.
    If the exploration reveals a surplus edge not prescribed in advance, stop the comparison and complete the auxiliary trees independently.
    We use the comparison only on events on which no such surplus edge occurs.
    We mark that the exploration is carried out far enough to determine the relevant degrees and marks since the exploration is possible inside $B_G(v, L_0)$ because $J_v \subseteq B_G(v, 4L)$ and $4L + r + 1 < L_0$.

    \paragraph{\textbf{The first case: tree balls.}}
    Fix $v \in V$ and let
    \[
        \A_v^{T} \defeq \set{B_G(v, L_0) \text{ is a tree}}.
    \]
    Apply the preceding exploration coupling from $v$ with no prescribed surplus edge.
    On $\A_v^{T}$, the exploration gives a rooted embedding of $J_v$ into the Galton--Watson tree.
    Every relevant non-root vertex has one parent edge and all its other incident edges are counted among
    its children.
    Hence a vertex of degree greater than $D$ is mapped to a node whose auxiliary offspring count is greater than $D - 3$.
    The embedding preserves the tree paths, so every bad node of $\T_v$ is mapped to a node of $\scrB$.

    Adding branches and changing a good rule to a bad rule can only increase a message with non-negative costs.
    Therefore, under this coupling,
    \[
        Y_v \le \X \qquad \text{on } \A_v^{T}.
    \]
    The comparison here is event-wise.
    In particular,
    \[
        \E{\id{\A_v^{T} \e^{t Y_v}}} \le M_X.
    \]
    By Markov inequality, for every $x \ge 0$,
    \begin{equation} \label{eq:tree-message-joint-tail}
        \Pr{\A_v^{T} \cap \set{Y_v > x}} \le M_X \e^{-tx}.
    \end{equation}

    \paragraph{\textbf{The second case: unicyclic balls}}
    A rooted core $(v, H)$ consists of a simple cycle together with a simple path from $v$ to that cycle,
    meeting the cycle only at the endpoint of the path.
    The path is allowed to have length zero.
    In particular, $\abs{V(H)} = \abs{E(H)}$.
    For a fixed deterministic rooted core $(v, H)$, define the event $\A_{v, H}$ as the joint event of
    \begin{align*}
        \set{B_G(v, L_0) \text{ is unicyclic}} \land \set{H \text{ is its unique cycle together with the path from } v \text{ to that cycle}}.
    \end{align*}
    Obviously, $\A_{v, H} \subseteq \set{H \subseteq G}$.

    Let $k = \abs{V(H)}$.
    Conditional on the event $H\subseteq G$, all edges outside $E(H)$ remain independent Bernoulli variables with parameter $d/n$.
    Explore outwards from the prescribed core.
    Give each core vertex its distance from $v$ within $H$ as its initial distance label, and explore in increasing order of these labels and the subsequently discovered
    distances, up to distance $L_0$.
    Stop the comparison if an additional surplus edge is discovered.
    On $\A_{v, H}$, no such additional surplus edge occurs in the explored ball, and the discovered branches are trees attached to $H$.
    Force the vertices of $H \cap J_v$ to be bad and contract this connected root set.
    If the cycle is present in $J_v$, use its already contracted cycle node.
    Increase the cost of the resulting root to $k a_0$.
    These operations cannot decrease the message: replacing a good rule by the sum rule increases it, and contracting connected bad nodes adds their costs and still retains their outgoing branches.

    The total number of branches leaving the core is dominated by
    \[
        K_k \sim \Poisson(k\lambda),
    \]
    obtained by summing $k$ independent $\Poisson(\lambda)$ variables.
    All descendant auxiliary trees can be chosen independently of one another and of $K_k$.
    Within each branch, high-degree marks are dominated by the high-offspring marks.
    Any additional bad mark caused by the contracted core or by a high-degree vertex outside that branch can affect only vertices within distance $r$ of the core.
    These vertices are covered by the extra bad levels in the definition of $\Y$, since $h \ge 3r$ and $\scrB^+$ contains every node of depth less than $2h$.

    Consequently, conditional on $H \subseteq G$, the exploration admits a coupling with
    \[
        Z_k \defeq k a_0 + \sum_{j = 1}^{K_k} \Y_j
    \]
    where the $\Y_j$ are independent copies of $\Y$ independent of $K_k$, such that
    \[
        Y_v \le Z_k \qquad \text{on } \A_{v, H}.
    \]
    The exponential moment of $Z_k$ satisfies
    \[
        \E{\e^{t Z_k}} = \exp\ab(k t a_0 + k \lambda(M_Y - 1)) \le \e^k.
    \]
    Hence
    \[
        \Pr{\A_{v, H} \cap \set{Y_v > x} \mid H \subseteq G} \le \e^{k - tx}.
    \]
    Since $\A_{v, H} \subseteq \set{H \subseteq G}$ and $H$ has exactly $k$ edges, we obtain the unconditional joint-event bound
    \begin{equation} \label{eq:unicyclic-message-joint-tail}
        \Pr{\A_{v, H} \cap \set{Y_v > x}} \le \ab(\frac{d}{n})^k \e^{k - tx}.
    \end{equation}

    Now we count the number of possible cores.
    If $B_G(v, L_0)$ is unicyclic, its unique cycle has length at most $2L_0 + 1$, and the path from $v$ to
    that cycle has length at most $L_0$.
    Thus its rooted core satisfies
    \[
        3 \le k \le k_{\max} \defeq 3L_0 + 1.
    \]
    Let $\frakH_k$ be the collection of all deterministic rooted cores $(v, H)$ on $k$ vertices in $[n]$.
    To encode such a core, choose an ordered list of $k$ distinct vertices, starting with $v$, and choose
    the position at which the path ends and the cycle begins.
    The cycle is then closed by joining the last vertex
    to that position.
    This gives the upper bound $\abs{\frakH_k} \le k n^k$.
    
    On $\mathcal E$, every vertex belongs to one of the tree or unicyclic cases considered above.
    Applying the union bound to intersections with $\mathcal E$, and then using \eqref{eq:tree-message-joint-tail} and \eqref{eq:unicyclic-message-joint-tail}, gives
    \begin{align*}
        \Pr{\mathcal E \cap \set{\exists v \in V \midc Y_v > x}} &\le \sum_{v \in V}
        \Pr{\A_v^T \cap \set{Y_v > x}} + \sum_{k = 3}^{k_{\max}}
        \sum_{(v, H) \in \frakH_{k}} \Pr{\A_{v, H} \cap \set{Y_v > x}} \\
        &\le n M_X \e^{-tx} + \sum_{k = 3}^{k_{\max}} kn^k\ab(\frac{d}{n})^k \e^{k - tx} \\
        &= \e^{-tx} \ab(nM_X + \sum_{k = 3}^{k_{\max}} k(\e d)^k).
    \end{align*}
    Recall that
    \[
        L = \left\lfloor \frac{\ln n}{1000\ln(2\e d)} \right\rfloor.
    \]
    Since $r$ is fixed as $n \to \infty$, we have $k_{\max} = 12L + O_{\eta, d}(1)$    and therefore
    \[
        \sum_{k = 3}^{k_{\max}} k(\e d)^k \le k_{\max}^2(\e d)^{k_{\max}} = n^{\frac{12\ln(\e d)}{1000\ln(2\e d)} + o_n(1)} = o(n).
    \]
    Taking $x = (4/t) \ln{n}$, we conclude that
    \[
        \Pr{\mathcal{E} \cap \set{\exists v \in V \midc Y_v > \frac{4}{t}\ln n}} = O(n^{-3}).
    \]

    Finally, by \Cref{lem:locally-tree-or-unicyclic}, $\Pr{\overline{\mathcal{E}}} = o_n(1)$.
    Thus
    \[
        \Pr{\mathcal{E} \cap \set{Y_v \le \frac{4}{t} \ln{n} \text{ for every } v \in V}} \ge 1 - \Pr{\overline{\mathcal{E}}} - O(n^{-3}) = 1 - o_n(1)
    \]
    which concludes the lemma with $C = 4/t$.
\end{proof}

\subsection{Logarithmic bound for coupling independence}
In this section, we complete the proof of \Cref{thm:coupling-independence}.

\begin{lemma}\label{lem:many-good-nodes-to-boundary}
    Assume \(\LocalTreeOrUnicyclic_{5L}\cap\MajorGood_{\floor{L/4},r}\). For all sufficiently large \(n\), every path in \(\T_v\) from \(\node(v)\) to a node of positive $\cost_Z$ contains at least \(\floor{L/8}\) good nodes other than its endpoint.
\end{lemma}

\begin{proof}
    Set \(L_0=\floor{L/4}\) and \(c_0=r+R+3\).
    If \(J_v=B_G(v,2L)\), then every \(u\in\Sigma\) satisfies
    \[
        \dist_G(v,u)\ge 2L+1-c_0.
    \]
    If \(J_v=B_G(v,4L)\), its unique cycle \(C\) is contained in \(B_G(v,2L)\), and hence
    \[
        \dist_G(u,C)\ge 2L+1-c_0
        \qquad\text{for every }u\in\Sigma.
    \]
    In particular, $C\cap\Sigma=\emptyset$, so every node of positive
    $\cost_Z$ corresponds to a single vertex $u\in\Sigma$.
    Thus, for sufficiently large \(n\), the path from \(\node(v)\) to \(\node(u)\) contains an internal subpath \(P\) on \(L_0\) vertices which is disjoint from the \(r\)-neighborhood of the cycle node and does not contain the terminal node.
    
    The subpath \(P\) corresponds to a simple path in \(G\). By \(\MajorGood_{L_0,r}\), at least \(L_0/2\) of its vertices lie outside \(B_G(V^{\HD},r)\). These vertices are also at distance greater than \(r\) from the cycle and therefore are good. Since
    \[
        \ceil{L_0/2}\ge \floor{L/8},
    \]
    we conclude the lemma.
\end{proof}

Now we are ready to prove \Cref{thm:coupling-independence}.

\begin{proof}[Proof of \Cref{thm:coupling-independence}]
    Pick $L$ as in \Cref{sec:random-graph-property} and $\kappa$, $\rho$, $R$ and $r$ as \Cref{cor:local-TID-on-tree}.
    We assume $\LocalTreeOrUnicyclic_{5L}$ and $\MajorGood_{\floor{L/4}, r}$.
    The joint probability is at least $1 - o_n(1)$ by \Cref{lem:locally-tree-or-unicyclic,lem:high-degree-small-occupancy}.

    These events also ensure that $G$ is $q$-colorable.
    Indeed, no component of $G[V^{\HD}]$ contains a simple path on $\floor{L/4}$ vertices, by $\MajorGood_{\floor{L/4},r}$.
    Each such component therefore lies in a ball of radius less than $L/4$ and is a tree or unicyclic by $\LocalTreeOrUnicyclic_{5L}$.
    Color these components using three colors, then greedily color the remaining vertices, all of which have degree at most $D < q$.

    For every $\Phi \in \mathcal L(G, q)$ and $v$, by~\Cref{lem:many-good-nodes-to-boundary} and \Cref{lem:CI-by-local-message,prop:decomposed-message-property}, applied with coefficients \(\rho\) and \(\rho^{1/2}\), it holds that
    \begin{align*}
        \W(\Phi, v) \le Y(\rho^{1/2}; \T_v, \node(v)) + \frac{\W(G, q)}{2(D^R + 1)}  \rho^{\floor{L/8}/2} \,Y(\rho^{1/2}; \T_v, \node(v)).
    \end{align*}
    By \Cref{lem:message-estimation}, with probability at least $1 - o_n(1)$, $Y(\rho^{1/2}; \T_v, \node(v)) \le C_{\eta, d} \ln n$ for some constant $C_{\eta, d} > 0$.
    Then for sufficiently large $n$,
    \begin{align*}
        \W(\Phi, v) \le C_{\eta, d} \ln{n} + \eps_n \W(G, q)
    \end{align*}
    where $\eps_n = O(\rho^{\floor{L/8}/2} \ln{n}) \to 0$ as $n \to \infty$.
    Take maximum on the left side to obtain that
    \begin{align*}
        \W(G, q) \le C_{\eta, d} \ln{n} + \eps_n \W(G, q)
    \end{align*}
    meaning that $\W(G, q) \le C' \ln n$ for some constant $C' > 0$ depending only on $\eta$ and $d$.
    Finally, consider any two feasible pinnings $\sigma,\tau$ on $\Lambda$ differing only at $v\in\Lambda$.
    The common pinning on $\Lambda\setminus\set{v}$ induces a feasible list-coloring instance $\Phi\in\mathcal L(G,q)$ on $V\setminus(\Lambda\setminus\set{v})$, since pinning preserves the list-degree slack.
    The two conditional distributions are obtained by pinning $v$ in $\Phi$ to $\sigma(v)$ and $\tau(v)$.
    Their Wasserstein distance is at most $\W(\Phi,v)\le\W(G,q)$: the common pinned coordinates contribute zero, and the disagreement at $v$ is already included in the definition of $\W(\Phi,v)$.
    Thus $\CI(\mu_{G,q})\le\W(G,q)\le C'\ln n$, concluding the theorem.
\end{proof}

\begin{proof}[Proof of \Cref{cor:spectral-independence}]
    Fix a feasible pinning \(\tau\) and a feasible source pair \(v \gets a\).
    For every \(w \neq v\) and color \(c\),
    \begin{align*}
        \Psi_\mu^\tau(v \gets a, w\gets c) = \sum_{b} \mu^\tau(v \gets b)\ab(\mu^{\tau}(w \gets c \midl v \gets a) - \mu^\tau(w \gets c \midl v \gets b)).
    \end{align*}
    Therefore,
    \begin{align*}
        \sum_{w \neq v} \sum_c \abs{\Psi_\mu^\tau(v\gets a, w \gets c)} &\le 2\sum_b \mu^\tau(v \gets b) \sum_{w\neq v} \TV{\mu_w^{\tau, v \gets a}}{\mu_w^{\tau, v \gets b}} \\
        &\le 2\sum_b \mu^\tau(v \gets b) \W_1\ab(\mu^{\tau, v \gets a}, \mu^{\tau,v\gets b}) \\
        &\le 2\CI(\mu).
    \end{align*}
    Hence \(\norm{\Psi_\mu^\tau}_{\infty} \le 2\CI(\mu)\), leading to the upper bound that $\lambda_{\max}(\Psi_{\mu}^{\tau}) \le 2\CI(\mu)$.
    The conclusion follows from \Cref{thm:coupling-independence}.
\end{proof}

\subsection{Sharpness of the logarithmic bounds}

Denote by $\SI(\mu)$ the smallest $\gamma$ such that $\mu$ is $\gamma$-spectrally independent.
The next proposition shows that the logarithmic orders in \Cref{thm:coupling-independence,cor:spectral-independence} cannot be improved when the independence parameters are required to hold under all feasible pinnings.

\begin{proposition} \label{prop:independence-lower-bounds}
    Fix $\eta\in(0,1)$ and a sufficiently large constant $d>1$, and fix an integer $q\ge(1+\eta)d$ independently of $n$.
    There exists a constant $c_{d,q}>0$ such that, with probability $1-o_n(1)$ over $G\sim\G(n,d/n)$, the graph $G$ is $q$-colorable and
    \begin{align*}
        \CI(\mu_{G,q})\ge c_{d,q}\ln n,
        \qquad
        \SI(\mu_{G,q})\ge c_{d,q}\ln n.
    \end{align*}
\end{proposition}

\begin{proof}
    For an integer $k\ge3$, let $H_k$ be the tree obtained from a path on $k$ vertices by attaching $q-2$ private leaves to every path vertex.
    Write
    \begin{align*}
%        m=(q-1)k, \qquad p=\frac dn, \qquad
        A_{d,q}=\frac{(d\e^{-d})^{q-1}}{(q-2)!}, \qquad
        \ell_{d,q}= -\ln{A_{d,q}}.
    \end{align*}
    Since $q>d>1$, we have $0<A_{d,q}<1$ and $\ell_{d,q}>0$. Fix $k=\floor{\frac{\ln n}{2\ell_{d,q}}}$. Then $k\ge3$ for all sufficiently large $n$. We first claim that $G$ contains a copy of $H_k$ with high probability by a direct second moment method.
%    Take
%    \begin{align*}
%        k=\floor{\frac{\ln n}{2\ell_{d,q}}}.
%    \end{align*}
%    Then $k\ge3$ for all sufficiently large $n$ and
%    $m=O_{d,q}(\ln n)$.

    Let $Z_k$ be the number of isolated connected components of $G$ isomorphic to $H_k$.
    The automorphism group of $H_k$ has size $2((q-2)!)^k$.
    Requiring the $m-1$ tree edges to be present and every other edge incident to its vertex set to be absent gives
    \begin{align} \label{eq:decorated-path-first-moment}
        \E{Z_k}
        =\frac{(n)_m}{2((q-2)!)^k}
        p^{m-1}
        (1-p)^{m(n-m)+\binom{m}{2}-(m-1)},
    \end{align}
    where $m = (q-1)k$ is the number of vertices in $H_k$, and $p = \frac dn$.
    Since $m=O_{d,q}(\ln n)$, we have $(n)_m=n^m\exp\ab(O\ab(\frac{m^2}{n}))$, and $$(1-p)^{m(n-m)+\binom{m}{2}-(m-1)}=\exp\ab(-dm+O_d\ab(\frac{m^2+m}{n})).$$
    Substituting these estimates into \eqref{eq:decorated-path-first-moment} yields
    \begin{align*}
        \E{Z_k}
        =(1+o_n(1))\frac{n}{2d}A_{d,q}^k
        =\Theta_{d,q}(\sqrt n).
    \end{align*}
    Two distinct candidate isolated components that overlap cannot both
    occur.
    For two disjoint candidates, the ratio between their joint
    occurrence probability and the product of their individual
    occurrence probabilities is $(1-p)^{-m^2}$, because the $m^2$
    absent edges between the two vertex sets are counted twice in the
    product.
    Consequently,
    \begin{align*}
        \frac{\Var{Z_k}}{\E{Z_k}^2}
        \le \frac{1}{\E{Z_k}}+(1-p)^{-m^2}-1=o_n(1).
    \end{align*}
    Chebyshev's inequality shows that $G$ contains an isolated copy of
    $H_k$ with probability $1-o_n(1)$.

    Intersect this event with
    $\LocalTreeOrUnicyclic_{5L}\cap
    \MajorGood_{\floor{L/4},r}$.
    As observed in the proof of \Cref{thm:coupling-independence}, on
    this latter event the graph is $q$-colorable: every component of
    $G[V^{\HD}]$ is a tree or unicyclic and can be colored using three
    colors, after which all remaining vertices can be colored greedily
    because their degrees are at most $D<q$.
    The intersection still has probability $1-o_n(1)$.

    Fix an isolated copy of $H_k$ on this intersection.
    Pin all vertices outside this component to any proper coloring of
    the remaining graph.
    At every path vertex, pin its $q-2$ private leaves bijectively to
    the colors $3,\ldots,q$.
    Denote this common feasible pinning by $\tau$.
    The path then has exactly two feasible colorings, the two alternating
    colorings with colors 1 and 2, and they have equal conditional mass.

    For coupling independence, extend $\tau$ by pinning one endpoint of
    the path to color 1 or to color 2.
    These two feasible pinnings differ at exactly one coordinate and
    induce deterministic path colorings that disagree at all $k$ path
    vertices.
    Therefore
    \begin{align*}
        \CI(\mu_{G,q})\ge k.
    \end{align*}

    For spectral independence, retain the common pinning $\tau$ and
    leave the path free.
    At each path vertex, order its two feasible colors so that the first
    alternating coloring has local label 1.
    This blockwise relabeling only conjugates the influence matrix by a
    permutation matrix.
    In the new labels all path variables are the same uniform binary
    random variable.
    Every diagonal $2\times2$ block of the influence matrix is zero,
    while every off-diagonal block is
    \begin{align*}
        B=\frac12
        \begin{pmatrix}
            1&-1\\
            -1&1
        \end{pmatrix}.
    \end{align*}
    Since $B(1,-1)^{\mathsf T}=(1,-1)^{\mathsf T}$, repeating the
    vector $(1,-1)$ at all $k$ path vertices gives an eigenvector with
    eigenvalue $k-1$.
    Hence
    \begin{align*}
        \SI(\mu_{G,q})\ge k-1.
    \end{align*}

    For sufficiently large $n$, $k-1\ge (4\ell_{d,q})^{-1}\ln n$.
    Taking $c_{d,q}=(4\ell_{d,q})^{-1}$ proves the proposition.
\end{proof}

%% file: section/rapid-mixing.tex
\section{Rapid Mixing of Glauber Dynamics} \label{sec:rapid-mixing}

Assuming the $(C_{\SI} \ln{n})$-spectral independence of $\mu = \mu_{G, q}$ for $G \sim \G(n, d/n)$ where $C_{\SI}$ depends only on $\eta$ and $d$, we show how to derive \Cref{thm:rapid-mixing}, the rapid mixing of Glauber dynamics $\PGD$ for $\mu_{G, q}$.
We follow the \emph{entropic factorization} in \cite{CLV26} with the degree separation on random graphs $\G(n, d/n)$ as in \cite{BGGS24} to obtain the rapid mixing of the block dynamics.
To complete the mixing time of Glauber dynamics, we compare it with the block dynamics by the Dirichlet form.

\subsection{Preliminaries}
We first provide some necessary preliminaries.

\begin{lemma}[Connectivity of list-colorings]
\label{lem:list-recoloring-connectivity}
    Let $H$ be a finite graph with color lists $L(x)$.
    Suppose that its vertices admit an ordering
    $x_1,\ldots,x_m$ such that, writing
    \[
        H_i\defeq H[\{x_i,\ldots,x_m\}],
    \]
    we have
    \[
        \abs{L(x_i)}\ge \deg_{H_i}(x_i)+2
        \qquad\text{for every }i.
    \]
    Then proper list-colorings exist, and any two proper
    list-colorings can be joined by a sequence of proper
    list-colorings in which consecutive colorings differ
    at exactly one vertex.
\end{lemma}

\begin{proof}
    Existence follows by coloring the vertices in reverse
    order.

    We prove connectivity by induction on $m$.
    The assertion is immediate for $m\le 1$.
    Let $\sigma$ and $\tau$ be two proper list-colorings.
    By induction, their restrictions to $H-x_1$ can be
    joined by single-site recolorings.

    We lift this sequence to $H$.
    Whenever a proposed recoloring of a neighbor of $x_1$
    would assign the current color of $x_1$ to that neighbor,
    first recolor $x_1$.
    Choose its new color from $L(x_1)$ so that it avoids
    all current neighbor colors and the proposed new color.
    At most $\deg_H(x_1)+1$ colors are forbidden, whereas
    $\abs{L(x_1)}\ge\deg_H(x_1)+2$.
    Thus such a recoloring is always possible.

    After lifting the entire sequence, all vertices other
    than $x_1$ agree with $\tau$.
    Finally recolor $x_1$ to $\tau(x_1)$ if necessary.
    This last step is proper because $\tau$ is proper.
\end{proof}

\subsubsection*{Markov chains and continuous generators}
We briefly introduce the concept of Markov chains here and recommend referring to \cite{LP17}.
For a state space $\Omega \subseteq [q]^V$, a discrete-time Markov chain consists of a transition kernel $P$ on $\Omega$ with its stationary distribution $\mu$.
Its transition process $X_0, \ldots, X_t, \ldots$ is generated by $X_{t + 1} \sim P(X_t, \cdot)$ for every $t \in \mathbb{N}$.
The mixing time of $P$ is defined as
\begin{align*}
    \tau_{\mix}(P, \eps) \defeq \inf\;\set{t \ge 0 \midc \forall x \in \Omega, \TV{P^t(x, \cdot)}{\mu} \le \eps}
\end{align*}
In this paper, we only consider the \emph{irreducible and self-adjoint} Markov chains.

A continuous Markov chain is specified by a generator $\L$ with non-negative off-diagonal entries and zero row sum.
To test it, for every function $f : \Omega \to \mathbb{R}$, we observe the following function
\begin{align} \label{eq:test-continuous-Markov-chain}
    (\L{f})(x) \defeq \sum_{y \in \Omega} \L(x, y)(f(y) - f(x)).
\end{align}
For simplicity, we sometimes use \eqref{eq:test-continuous-Markov-chain} to specify the generator $\L$.
The transition semigroup is denoted by $H_t \defeq \e^{t \L}$, meaning that
\begin{align*}
    (H_t{f})(x) = \E{f(X_t) \mid X_0 = x}, \quad \frac{\d}{\d{t}} H_t{f} \bigl\vert_{t = 0} = \L{f}.
\end{align*}
Given a discrete-time Markov chain $P$, its continuous-time version is the generator $\L = P - I$ obtained by applying a $P$-update at times of a rate-one Poisson clock.
We always use the uppercase letter $P$ to denote a discrete-time Markov chain and use the calligraphic letter $\L$ to denote the continuous-time generator.

\subsubsection*{Dirichlet form and Markov chain comparison}
The \emph{Dirichlet form} of a discrete-time Markov chain $P$ is the following functional defined as: for all observable functions $f, g : \Omega \to \mathbb{R}$,
\begin{align*}
    \DF_{P}(f, g) \defeq \frac{1}{2} \sum_{x, y \in \Omega} \mu(x)P(x, y) (f(x) - f(y))(g(x) - g(y)).
\end{align*}
Meanwhile, for a continuous-time generator $\L$, the Dirichlet form is of the following form:
\begin{align*}
    \DF_{\L}(f, g) \defeq \frac{1}{2} \sum_{x, y \in \Omega} \mu(x)\L(x, y) (f(x) - f(y))(g(x) - g(y)).
\end{align*}
The following Rayleigh--Ritz variational characterization relates the Dirichlet form with the spectral gap:
\begin{align*}
    \Gap{P} &\defeq 1 - \lambda_2(P) = \inf_{\Var[\mu]{f} > 0} \frac{\DF_{P}(f, f)}{\Var[\mu]{f}}, \quad
    \Gap{\L} \defeq -\lambda_2(\L) = \inf_{\Var[\mu]{f} > 0} \frac{\DF_{\L}(f, f)}{\Var[\mu]{f}}.
\end{align*}
Set the relaxation time $T_{\rel}(\M) \defeq \Gap{\M}^{-1}$ with convention $1/0 = \infty$.

The next comparison is between the single-site Glauber dynamics $\PGD$ and the heat-bath block dynamics $\L_{\B}$.
For a subset $S \subseteq V$, define its block heat-bath operator as
\begin{align*}
    (\K_S{f})(x) \defeq \E[\mu]{f(X) \mid X(V \setminus S) = x(V \setminus S)}.
\end{align*}
For a block cover $\B = \set{B_1, \ldots, B_{r}}$ of $V$ ($\cup_{B \in \B} B = V$) with update rates $w : \B \to \mathbb R_{> 0}$, the continuous-time heat-bath block dynamics is specified by the generator $\L_{\B}$ defined as
\begin{align*}
    \L_{\B} \defeq \sum_{B \in \B} w(B) (\K_{B} - I).
\end{align*}
The single-site Glauber dynamics $\PGD$ is a special case of the heat-bath block dynamics: let $\B$ be the natural single-vertex cover of $V$ and set $w(v) = 1/\abs{V}$ for every vertex $v \in V$.
Then the transition kernel of the discrete-time single-site Glauber dynamics is
\begin{align*}
    \PGD = \frac{1}{\abs{V}} \sum_{v \in V} \K_v.
\end{align*}
The following comparison lemma builds a bridge between $\PGD$ and the block dynamics $\L_{\B}$.
\begin{lemma}[\cite{Martinelli99,LM11}] \label{lem:Markov-chain-comparison}
    Given a block cover $\B = \set{B_1, \ldots, B_r}$ with update rates $w : \B \to \mathbb{R}_{> 0}$, for every $i = 1, \ldots, r$ and every feasible partial assignment $\tau$ on $V \setminus B_i$, define the generator:
    \begin{align*}
        \ab(\L_i^{\tau}{f})(x_{B_i}) \defeq \sum_{v \in B_i} \ab(\E[\mu^{\tau}]{f(X_{B_i}) \mid X(B_i \setminus \set{v}) = x(B_i \setminus \set{v})} - f(x_{B_i})).
    \end{align*}
    If there is a factor $R > 0$ such that for every $i = 1, \ldots, r$ and every feasible partial assignment $\tau$ on $V \setminus B_i$, $T_{\rel}(\L_i^{\tau}) \le R$, then it holds that
    \begin{align*}
        T_{\rel}(\PGD) \le n R \cdot T_{\rel}(\L_{\B}) \cdot \max_{v \in V}\; \sum_{B \in \B: v \in B} w(B).
    \end{align*}
\end{lemma}
The following lemma shows how to derive the mixing time from the relaxation time.
\begin{lemma}[\cite{LP17}] \label{lem:relaxation-time-to-mixing-time}
    For a reversible, positive semidefinite discrete-time Markov chain $P$ with stationary distribution $\mu$, it holds that
    \begin{align*}
        \tau_{\mix}(P, \eps) \le T_{\rel}(P) \ab(\ln\frac{1}{\min_{x : \mu(x) > 0}\; \mu(x)} + \ln\frac{1}{\eps}).
    \end{align*}
\end{lemma}

\subsubsection*{Entropic factorization}
Given a state space $\Omega \subseteq [q]^V$ and a probability distribution $\mu$ on $\Omega$, for an observable function $f : \Omega \to \mathbb{R}$, the expectation of $f$ with respect to $\mu$ is denoted by $\E[\mu]{f} \defeq \sum_{x \in \Omega} \mu(x) f(x)$ and its variance is $\Var[\mu]{f} \defeq \E[\mu]{f^2} - \E[\mu]{f}^2$.
If $f$ is non-negative valued, define its entropy as
\begin{align*}
    \Ent[\mu]{f} \defeq \E[\mu]{f \ln f} - \E[\mu]{f} \ln \E[\mu]{f}
\end{align*}
with convention $0 \ln 0 = 0$.
For $S \subseteq V$, define $\ENT_{\mu}^{S}(f) = \E[\tau \sim \mu_{V \setminus S}]{\Ent[\mu^{\tau}]{f}}$.
The following inequality holds immediately for the tensor product.
\begin{fact}[{\cite[Fact 15]{BGGS24}}] \label{fact:tensor-product-entropic-factoriation}
    Let $\mu = \mu_1 \otimes \cdots \otimes \mu_k$ be the product distribution on $\Omega = \Omega_1 \times \cdots \times \Omega_k$ where $\mu_i$ is a distribution on $[q]^{V_i}$ and $V_1, \ldots, V_k$ are pairwise disjoint sets.
    Then for every test function $f : \Omega \to \mathbb{R}_{\ge 0}$, $\Ent[\mu]{f} \le \sum_{i = 1}^k \ENT_{\mu}^{V_i}(f)$.
\end{fact}
\begin{remark} \label{rem:variance-tensorization}
    We remark here that the same inequality also holds for $\Var[\mu]{f}$.
\end{remark}

A distribution $\mu$ on $\Omega \subseteq [q]^V$ is said to be $b$-marginally bounded if for every boundary $\Lambda \subseteq V$, every feasible partial assignment $\tau$ on $\Lambda$, $v \in V \setminus \Lambda$ and every color $c \in [q]$, either $\mu^{\tau}(v \gets c) = 0$ or $\mu^{\tau}(v \gets c) \ge b$.
A distribution is said to have
\emph{connected conditional supports} if, under every
feasible pinning, its support is connected in the graph
whose edges join configurations differing at exactly one
free coordinate.

We use the following consequence of
\cite[Lemma 4.2]{CLV26}, with its connectivity assumption
made explicit.

\begin{lemma}[{\cite[Lemma 4.2]{CLV26}}]
\label{lem:marginal-boundedness-to-entropic-factorization}
    Let $\mu$ be a probability distribution on $k\ge 1$
    coordinates.
    Suppose that $\mu$ has connected conditional supports
    and is $b$-marginally bounded for some $b\in(0,1/2]$.
    Then, for every non-negative function $f$,
    \[
        \Ent[\mu]{f}
        \le
        4b^{-(3k+1)}
        \sum_v \ENT_\mu^v(f).
    \]
\end{lemma}

\begin{proof}
    For $k=1$, the assertion follows from
    $\Ent[\mu]{f}=\ENT_\mu^v(f)$.

    For $k\ge 2$, the bound in
    \cite[Lemma 4.2]{CLV26} gives
    \[
        \Ent[\mu]{f}
        \le
        \frac{3k^2\ln(1/b)}{2b^{2k+2}}
        \sum_v\ENT_\mu^v(f).
    \]
    The elementary inequality
    \[
        \frac{3k^2\ln(1/b)}{2b^{2k+2}}
        \le 4b^{-(3k+1)},
        \qquad k\ge 2,\quad 0<b\le \frac12,
    \]
    completes the proof.
\end{proof}

The following theorem shows that spectral independence together with marginal boundedness implies the entropic factorization.
\begin{theorem}[\cite{CLV26}] \label{thm:entropic-factorization}
    For a probability distribution $\mu$ on $\Omega \subseteq [q]^V$, suppose that $\mu$ is $\gamma$-spectrally independent and $b$-marginally bounded with connected conditional supports.
    Then for every $r = 1, \ldots, \abs{V}$, there exists a factor $C_r > 0$ such that
    \begin{align} \label{eq:block-entropic-factorization}
        \frac{r}{\abs{V}} \Ent[\mu]{f} \le C_r \frac{1}{\binom{\abs{V}}{r}} \sum_{S \in \binom{V}{r}} \ENT_{\mu}^S(f).
    \end{align}
    Furthermore, if $r \ge 2R_0$ where $R_0 = \ceil{4\gamma/b^2}$, one can take
    \begin{align*}
        C_r \le \ab(\frac{\abs{V} - R_0}{r - R_0})^{R_0} \le \ab(\frac{2\abs{V}}{r})^{R_0}.
    \end{align*}
\end{theorem}
\begin{remark}
    Write $N = \abs{V}$. For $1 \le r \le N$ with $r \ge 2R_0$, the product estimate in the proof of \cite[Lemma 2.5]{CLV26} gives the following upper bound for a valid block factorization constant:
    \begin{align*}
%        C_r = \frac{r}{\abs{V}} \frac{\sum_{k = 0}^{\abs{V} - 1} \Gamma_k}{\sum_{k = \abs{V} - r}^{\abs{V} - 1} \Gamma_k}, \quad \Gamma_k = \prod_{j = 0}^{k - 1} \max\;\set{0, 1 - \frac{4\gamma}{b^2(\abs{V} - 1 - j)}}.\\
        C_r \le \prod_{i = 1}^{R_0} \frac{N - i}{r - i} \le \ab(\frac{N - R_0}{r - R_0})^{R_0} \le \ab(\frac{2N}{r})^{R_0}.
    \end{align*}
    Here the empty product is one when $R_0 = 0$.
%    For convenience, we use the explicit form in \Cref{thm:entropic-factorization}.
\end{remark}

The following proposition provides that the entropic factorization implies an upper bound for the relaxation time.
\begin{proposition}[{\cite[Fact 3.5]{CLV26}}] \label{prop:entropic-factorization-to-relaxation-time}
    If $\Ent[\mu]{f} \le C \sum_{v \in V} \ENT_{\mu}^{v}(f)$ for every observable $f$, then the following \Poincare inequality holds:
    \begin{align*}
        \Var[\mu]{g} \le C \sum_{v \in V} \E[\mu]{\Var[\mu]{g \mid X(V \setminus \set{v})}} = C\DF_{\L}(g, g)
    \end{align*}
    where $\L$ is the corresponding continuous-time single-site Glauber dynamics of $\mu$, \IE
    \begin{align*}
        \L = \sum_{v \in V} \ab(\K_v - I).
    \end{align*}
    Thus $T_{\rel}(\L) \le C$.
\end{proposition}

\subsection{High-degree buffers}
Now we show how to handle the high-degree vertices in $G \sim \G(n, d/n)$.
Recall the high-degree cutoff $D = \floor{(1 + \eta/3)d}$ and the high-degree vertex set
\begin{align*}
    V^{\HD} = \set{v \in V \mid \deg_G(v) > D}.
\end{align*}
Define $S \defeq N_{G}[V^{\HD}]$ and $U \defeq V \setminus S$.
Take $d$ sufficiently large such that $q - D \ge 3$, $q \ge 5$ and $D - 6 > d$.

The following buffered percolation accommodates the dependence between high-degree events.
\begin{lemma}[Buffered percolation] \label{lem:buffered-percolation}
    There is a fixed real number $p \in (0, 1)$ depending on $\eta$ and $d$ such that the following holds.
    Let $Z$ mark each vertex independently with probability $p$ independently of $G$ and set $A = V^{\HD} \cup Z$.
    For every fixed $J > 0$, there is a $K_J < \infty$ such that
    \begin{align*}
        \Pr[G, Z]{\max_{R\; \text{component of}\; G[N_G[A]]} \abs{R} > K_J \ln n} \le n^{-J}
    \end{align*}
    for all sufficiently large $n$.
    Moreover, with probability at least \(1-O_{\eta,d}(n^{-1})\) over \(G\) every component of $G[S]$ is a tree or unicyclic and has at most $K_1 \ln n$ vertices.
\end{lemma}
\begin{proof}
    For a fixed edge set $F$, by \Cref{lem:joint-degree-bound},
    \begin{align*}
        \Pr{B \subseteq V^{\HD} \mid F \subseteq G} \le \e^{2t\abs{F}} a_0^{\abs{B}} \quad \text{where} \quad a_0 = \exp\ab(-tD + \frac{d}{2}(\e^{2t} - 1)).
    \end{align*}
    Therefore,
    \begin{align*}
        \Pr{B \subseteq A \mid F \subseteq G} &\le \sum_{C \subseteq B} p^{\abs{B} - \abs{C}} \Pr{C \subseteq V^{\HD} \mid F \subseteq G} \le \e^{2t\abs{F}} (a_0 + p)^{\abs{B}}.
    \end{align*}
    Define
    \begin{align*}
        \psi(x) &\defeq x \ln{x} - x + 1, \quad x = \frac{D - 6}{d}, \quad t = \frac{1}{2} \ln{x}, \\
        \rho &\defeq \e^{6t} (a_0 + p) = \e^{-\frac{d}{2}\psi((D - 6)/d)} + p\e^{6t}.
    \end{align*}
    For sufficiently large $d$, pick $p$ small enough such that
    \begin{align*}
        \theta_0 \defeq (2\e d)^3 \rho < \frac{1}{16}.
    \end{align*}
    The choice of $p$ is possible and only depends on $\eta$ and $d$ since $\e^{-\frac{d}{2}\psi((D - 6)/d)}$ decays exponentially in $d$ for fixed $\eta$.

    Conditional on a revealed pair of $(G, Z)$, we say the vertices in $A$ are \emph{seeds}.
    For every non-seed vertex in $N_G[A]$, we assign it to one adjacent seed using a fixed tie-breaking rule.
    Additionally each seed owns itself.
    The assigned edges together with some seeds themselves form disjoint stars covering $N_G[A]$.
    We contract all star edges and keep other edges including loops and parallel edges.
    If a component contains at least $j$ seeds, a tree on exactly $j$ vertices can be selected in the quotient graph after contraction.
    Now we lift this quotient tree back to $G$, including its $j$ center seeds, the original edges representing its $j - 1$ quotient edges and the star arms needed to connect their endpoints to their centers.

    We count the number of original vertices after lifting.
    Observe that a quotient edge introduces at most two non-seed vertices.
    Then the lifted certificate is connected on $j \le k \le 3j - 2$ vertices.
    Note that there are exactly $k - j$ star arms in this lifting graph with $j - 1$ inter-star edges.
    Then it has $k - 1$ edges, leading to the fact that it is a tree together with connectivity.
    For such a certificate $B$, by \Cref{lem:joint-degree-bound},
    \begin{align*}
        \Pr{B \subseteq A \mid F \subseteq G} \le \e^{2t(k - 1)} (a_0 + p)^j \le \rho^j
    \end{align*}
    since $k - 1 \le 3j$.
    Then with Cayley's formula, the expected number of certificates is at most
    \begin{align*}
        \sum_{k = j}^{3j} \binom{n}{k} k^{k - 2} 2^k \ab(\frac{d}{n})^{k - 1} \rho^j \le O_d\ab(n(j + 1) \theta_0^j).
    \end{align*}
    Fix $J > 0$ and we pick $L_J$ so that $L_J \ln(1/\theta_0) > J + 4$.
    Markov inequality provides that for all sufficiently large $n$,
    \begin{align*}
        \Pr{\text{some component contains at least $L_J \ln n$ seeds}} \le n^{-J - 2}.
    \end{align*}

    Now we bound the neighborhood of seeds.
    Fix a vertex set $W$ of size $k$ and a tree $F$ on $W$.
    Conditional on $F \subseteq G$, a similar calculation as \Cref{lem:joint-degree-bound} shows, for any fixed $s > 0$,
    \begin{align*}
        \E{\e^{s \sum_{v \in W} \deg_G(v)} \mid F \subseteq G} \le \exp\ab(2s(k - 1) + \frac{kd}{2}\ab(\e^{2s} - 1)).
    \end{align*}
    Consequently by Markov inequality,
    \begin{align*}
        \Pr{F \subseteq G, \sum_{v \in W} \deg_G(v) > M \ln{n}} \le \ab(\frac{d}{n})^{k - 1} n^{-sM} \exp\ab(2s(k - 1) + \frac{kd}{2}\ab(\e^{2s} - 1)).
    \end{align*}
    Set
    \begin{align*}
        B_d \defeq \e d \exp\ab(2s + \frac{d}{2}\ab(\e^{2s} - 1)) > 1.
    \end{align*}
    Union over all spanning trees and $k = 1, \ldots, 3L_J \ln{n}$, and we obtain that
    \begin{align*}
        \Pr{\exists W \subseteq V, G[W]\; \text{connected}, \abs{W} \le 3L_J\ln{n}, \sum_{v \in W} \deg_G(v) > M \ln{n}} \le n^{1 - sM + 3L_J \ln B_d + o(1)}.
    \end{align*}
    Pick an arbitrary fixed $s > 0$ and choose $M$ large enough such that $sM > 3L_J \ln B_d + J + 4$.
    Then the probability is at most $n^{-J - 2}$.

    Now, suppose that every component of $G[N_G[A]]$ contains at most $L_J \ln n$ seeds, and for every connected vertex subset $W$ of size at most $3L_J \ln{n}$, $\sum_{v \in W} \deg_G(v) \le M \ln{n}$.
    For every connected component $R$ in $G[N_G[A]]$, let $B_R = A \cap R$ and $r = \abs{B_R} \le L_J \ln{n}$.
    Lift its quotient spanning tree connecting all its seeds.
    Its vertex set $W$ is connected and contains at most $3r \le 3L_J\ln{n}$ vertices.
    Therefore,
    \begin{align*}
        \abs{R} \le \abs{B_R} + \sum_{v \in W} \deg_G(v) \le (L_J + M) \ln n.
    \end{align*}
    Take $K_J \defeq L_J + M$ which depends only on $\eta, d$ and $J$ and we conclude the first argument.

    To show the second argument, for a connected graph $Q$, define its cycle rank:
    \begin{align*}
        \CycleRank(Q) \defeq \abs{E(Q)} - \abs{V(Q)} + 1.
    \end{align*}
    Suppose that a component $Q$ of $G[S]$ has $\CycleRank(Q)\ge 2$.
    Use the same star partition with seeds $V^{\HD}$ only, and contract its star edges, retaining loops and parallel edges.
    Each contraction preserves the cycle rank, so the quotient contains a spanning tree together with two additional edges.
    Suppose that this quotient certificate has $j$ seed vertices and hence $j+1$ edges.
    Lifting it back to $G$ gives a connected graph $F$ on $k$ vertices with
    \[
        j\le k\le j+2(j+1)=3j+2, \qquad |E(F)|=(k-j)+(j+1)=k+1,
    \]
    since each quotient edge introduces at most two non-seed endpoints and there are $k-j$ required star arms.
    Its seed set $B\subseteq V(F)$ has size $j$ and satisfies $B\subseteq V^{\HD}$.
    By \Cref{lem:joint-degree-bound},
    \[
        \Pr{B\subseteq V^{\HD}\mid F\subseteq G}
        \le \e^{2t(k+1)}a_0^j.
    \]
    Write $\beta=\e^{6t}a_0$.
    Since $\beta\le\rho<1$, $k+1\le3j+3$ and $j\ge(k-2)/3$, we have
    \[
        \e^{2t(k+1)}a_0^j
        \le \e^{6t}\beta^j
        \le \e^{6t}\beta^{(k-2)/3}.
    \]
    Therefore the expected number of such certificates is at most
    \begin{align*}
        &\sum_{j\ge1}\sum_{k=j}^{3j+2} \binom nk k^{k-2}k^4 2^k \left(\frac dn\right)^{k+1} \e^{2t(k+1)}a_0^j \\
        &\qquad\le \frac{C_{\eta,d}}{n} \sum_{k\ge1}k^3 \left((2\e d)^3\beta\right)^{k/3}
        =O_{\eta,d}(n^{-1}),
    \end{align*}
    where the constant absorbs the factor $\e^{6t}\beta^{-2/3}$, and the series converges because
    \[
        (2\e d)^3\beta
        \le (2\e d)^3\rho
        <\frac1{16}.
    \]
    Thus, with probability \(1-O_{\eta,d}(n^{-1})\), every
    component of \(G[S]\) is a tree or unicyclic. Combining this
    with the first part of the proof gives the asserted
    \(K_1\ln n\) component-size bound.
    Put everything together and we conclude the lemma.
\end{proof}

\subsection{Buffered marginals}
Now we study the marginal distribution on $U$.
Let $\nu = \mu_{U}$ be the marginal distribution of $\mu = \mu_{G, q}$ projected on $U$.
Assume the event that every component in $G[S]$ is a tree or unicyclic of size at most $K_1 \ln{n}$.

We first state the following marginal lower bound for $\nu$.
\begin{proposition} \label{prop:marginal-lower-bound}
    With high probability, $\abs{U} \ge c_d n$ where $c_d = \e^{-d}/2$.
    Under the standing assumption that every component of $G[S]$ is a tree or unicyclic, every proper coloring of $G[U]$ extends to a proper coloring of $G$.
    Furthermore, $\nu$ is $b$-marginally bounded under all its feasible pinnings with
    \begin{align*}
        b = q^{-(D + 1)}.
    \end{align*}
\end{proposition}
\begin{proof}
    Fix a proper coloring $\tau$ of $G[U]$ and define the residual lists on $S$ by
    \[
        L_\tau(x)\defeq[q]\setminus\{\tau(y):y\in N_G(x)\cap U\}.
    \]
    Since every component of $G[S]$ is a tree or unicyclic, $G[S]$ is $2$-degenerate.
    Choose an ordering $x_1,\ldots,x_m$ such that, writing $H_i=G[\{x_i,\ldots,x_m\}]$, we have $\deg_{H_i}(x_i)\le2$.
    If $x_i\in V^{\HD}$, all its neighbors belong to $S$, so
    \[
        \abs{L_\tau(x_i)}=q\ge5\ge\deg_{H_i}(x_i)+2.
    \]
    Otherwise $\deg_G(x_i)\le D$, and
    \begin{align*}
        \abs{L_\tau(x_i)}
        &\ge q-\abs{N_G(x_i)\cap U} =\deg_{G[S]}(x_i)+q-\deg_G(x_i)\\
        &\ge\deg_{H_i}(x_i)+q-D
        \ge\deg_{H_i}(x_i)+2.
    \end{align*}
    By \Cref{lem:list-recoloring-connectivity}, these lists admit a proper coloring of $G[S]$, which extends $\tau$ to $G$.
    Since $G[U]$ has maximum degree at most $D<q$, it has a proper $q$-coloring; thus $G$ is $q$-colorable and $\mu$ and $\nu$ are well-defined.
    
    Now we show the marginal boundedness of $\nu$.
    Fix a feasible pinning $\tau$ on $P \subseteq U$, a free vertex $v \in U \setminus P$ and a color $c$ which is feasible at $v$.
    Define
    \begin{align*}
        \Omega^{\tau} \defeq \set{\sigma \in \Omega(G, q) \midc \sigma(P) = \tau}, \quad \Omega^{\tau, c} \defeq \set{\sigma \in \Omega^{\tau} \midc \sigma(v) = c}.
    \end{align*}
    Then
    \begin{align*}
        \nu^{\tau}(v \gets c) = \frac{\abs{\Omega^{\tau, c}}}{\abs{\Omega^{\tau}}}.
    \end{align*}
    We construct a mapping from $\Omega^{\tau}$ to $\Omega^{\tau, c}$ as follows.
    Because $v\in U=V\setminus N_G[V^{\HD}]$, both $v$ and every neighbor $u\in N_G(v)$ have degree at most $D$.
    Moreover, feasibility of $c$ at $v$ under $\tau$ implies that no pinned neighbor $u\in N_G(v)\cap P$ has color $c$.
    Start from $\sigma \in \Omega^{\tau}$.
    Check the neighbors of $v$ in a prescribed order.
    Whenever a neighbor $u$ currently is assigned the color $c$, recolor it to the smallest color which is not $c$ and also not used by its neighbors.
    There are at least $q-(\deg_G(u)+1)\ge2$ such colors, and no vertex of $P$ is recolored; hence every step is proper and preserves $\tau$.
    After the exploration, change the color of $v$ to $c$.
    Therefore, for any output $\pi \in \Omega^{\tau, c}$, there are at most $q^{\abs{N_G[v]}}$ pre-images in $\Omega^{\tau}$ corresponding to $\pi$, leading to the bound
    \begin{align*}
        \abs{\Omega^{\tau}} \le q^{\abs{N_G[v]}} \abs{\Omega^{\tau, c}} \le q^{D + 1} \abs{\Omega^{\tau, c}}.
    \end{align*}
    This bound concludes the marginal boundedness of $\nu$.

    To see $\abs{U} \ge \e^{-d}n/2$, it is convenient to calculate the number of isolated vertices in $G$.
    A typical calculation shows that with probability $1 - o_n(1)$, $G$ contains at least $\e^{-d}n/2$ isolated vertices, and thus $\abs{U} \ge c_d n$ for $c_d = \e^{-d}/2$.
    For simplicity we omit the calculation here and suggest referring to \cite{FK15}.
\end{proof}

\begin{corollary}[Connectivity of buffered marginals]
\label{cor:buffered-conditional-connectivity}
    Assume the extension property in
    \Cref{prop:marginal-lower-bound} and $q-D\ge 2$.
    Then $\nu=\mu_U$ has connected conditional supports.
    Moreover, every coordinate marginal of a feasible
    conditional distribution of $\nu$ also has connected
    conditional supports.
\end{corollary}

\begin{proof}
    Fix a feasible pinning $\tau$ on $P\subseteq U$,
    and put $F=U\setminus P$.
    For each $x\in F$, define
    \[
        L_\tau(x)
        \defeq
        [q]\setminus
        \{\tau(y):y\in N_G(x)\cap P\}.
    \]
    By the extension property, the support of $\nu^\tau$
    is exactly the set of proper list-colorings of $G[F]$
    with lists $L_\tau$.
    Although $\nu^\tau$ need not be uniform, this
    identification of its support is sufficient.

    Since every vertex of $U$ has degree at most $D$,
    \begin{align*}
        \abs{L_\tau(x)}
        &\ge q-\abs{N_G(x)\cap P} = \deg_{G[F]}(x)+q-\deg_{G[U]}(x)\\
        &\ge \deg_{G[F]}(x)+q-D \ge \deg_{G[F]}(x)+2.
    \end{align*}
    Thus any vertex ordering satisfies the hypothesis of
    \Cref{lem:list-recoloring-connectivity}.
    Hence the support of $\nu^\tau$ is connected.
    Since $P$ and $\tau$ were arbitrary, $\nu$ has
    connected conditional supports.

    Finally, the coordinate projection of a connected
    single-site support is connected: projecting a
    single-site path gives a path after consecutive
    repetitions are removed.
    Applying this observation after any additional
    feasible pinning proves the assertion about marginals.
\end{proof}

The second ingredient on the buffered probability distribution $\nu$ is the spectral independence of it.
\begin{proposition} \label{prop:buffered-distribution-spectral-independence}
    Let $\mu$ be a distribution on $[q]^V$ and $\nu = \mu_U$ for $U \subseteq V$.
    If $\mu$ is $\gamma$-spectrally independent under every feasible pinning, then $\nu$ is also $\gamma$-spectrally independent.
\end{proposition}
\begin{proof}
    Observe that for every feasible pinning $\tau$ on $\Lambda \subseteq U$, $\Psi_{\nu}^{\tau}$ is a principal submatrix of $\Psi_{\mu}^{\tau}$.
    After symmetrizing $\Psi_{\mu}^{\tau}$ and by Cauchy interlacing theorem, it is easy to show
    \begin{align*}
        \lambda_{\max}\ab(\Psi_{\nu}^{\tau}) \le \lambda_{\max}\ab(\Psi_{\mu}^{\tau}) \le \gamma.
    \end{align*}
    Thus we conclude this proposition.
\end{proof}

\subsection{Entropic factorization}
Now we establish the entropic factorization for $\nu$.
Pick $p$ as in \Cref{lem:buffered-percolation}, $\zeta = p/4$ and set $r = \floor{\zeta \abs{U}}$, $\vartheta = r/\abs{U}$.
For sufficiently large $n$, $\zeta/2 \le \vartheta \le \zeta$ and $r \ge 2R_0$.
Pick
\begin{align*}
    a \defeq \ab(\frac{4C_{\SI}}{b^2} + 2) \ln{(4/\zeta)} + 2.
\end{align*}
\begin{lemma} \label{lem:small-block}
    Fix $J > 2a + 6$ and take $K = K_J$ in \Cref{lem:buffered-percolation}.
    With high probability over $G \sim \G(n, d/n)$,
    \begin{align*}
        \Pr[T \sim \binom{U}{r}]{\max_{R\;\text{component of}\; G[S \cup T]} \abs{R} > K\ln{n} \mid G} \le n^{-a - 2}.
    \end{align*}
\end{lemma}
\begin{proof}
    Conditional on $G$, couple the uniform subset $T \in \binom{U}{r}$ to Bernoulli-$p$ marks on $U$ as follows: draw a permutation $u_1, \ldots, u_{\abs{U}}$ of $U$ uniformly at random and independently draw $B \sim \Binomial(\abs{U}, p)$.
    Set
    \begin{align*}
        T = \set{u_1, \ldots, u_r}, \quad Z_U = \set{u_1, \ldots, u_B}.
    \end{align*}
    It is clear that $T$ is a uniform $r$-subset of $U$ and $Z_U$ follows the Bernoulli-$p$ marks on $U$.
    Additionally, mark the vertices of $S$ independently with probability $p$.
    Let $Z$ be the full marked set.
    Then for every $W \subseteq V$,
    \begin{align*}
        \Pr{Z = W \mid G} = p^{\abs{W}}(1 - p)^{n - \abs{W}}.
    \end{align*}
    Since $\vartheta = r/\abs{U} \le \zeta = p/4$ and $\E{B \mid G} = p\abs{U}$, the Chernoff bound provides that
    \begin{align*}
        \Pr{B < r \mid G} \le \exp\ab(-\frac{9p\abs{U}}{32}).
    \end{align*}
    On the event $B \ge r$, it holds that $T \subseteq Z_U$ and hence $T \subseteq Z$, meaning that $S \cup T \subseteq N_G[V^{\HD} \cup Z]$.
    Therefore, conditional on the event $B \ge r$, the event that a component of $G[S \cup T]$ has size surpassing $K\ln{n}$ forces the event that a component of $G[N_G[V^{\HD} \cup Z]]$ is larger than $K\ln{n}$.

    Denote by $\mathcal{E}$ the event that there is a component of $G[S \cup T]$ of size larger than $K \ln{n}$.
    By \Cref{lem:buffered-percolation}, it holds that
    \begin{align*}
        \E{\Pr{\mathcal{E}} \cdot \id{\abs{U} \ge c_d n}} &\le n^{-J} + \E[G]{\id{\abs{U} \ge c_d n} \cdot \Pr{B < r \mid G}} \\
        &\le n^{-J} + \exp\ab(-\frac{9p c_d n}{32}).
    \end{align*}
    By Markov inequality,
    \begin{align*}
        \Pr[G]{\Pr{\mathcal{E}} > n^{-a-2} \land \id{\abs{U} \ge c_d n}} \le n^{a + 2 - J} + n^{a + 2}\exp\ab(-\frac{9p c_d n}{32}).
    \end{align*}
    Union bound together with \Cref{prop:marginal-lower-bound} provides that
    \begin{align*}
        \Pr[G]{\Pr{\mathcal{E}} > n^{-a-2}} \le o_n(1) + n^{a + 2 - J} + n^{a + 2} \exp\ab(-\frac{9p c_d n}{32})
    \end{align*}
    which is still $o_n(1)$ for sufficiently large $n$.
    Thus we conclude the lemma.
\end{proof}

Now we are ready to establish the following entropic factorization for $\nu$.
\begin{proposition} \label{prop:buffered-entropic-factorization}
    With high probability over $G \sim \G(n, d/n)$, for every observable test function $f$,
    \begin{align*}
        \Ent[\nu]{f} \le C_{\LD} \sum_{u \in U} \ENT_{\nu}^{u}(f), \quad C_{\LD} = 2C_r \cdot 4b^{-(3K\ln{n} + 1)}
    \end{align*}
    where $K$ is decided by \Cref{lem:small-block}.
    Consequently, the continuous-time marginal Glauber dynamics $\L^{\nu}$ with rate one per vertex of $U$ has relaxation time at most $C_{\LD} \le n^B$ for a fixed constant $B > 0$ depending only on $\eta, d$ and $q$.
\end{proposition}
\begin{proof}
    Assume the joint event of \Cref{prop:marginal-lower-bound,lem:small-block}.
    In addition, assume that $\mu = \mu_{G, q}$ is $(C_{\SI}\ln{n})$-spectrally independent for some constant $C_{\SI} = C_{\SI}(\eta, d)$ and therefore $\nu$ is $(C_{\SI}\ln{n})$-spectrally independent.
    By \Cref{cor:buffered-conditional-connectivity}, $\nu$ has connected conditional supports. Together with its marginal boundedness and spectral independence, \Cref{thm:entropic-factorization} gives
    \begin{align} \label{eq:buffered-block-entropic-factorization}
        \vartheta \cdot \Ent[\nu]{f} \le C_r \E[T \sim \binom{U}{r}]{\ENT_{\nu}^{T}(f)}
    \end{align}
    where $C_r \le n^{a}$ (recall that $a = (4C_{\SI}/b^2 + 2) \ln{(4/\zeta)} + 2$).
    For simplicity, we say an $r$-subset $T$ of $U$ is good if every component of $G[S \cup T]$ has at most $K\ln{n}$ vertices and otherwise it is bad.
    By \Cref{lem:small-block}, $\Pr[T]{T\;\text{is bad}} \le n^{-a-2}$.

    Now we fix a good subset $T$ and a feasible assignment $\tau$ on $U \setminus T$.
    Let $R$ range over all connected components of $G[S \cup T]$ and set $A_R = R \cap T$.
    Then it holds that
    \begin{align*}
        \nu^{\tau} = \bigotimes_{R : A_R \neq \emptyset} \nu\ab(A_R \gets \cdot \mid \tau).
    \end{align*}
    For every $R$, since $T$ is good, by \Cref{prop:marginal-lower-bound}, the following properties hold for $\nu(A_R \gets \cdot \midl \tau)$:
    \begin{itemize}
        \item it has at most $K \ln{n}$ coordinates since $\abs{A_R} \le \abs{R} \le K\ln{n}$;
        \item it has connected conditional supports by \Cref{cor:buffered-conditional-connectivity};
        \item it is $b$-marginally bounded for $b = q^{-(D + 1)}$.
    \end{itemize}
    First apply \Cref{fact:tensor-product-entropic-factoriation} to this product distribution:
    \begin{align*}
        \Ent[\nu^{\tau}]{f}
        \le \sum_{R:A_R\neq\emptyset}\ENT_{\nu^{\tau}}^{A_R}(f).
    \end{align*}
    For each nonempty $A_R$, fix the coordinates in $T\setminus A_R$.
    By independence, the conditional law on $A_R$ is still $\nu(A_R\gets\cdot\midl\tau)$.
    Apply \Cref{lem:marginal-boundedness-to-entropic-factorization} to $f$ with these outside coordinates fixed, and then average over them, obtaining
    \begin{align*}
        \ENT_{\nu^{\tau}}^{A_R}(f)
        \le 4b^{-(3K\ln{n} + 1)}\sum_{u\in A_R}\ENT_{\nu^{\tau}}^u(f).
    \end{align*}
    Summing over the components yields
    \begin{align*}
        \Ent[\nu^{\tau}]{f} \le 4b^{-(3K\ln{n} + 1)} \sum_{u \in T} \ENT_{\nu^{\tau}}^{u}(f).
    \end{align*}
    Taking expectation over both sides, by the tower property of expectation, for every good $r$-subset $T$,
    \begin{align*}
        \ENT_{\nu}^{T}(f) \le 4b^{-(3K\ln{n} + 1)} \sum_{u \in T} \ENT_{\nu}^{u}(f).
    \end{align*}
    For a bad $r$-subset $T$, the chain rule of entropy shows
    \begin{align*}
        \Ent[\nu]{f} = \ENT_{\nu}^{T}(f) + \Ent[\nu]{\E[\nu]{f \mid X(U \setminus T)}}
    \end{align*}
    which implies that $\ENT_{\nu}^{T}(f) \le \Ent[\nu]{f}$ by non-negativity.

    Back to \eqref{eq:buffered-block-entropic-factorization}, by \Cref{lem:small-block},
    \begin{align*}
        \vartheta \cdot \Ent[\nu]{f} \le C_r\ab(\vartheta \cdot 4b^{-(3K\ln{n} + 1)} \sum_{u \in U} \ENT_{\nu}^{u}(f) + n^{-a - 2} \cdot \Ent[\nu]{f}).
    \end{align*}
    For sufficiently large $n$, $C_r n^{-a-2} \le n^{-2} \le \vartheta/2$.
    Therefore,
    \begin{align*}
        \Ent[\nu]{f} \le 2C_r \cdot 4b^{-(3K\ln{n} + 1)} \sum_{u \in U} \ENT_{\nu}^{u}(f)
    \end{align*}
    which is the desired entropic factorization.
    With the entropic factorization for $\nu$ and the choice of $B = a + 3K\ln{(1/b)} + 3$, the relaxation time is a direct corollary of \Cref{prop:entropic-factorization-to-relaxation-time}.
\end{proof}

\subsection{Comparison back to the original chain}
Now we put back our attention to the original distribution $\mu$ and its corresponding discrete-time Glauber dynamics $\PGD$.
For every $u \in U$, define $B_u = S \cup \set{u}$.
Then the block cover $\B$ is defined as $\B = \set{B_u \midc u \in U}$.
Define the operator $P_u$ as: for every test function $f$, $P_u{f} = \E[\mu]{f \midl X(U \setminus \set{u})}$.
Hence the generator and its Dirichlet form are set by
\begin{align*}
    \L_{\B} \defeq \sum_{u \in U} (P_u - I), \quad \DF_{\L_{\B}}(f, f) = \sum_{u \in U} \E[\mu]{\Var{f \mid X(U \setminus \set{u})}}.
\end{align*}
\begin{lemma} \label{lem:block-operator-spectral-gap}
    The following identity holds for this block dynamics:
    \begin{align*}
        \Gap{\L_{\B}} = \min\;\set{\abs{U}, \Gap{\L^{\nu}}}
    \end{align*}
    where $\L^{\nu}$ is the continuous-time single-site Glauber dynamics for $\nu$, with the $\abs{U}$ branch omitted if its invariant subspace is zero.
    Consequently, its relaxation time is at most $T_{\rel}\ab(\L_{\B}) \le C_{\LD}$ assuming \Cref{prop:buffered-entropic-factorization}.
\end{lemma}
\begin{proof}
    We decompose the space $L^2(\mu)$.
    Define the operator $Q$ as for every $f \in L^2(\mu)$,
    \begin{align*}
        Q{f} = \E[\mu]{f \mid X(U)}.
    \end{align*}
    Consider the two subspaces
    \begin{align*}
        \calF \defeq \set{f \in L^2(\mu) \midc f(x) = f(y)\;\text{whenever}\; x(U) = y(U)}, \quad \calH \defeq \set{h \midc \E[\mu]{h \mid X(U)} = 0}
    \end{align*}
    Therefore, for every $f \in L^2(\mu)$, $f$ has a unique decomposition $f = Q{f} + (f - Q{f})$ such that $Q{f} \in \calF$ and $f - Q{f} \in \calH$.
    It is easy to verify that for every $g \in \calF$ and $h \in \calH$, $\E[\mu]{g h} = 0$.
    Then we represent $L^2(\mu) = \calF \oplus \calH$.

    Since $\nu$ is the marginal of $\mu$ projected on $U$, the first identity is $\E[\mu]{g(X(U))^2} = \E[\nu]{g^2}$.
    Moreover, let $f(x) = g(x(U))$.
    It holds that
    \begin{align*}
        P_u{f} = \E[\mu]{g(X(U)) \mid X(U \setminus \set{u})} = \E[\nu]{g \mid X(U \setminus \set{u})}.
    \end{align*}
    Summing over $u$, we obtain that
    \begin{align*}
        \L_{\B}(g(X(U))) = \ab(\L^{\nu} g)(X(U)).
    \end{align*}
    For $h \in \calH$, the tower property provides that
    \begin{align*}
        P_u{h} = \E[\mu]{h \mid X(U \setminus \set{u})} = \E[\mu]{\E[\mu]{h \mid X(U)}\mid X(U \setminus \set{u})} = 0.
    \end{align*}
    Hence $\L_{\B}h = -\abs{U}h$.
    Then for every non-constant $f \in L^2(\mu)$ with decomposition $f = g + h$,
    \begin{align*}
        \frac{\DF_{\L_{\B}}(f, f)}{\Var[\mu]{f}} = \frac{\DF_{\L^{\nu}}(g, g) + \abs{U} \cdot \E[\mu]{h^2}}{\Var[\nu]{g} + \E[\mu]{h^2}} \ge \min\;\set{\Gap{\L^{\nu}}, \abs{U}}.
    \end{align*}
    To achieve equality, we only need to consider $\calF$ and $\calH$ respectively.
    Moreover, when $\calH = \set{0}$, the $\abs{U}$ is obviously omitted by calculation.
    Thus we conclude the spectral gap for $\L_{\B}$.

    To show the relaxation time, it is clear that $T_{\rel}(\L^{\nu}) \le C_{\LD}$.
    When $U \neq \emptyset$, by \Cref{prop:entropic-factorization-to-relaxation-time,prop:buffered-entropic-factorization},
    \begin{align*}
        \Var[\nu]{g} &\le C_{\LD} \sum_{u \in U} \E[\nu]{\Var{g \mid X(U \setminus \set{u})}} \le C_{\LD} \abs{U} \cdot \Var[\nu]{g}.
    \end{align*}
    Therefore $\abs{U}^{-1} \le C_{\LD}$.
    Hence by \Cref{prop:entropic-factorization-to-relaxation-time},
    \begin{align*}
        T_{\rel}(\L_{\B}) \le \max\;\set{T_{\rel}\ab(\L^{\nu}), \abs{U}^{-1}} \le C_{\LD}.
    \end{align*}
    Thus we conclude the full lemma.
\end{proof}

The following proposition provides the relaxation time of Glauber dynamics restricted on $B_u$.
\begin{proposition} \label{prop:single-block-relaxation-time}
    Assume the event in \Cref{lem:buffered-percolation}.
    For every block $B_u$ and every feasible outside pinning $\tau$ on $V \setminus B_u$, the continuous-time single-site Glauber dynamics is connected and has relaxation time at most $q^{2(1 + D K_1) \ln{n} + 1}$ for all sufficiently large $n$.
\end{proposition}
\begin{proof}
    Fix a vertex $u \in U$ and a feasible pinning $\tau$ on $V \setminus B_u = U \setminus \set{u}$.
    We firstly consider the connected component containing $u$ in $G[S \cup \set{u}]$.
    Note that $u$ connects at most $D$ components in $G[S]$.
    Therefore this component is of size at most $1 + DK_1\ln{n} \le (1 + D K_1)\ln{n}$ for sufficiently large $n$.
    Other components in $G[S]$ remain unchanged.
    Hence every component in $G[B_u]$ has at most $(1 + D K_1) \ln{n}$ vertices.
    We next prove connectivity under the outside pinning. Let $H=G[B_u]$, and for each $x\in B_u$ define
    \[
        L_\tau(x)
        \defeq
        [q]\setminus
        \{\tau(y):y\in N_G(x)\setminus B_u\}.
    \]
    The conditional support is exactly the set of proper list-colorings of $H$ with these lists.

    First, $H$ is $3$-degenerate.
    Indeed, every component of $G[S]$ is a tree or unicyclic,
    so every nonempty induced subgraph of $G[S]$ has a vertex
    of degree at most $2$.
    For any nonempty $A\subseteq B_u$ with
    $A\setminus\{u\}\neq\emptyset$, choose such a vertex in
    $G[A\setminus\{u\}]$.
    It has degree at most $3$ in $G[A]$.
    The remaining case $A=\{u\}$ is immediate.
    
    Choose a degeneracy ordering $x_1,\ldots,x_m$ of $H$,
    and write
    \[
        H_i=H[\{x_i,\ldots,x_m\}],
        \qquad
        \deg_{H_i}(x_i)\le 3.
    \]
    We verify the list-size condition in
    \Cref{lem:list-recoloring-connectivity}.
    
    If $x_i\in V^{\HD}$, then all neighbors of $x_i$ belong to
    $S=N_G[V^{\HD}]\subseteq B_u$.
    Thus no outside pinning removes a color from its list, and
    \[
        \abs{L_\tau(x_i)}
        =q
        \ge 5
        \ge \deg_{H_i}(x_i)+2.
    \]
    If $x_i\notin V^{\HD}$, then $\deg_G(x_i)\le D$, and
    \begin{align*}
        \abs{L_\tau(x_i)}
        &\ge q-\abs{N_G(x_i)\setminus B_u} = \deg_H(x_i)+q-\deg_G(x_i)\\
        &\ge \deg_H(x_i)+q-D \ge \deg_{H_i}(x_i)+2.
    \end{align*}
    Therefore \Cref{lem:list-recoloring-connectivity} applies.
    Every proper single-site recoloring has positive
    heat-bath transition rate, so the conditional
    continuous-time Glauber dynamics on $B_u$ is irreducible.
    
    To bound the relaxation time, for a component $R$ of $G[S \cup \set{u}]$, let $s$ be its size and $\Omega_R$ be its proper list-colorings.
    Write $N = \abs{\Omega_R} \le q^s$ and $\pi$ be the conditional probability distribution of $\mu^{\tau}$ on $R$.
    Observe that $\pi$ is uniform over $\Omega_R$.
%    For a transition edge $e$ from $\sigma$ to $\xi$ and every function $g$, define $\Delta_e g \defeq g(\sigma) - g(\xi)$.
%    For every ordered pair $(\sigma, \xi)$, we choose a simple transition path $P_{\sigma, \xi}$.
%    Each path has length at most $N - 1$.
%    By Cauchy-Schwarz inequality,
%    \begin{align*}
%        \Var[\pi]{g} &= \frac{1}{2N^2} \sum_{\sigma, \xi} (g(\sigma) - g(\xi))^2 \\
%        &\le \frac{1}{2N^2} \sum_{\sigma, \xi} (N - 1) \sum_{e \ni P_{\sigma, \xi}} (\Delta_e g)^2 \\
%        &\le \frac{N - 1}{2} \sum_{e} (\Delta_e g)^2.
%    \end{align*}
%    For the Dirichlet form denoted by $\DF_R(g, g)$, since the transition probability is at least $1/q$,
%    \begin{align*}
%        \DF_R(g, g) &\ge \frac{1}{2} \sum_{\sigma, \xi} \frac{1}{N} \frac{1}{q} (\Delta_e g)^2 \ge \frac{1}{qN} \sum_{e} (\Delta_e g)^2.
%    \end{align*}
    Let $\mathcal{L}_R$ be the generator of the continuous-time Glauber dynamics on $\Omega_R$. Define the directed and undirected transition-edge sets by
    \begin{align*}
        \overrightarrow{\mathcal{T}}_R
        \defeq
        \set{(\sigma,\xi)\in\Omega_R^2:
              \sigma\neq\xi,\ \mathcal{L}_R(\sigma,\xi)>0},\quad
        \mathcal{T}_R
        \defeq
        \set{\{\sigma,\xi\}:    (\sigma,\xi)\in\overrightarrow{\mathcal{T}}_R}.
    \end{align*}
    For each $e=\{\sigma,\xi\}\in\mathcal{T}_R$, fix an arbitrary orientation and define
    \[
        \Delta_e g\defeq g(\sigma)-g(\xi).
    \]
    For every ordered pair $(\sigma,\xi)\in\Omega_R^2$, choose a simple path $P_{\sigma,\xi}$ in the undirected transition graph. Since this path has length at most $N-1$, the Cauchy--Schwarz inequality gives
    \begin{align*}
        \Var[\pi]{g}
        &=\frac{1}{2N^2}
          \sum_{\sigma,\xi\in\Omega_R}
          \bigl(g(\sigma)-g(\xi)\bigr)^2\\
        &\le
          \frac{1}{2N^2}
          \sum_{\sigma,\xi\in\Omega_R}
          (N-1)\sum_{e\in P_{\sigma,\xi}}(\Delta_e g)^2\\
        &\le
          \frac{N-1}{2}\sum_{e\in\mathcal{T}_R}(\Delta_e g)^2.
    \end{align*}
    Since every off-diagonal transition rate is at least $1/q$,
    \begin{align*}
        \DF_R(g,g)
        &=\frac{1}{2}
          \sum_{(\sigma,\xi)\in\overrightarrow{\mathcal{T}}_R}
          \frac{1}{N}\mathcal{L}_R(\sigma,\xi)
          \bigl(g(\sigma)-g(\xi)\bigr)^2\\
        &\ge
          \frac{1}{2qN}
          \sum_{(\sigma,\xi)\in\overrightarrow{\mathcal{T}}_R}
          \bigl(g(\sigma)-g(\xi)\bigr)^2\\
        &=\frac{1}{qN}
          \sum_{e\in\mathcal{T}_R}(\Delta_e g)^2.
    \end{align*}
    Therefore the following \Poincare inequality holds:
    \begin{align*}
        \Var[\pi]{g} \le \frac{qN(N - 1)}{2} \DF_R(g, g).
    \end{align*}
    Thus the relaxation time for a single component $R$ is at most
    \begin{align*}
        qN^2 \le q^{1 + 2s} \le q^{1 + 2(1 + D K_1)\ln{n}}.
    \end{align*}
    By \Cref{fact:tensor-product-entropic-factoriation,rem:variance-tensorization}, we conclude the relaxation time for the continuous-time Glauber dynamics for $\mu_{B_u}^{\tau}$.
\end{proof}

\begin{proof}[Proof of \Cref{thm:rapid-mixing}]
    Assume the joint event of \Cref{cor:spectral-independence,lem:buffered-percolation,prop:marginal-lower-bound,lem:small-block}.
    This joint event happens with probability $1 - o_n(1)$.
    Write $\PGD$ as the discrete-time Glauber dynamics for $\mu$, $\L^{\mu}$ as the continuous-time single-site chain with rate one clock per vertex of $V$, $\L_{\B}$ as the above heat-bath block dynamics.
    By \Cref{lem:Markov-chain-comparison} with \Cref{prop:buffered-entropic-factorization,lem:block-operator-spectral-gap,prop:single-block-relaxation-time},
    \begin{align*}
        \Var[\mu]{f} \le C_{\LD} \DF_{\L_{\B}}(f, f) \le C_{\LD} \cdot n \cdot q^{1 + 2(1 + DK_1)\ln{n}} \DF_{\L^{\mu}}(f, f) \le n^{C_1} \cdot \DF_{\PGD}(f, f)
    \end{align*}
    for the constant $C_1 = B + 2(1 + D K_1) \ln{q} + 10$ which depends only on $\eta, d$ and $q$.
    Therefore $T_{\rel}(\PGD) \le n^{C_1}$.
    Each $\K_v$ is a conditional expectation operator and hence an orthogonal projection in $L^2(\mu)$.
    Thus $\PGD=\abs{V}^{-1}\sum_{v\in V}\K_v$ is positive semidefinite, as required by \Cref{lem:relaxation-time-to-mixing-time}.
    Note that for every feasible coloring $\sigma \in \Omega_{G, q}$, $\mu(\sigma) \ge q^{-n}$.
    Therefore by \Cref{lem:relaxation-time-to-mixing-time}, $\PGD$ exhibits mixing time at most $n^{C_1}\ab(n \ln{q} + \ln(1/\eps)) \le n^C (1 +\ln(1/\eps))$ for some constant $C$ depending only on $\eta, d$ and $q$.
\end{proof}

%% file: appendix/CLMM-clarification.tex
\section{SSM and TID for List-coloring Instances} \label{sec:CLMM-clarification}

We record the list-coloring extension of \Cref{prop:SSM-TID-on-tree} used in \Cref{sec:coupling-independence}.

\begin{lemma} \label{lem:CLMM-list-coloring-extension}
    Fix $\eta \in (0, 1)$, let $\Delta \ge 7/\eta^2$, and suppose that $q = (1 + \eta) \Delta$.
    Let $T$ be a tree of maximum degree at most $\Delta$, and let $\Phi = (T, L)$ be a feasible list-coloring instance satisfying
    \begin{equation} \label{eq:list-degree-gap}
        \abs{L(v)} - \deg_T(v) \ge q - \Delta \qquad \text{for every } v \in V(T).
    \end{equation}
    Then for any feasible pinning $\pi$ on $T$, the Gibbs distribution $\mu_{\Phi}^{\pi}$ exhibits strong spatial mixing and total influence decay with exponential decay rate $\exp(-\eta/4)$ and constants
    \begin{align*}
        C_{\SSM} &= \frac{2q}{(\eta\Delta)^{1/2}} \exp(1/\eta + \eta/4), \\
        C_{\TID} &= \frac{q^3}{(\eta\Delta)^{3/2}(q - 1)} \exp(3/\eta)
    \end{align*}
    respectively.
    The multiplicative SSM bound is understood at distances $\ell\ge2$ and for colors feasible at the root under the common pinning, as in \Cref{def:strong-spatial-mixing}.
\end{lemma}
\begin{proof}
    Fix a feasible pinning $\pi$ on $\Lambda \subseteq V(T)$.
    Let $T^\pi \defeq T[V(T) \setminus \Lambda]$
    be the remaining free forest, and for every free vertex $v$ let 
    \[
        L^\pi(v) \defeq L(v) \setminus \set{\pi(u) \midc u \in N_T(v) \cap \Lambda}.
    \]
    It clearly holds that
    \begin{equation} \label{eq:pinned-list-gap}
    \begin{aligned}
        \abs{L^\pi(v)} - \deg_{T^\pi}(v) &\ge \abs{L(v)} - \abs{N_T(v) \cap \Lambda} - \ab(\deg_T(v) - \abs{N_T(v) \cap \Lambda})\\ &= \abs{L(v)} - \deg_T(v) \ \ge q - \Delta.
    \end{aligned}
    \end{equation}
    We now convert the list constraints into ordinary coloring constraints.
    For every free vertex $v$ and every color $c \in [q] \setminus L^\pi(v)$, attach a new leaf $z_{v, c}$ to $v$ and pin $z_{v, c}$ to color $c$.
    Denote the resulting augmented forest by $\widehat{T}^\pi$.
    For every original free vertex $v$, by \eqref{eq:pinned-list-gap},
    \begin{align*}
        \deg_{\widehat{T}^\pi}(v) &= \deg_{T^\pi}(v) + q - \abs{L^\pi(v)} \le \Delta.
    \end{align*}
    Meanwhile, all newly added vertices are leaves.
    Hence every component of $\widehat{T}^\pi$ is a tree of maximum degree at most $\Delta$.

    Observe that conditional on the colors of the newly added leaves, the uniform distribution on proper $q$-colorings of $\widehat{T}^\pi$, restricted to the original free vertices, is exactly $\mu_{\Phi}^{\pi}$.
    Apply \cite[Proposition~8.8]{CLMM26} to every component of $\widehat{T}^\pi$.
    It gives strong spatial mixing and total influence decay with rate $\exp(-\eta/4)$ and the displayed
    constants.
    The list-degree slack in \eqref{eq:pinned-list-gap} ensures that a color feasible at a free vertex under the common pinning remains feasible under additional feasible pinnings at distance at least two: fix that color and greedily color the remaining free forest while respecting the pinned boundary.
    Hence the conditional probabilities in the multiplicative SSM bound are positive.
    Adding pinned leaves does not change distances between original vertices, and the added vertices never appear among the free vertices in the definitions of SSM or TID.
    Thus the same bounds hold for $\mu_\Phi^\pi$ on the original list-coloring instance, concluding this lemma.
\end{proof}

%% file: appendix/Poisson-estimation.tex
\section{Poisson Estimations on Galton--Watson Trees} \label{sec:Poisson-estimation}

This section provides a proof of \Cref{prop:GWT-message-well-definedness-finiteness}.
At the beginning, we introduce the following properties of a Poisson random variable and its related variants.
\begin{lemma} \label{lem:Poisson-inequality}
    Let $K \sim \Poisson(d + 1)$ be a Poisson random variable of rate $d + 1$.
    Recall that $D = \floor{(1 + \eta/3)d}$.
    The following inequalities hold for sufficiently large $d > 1$.
    \begin{itemize}
        \item \textbf{Poisson tail bound:} there exists some constant $C_\tail = C_\tail(\eta) > 0$ such that
        \begin{align*}
            p_{\scrH} \defeq \Pr{K > D - 3} \le \e^{-C_\tail d}.
        \end{align*}
        \item \textbf{Finite moments:} for every integer $p \ge 2$, there exists some constant $C_\moment = C_\moment(\eta) > 0$ such that
        \begin{align*}
            \norm{K}_p, \norm{K + 1}_p, \norm{K \mid K > D - 3}_p \le C_\moment(d + p).
        \end{align*}
    \end{itemize}
\end{lemma}
\begin{proof}
    The Poisson tail bound is a direct corollary of the Markov inequality: for every $x > d + 1$,
    \begin{align*}
        \Pr{K \ge x} \le \exp\ab(-(d + 1) \ab(\frac{x}{d + 1} \ln \frac{x}{d + 1} - \frac{x}{d + 1} + 1)).
    \end{align*}
    For sufficiently large $d > 1$, $D - 2 \ge (1 + \eta/5)(d + 1)$ and thus
    \begin{align*}
        \Pr{K > D - 3} \le \Pr{K \ge D - 2} \le \e^{-\eta^2 d/60}.
    \end{align*}

    To show finite moments of random variables, we show the standard calculation for the moment generating function of $K$ as
    \begin{align*}
        \E{\e^K} = \exp((d + 1)(\e - 1))
    \end{align*}
    and therefore for every $t \ge 0$, by the Markov inequality,
    \begin{align*}
        \Pr{K > (d + 1)(\e - 1) + t} \le \e^{-t}.
    \end{align*}
    For the random variable $X = \max\;\set{K - (d + 1)(\e - 1), 0}$, by the tail integral formula,
    \begin{align*}
        \E{X^p} = \int_0^\infty p x^{p - 1} \Pr{X > x} \d{x} \le \int_0^\infty p x^{p - 1} \e^{-x} \d{x} = p! \le p^p.
    \end{align*}
    Then by Minkowski's inequality,
    \begin{align*}
        \norm{K}_p \le (\e - 1)(d + 1) + p.
    \end{align*}
    
    For the random variable $K + 1$, the Minkowski's inequality gives that
    \begin{align*}
        \norm{K + 1}_p \le 1 + \norm{K}_p \le 1 + (\e - 1)(d + 1) + p.
    \end{align*}

    To show the finite moments of $K \;|\; K > D - 3$, we first pick $d > 1$ sufficiently large such that $K_0 \defeq D - 2 \ge (1 + \eta/5)(d + 1)$.
    For integers $j \ge 0$ and $k > 0$, it holds that
    \begin{align*}
        \frac{\Pr{K \ge K_0 + j + k}}{\Pr{K \ge K_0 + j}} \le \frac{(d + 1)^k}{(K_0 + j + k) \cdots (K_0 + j + 1)} \le \ab(\frac{d + 1}{K_0 + 1})^k.
    \end{align*}
    Hence we have
    \begin{align*}
        \Pr{K \ge K_0 + k \mid K \ge K_0} \le \ab(\frac{d + 1}{K_0 + 1})^k.
    \end{align*}
    Consider the random variable $X = K - K_0$ conditional on $K \ge K_0$ and set $\theta = -\ln\frac{d + 1}{K_0 + 1}$.
    For every $x \ge 0$, it holds that
    \begin{align*}
        \Pr{X > x} \le \ab(\frac{d + 1}{K_0 + 1})^x \le \e^{-\theta x}.
    \end{align*}
    By the tail integral formula,
    \begin{align*}
        \E{X^p} = \int_0^\infty p x^{p - 1} \Pr{X > x} \d{x} \le \int_0^\infty p x^{p - 1} \e^{-\theta x} \d{x} = \frac{p!}{\theta^p} \le \ab(\frac{p}{\theta})^p.
    \end{align*}
    Then by the Minkowski's inequality,
    \begin{align*}
        \norm{K \mid K > D - 3}_p = \norm{K \mid K \ge K_0}_p \le K_0 + \frac{p}{\theta} \le \frac{6}{\eta}(d + p).
    \end{align*}

    Pick $C_{\moment} = 6/\eta$ and we conclude the finite moments of these three random variables.
\end{proof}

Based on the finite exponential moment of the offspring variable, we show a similar bound for the moment of the population through any fixed height.
\begin{lemma} \label{lem:population-moment}
    Let $\T$ be a Galton-Watson tree rooted at $\varrho$ whose nodes have independent offspring counts distributed as a non-negative integer-valued random variable $U$.
    Define $N_j$ as the number of nodes of depth $j$ and $T_h \defeq \sum_{j = 0}^{h} N_j$.
    Then:
    \begin{enumerate}
        \item If $\E{\e^{tU}} < \infty$ for some $t > 0$, then for every fixed integer $h \ge 0$, there exists $s_h > 0$ such that $\E{\e^{s_h T_h}} < \infty$.

        \item For every $p \ge 1$, if $\norm{U}_p \le B_p < \infty$, then $\norm{T_h}_p \le (1 + B_p)^h$.
    \end{enumerate}
    The conclusion in (2) also holds when the nodes have independent, possibly different offspring counts $U_\omega$ satisfying $\norm{U_\omega}_p \le B_p$ uniformly in $\omega$.
\end{lemma}
\begin{proof}
    To show this lemma, observe that
    \begin{align*}
        \forall h \ge 1, \quad T_h \diseq 1 + \sum_{i = 1}^{U} T_{h - 1}^{(i)}
    \end{align*}
    where $T_{h - 1}^{(i)}$ is an independent copy of $T_{h - 1}$ and is independent of $U$.
    To see this recursion, the root contributes itself.
    For each of its $U$ children, it is an independent subtree of height $h - 1$.

    For simplicity, we denote by $M_h(s) \defeq \E{\e^{s T_h}}$ the moment generating function of $T_h$ at $s$.
    By the recursion of $T_h$ and the independence of random variables $T_{h - 1}^{(i)}$,
    \begin{align*}
        M_h(s) &= \e^s \E{\E{\prod_{i = 1}^U \e^{s T_{h - 1}^{(i)}} \mid U}} \\
        &= \e^s \E{M_{h - 1}(s)^U} \\
        &= \e^s \E{\exp\ab(U \ln M_{h - 1}(s))}.
    \end{align*}
    We prove the first argument by induction on $h$.
    \begin{itemize}
        \item \textbf{Boundary case: $h = 0$.}
        It holds that $T_0 = N_0 = 1$.
        Direct calculation on $M_0(s)$ tells that $M_0(s) = \e^s < \infty$ for every $s > 0$.

        \item \textbf{Inductive case: $h > 0$.}
        Suppose that $M_{h - 1}(\eta) < \infty$ for some $\eta > 0$.
        For every $s \in [0, \eta]$, $\e^{s T_{h - 1}} \le \e^{\eta T_{h - 1}}$.
        Dominated convergence provides that as $s \downarrow 0$, $M_{h - 1}(s) \to 1$ and consequently, $\ln M_{h - 1}(s) \to 0$.
        Choose $s_h$ sufficiently small to make
        \begin{align*}
            \ln M_{h - 1}(s_h) \le t.
        \end{align*}
        Then by assumption $\E{\e^{t U}} < \infty$,
        \begin{align*}
            M_h(s_h) = \e^{s_h} \E{\exp\ab(U \ln M_{h - 1}(s_h))} \le \e^{s_h} \E{\e^{t U}} < \infty.
        \end{align*}
        We note here that the choice of $s_h$ only depends on $h$ and the offspring distribution.
    \end{itemize}

    For the bound of the $\ell_p$-norm of $T_h$, by convexity of $x \mapsto x^p$,
    \begin{align*}
        \forall x_1, \ldots, x_m \ge 0, \quad \ab(\frac{1}{m} \sum_{i = 1}^m x_i)^{p} \le \frac{1}{m} \sum_{i = 1}^m x_i^p.
    \end{align*}
    Conditional on $U = m \ge 1$,
    \begin{align*}
        \E{\ab(\sum_{i = 1}^{U} T_{h - 1}^{(i)})^p \mid U = m} \le m^{p - 1} \sum_{i = 1}^m \E{\ab(T_{h - 1}^{(i)})^p} = m^p \E{T_{h - 1}^p}.
    \end{align*}
    When $m = 0$, both sides are zero.
    By law of total expectation,
    \begin{align*}
        \E{\ab(\sum_{i = 1}^{U} T_{h - 1}^{(i)})^p} \le \E{U^{p}} \E{T_{h - 1}^{p}}
    \end{align*}
    which is equivalent to
    \begin{align*}
        \norm{\sum_{i = 1}^U T_{h - 1}^{(i)}}_p \le \norm{U}_p \norm{T_{h - 1}}_p.
    \end{align*}
    The Minkowski's inequality provides that
    \begin{align*}
        \forall h \ge 1, \quad \norm{T_h}_p \le 1 + B_p \norm{T_{h - 1}}_p
    \end{align*}
    with the boundary condition $\norm{T_0}_p = 1$.
    Solving this recursion and using the binomial theorem, we conclude that $\norm{T_h}_p \le \sum_{i = 0}^h B_p^i \le (B_p + 1)^h$.
    For independent offspring counts with different laws, apply the same argument inductively with a uniform upper bound on the $\ell_p$-norms of all depth-$(h - 1)$ descendant populations.  The same recursion, with boundary value $1$, gives the asserted bound.
\end{proof}

The next step is replacing high-degree marks by independent block marks to simplify the analysis.
Set
\begin{align*}
    L_{\block} = L_{\block}(\alpha) \defeq \left\lceil \max\;\set{3r, \frac{4\ln{4}}{\ln(1/\alpha)}} \right\rceil.
\end{align*}
It is clear to see $L_{\block} = O_{\eta}((\ln{d})^2)$.
We partition the original tree into blocks containing levels $0, \ldots, L_{\block} - 1$ relative to each block root.
The depth-$L_{\block}$ frontier of each block consists of the roots of its child blocks.
Conditional on the existence of the root, a block reveals exactly the offspring of its internal nodes.

For a block, define $\BlockMark$ as the event that it contains a node in $\scrH$.
Set $\delta = \Pr{\BlockMark}$ for convenience.
Independence of a node's offspring from its existence in its generation provides that
\begin{align*}
    p_{\scrH} \le \delta \le \sum_{j = 0}^{L_{\block} - 1} (d + 1)^{j} p_{\scrH} \le \exp\ab(-C_{\tail} d + O_{\eta}\ab((\ln{d})^3)).
\end{align*}

Let $S$ be the number of internal nodes plus the frontier size of a block.
Independently sample $S^+$ with the conditional law $S \midl \BlockMark$ and $S^-$ as $S \midl \overline{\BlockMark}$ and set $Z \defeq S^+ + S^-$.
Moreover, to choose the actual block data, independently pick a $\Bernoulli(\delta)$ selector.
Recursively place the selected child blocks in the first available slots of a node with $Z$ children and fill extra slots independently by the same construction.
This process embeds the actual block tree into a Galton-Watson tree with offspring $Z$ and independent $\Bernoulli(\delta)$ marks.
\begin{proposition} \label{prop:moments-block-random-variable}
    For the random variable $Z$, there exists some $t > 0$ making $\E{\e^{t Z}} < \infty$.
    Moreover, for every integer $p \ge 2$,
    \begin{align*}
        \frac{\ln \E{(1 + Z)^p}}{p} \le C_\eta L_{\block} \ln{(d + p)}
    \end{align*}
    where $C_\eta$ depends only on $\eta$.
\end{proposition}
\begin{proof}
    Let $S_h$ be the total population through depth $h$ and set $M_h(t) \defeq \E{\e^{t S_h}}$.
    The branching recursion provides that
    \begin{align*}
        \forall h \ge 1, \quad M_h(t) = \exp\ab(t + (d + 1)(M_{h - 1}(t) - 1))
    \end{align*}
    with boundary condition $M_0(t) = \e^t$.
    For every $t > 0$ and $h \ge 0$, this recursion implies that $M_h(t) < \infty$.

    For $\BlockMark$, it contains the event that the root of this block is in $\scrH$ while its complement contains the event that the root has no children.
    Therefore,
    \begin{align*}
        \E{\e^{t S^+}} = \frac{\E{\e^{tS} \id{\BlockMark}}}{\delta} < \infty, \quad \E{\e^{t S^-}} = \frac{\E{\e^{tS} \id{\overline\BlockMark}}}{1 - \delta} < \infty.
    \end{align*}
    By independence of $S^+$ and $S^-$, $\E{\e^{tZ}} = \mathbb E\e^{tS^+} \mathbb E\e^{tS^-} < \infty$.

    Define a size-biased random variable $K^*$ as
    \begin{align*}
        \Pr{K^* = k} = \frac{k \Pr{K = k}}{d + 1}
    \end{align*}
    and therefore it follows the law $1 + \Poisson(d + 1)$.
    Moreover, let $K^+$ be the conditional random variable with law $K \midl K > D - 3$.
    By \Cref{lem:Poisson-inequality},
    \begin{align*}
        \norm{K}_p, \norm{K^*}_p, \norm{K^+}_p \le C_{\moment} (d + p).
    \end{align*}
    Choose $b_p = C(\eta) (d + p)$ such that
    \begin{align} \label{eq:block-tree-moment-upper-bound}
        b_p \ge \max\;\set{3, d + 1, D + 1, 1 + \max\;\set{\norm{K}_p, \norm{K^*}_p, \norm{K^+}_p}}.
    \end{align}

    Observe that $\BlockMark$ requires at least one high-degree node in the block:
    \begin{align*}
        \id{\BlockMark} \le \sum_{j = 0}^{L_{\block} - 1} \sum_{\depth(\omega) = j} \id{K_{\omega} > D - 3}.
    \end{align*}
    Hence
    \begin{align} \label{eq:expected-block-event-inequality}
        \E{S^p \id{\BlockMark}} \le \sum_{j = 0}^{L_{\block} - 1} \E{\sum_{\depth(\omega) = j} S^p \id{K_{\omega} > D - 3}}.
    \end{align}
    Recall the law of the size-biased random variable $K^*$:
    \begin{align*}
        \Pr{K^* = k} = \frac{k \Pr{K = k}}{d + 1}.
    \end{align*}
    Hence we obtain that for every level $j$,
    \begin{align*}
        \E{\sum_{\depth(\omega) = j} S^p \id{K_{\omega} > D - 3}} = (d + 1)^j p_{\scrH}\E{\widetilde{S}_j^p}
    \end{align*}
    where $\widetilde{S}_j$ is the total population through $L_{\block}$ in the following modified tree.
    \begin{itemize}
        \item The $j$ ancestors on a distinguished path from root to a node of depth $j$ have offspring law $K^*$.
        \item At each such ancestor, one uniform child continues the path.
        \item The endpoint has offspring law $K^+$.
        \item Every other vertex has ordinary offspring law $K$.
    \end{itemize}

    Relabeling the children so that the distinguished child is always first does not change the total population.  After this relabeling, the offspring counts in the modified tree are independent, with laws $K^*$, $K^+$ or $K$ at the prescribed nodes, and all satisfy \eqref{eq:block-tree-moment-upper-bound}.
    By \Cref{lem:population-moment},
    \begin{align*}
        \E{\widetilde{S}_j^p} \le b_p^{p L_{\block}}.
    \end{align*}
    Plugging above arguments into \eqref{eq:expected-block-event-inequality}, we have that
    \begin{align*}
        \E{S^p \id{\BlockMark}} \le p_{\scrH} \sum_{j = 0}^{L_{\block} - 1} (d + 1)^j b_p^{p L_{\block}} \le p_{\scrH} (b_p^p (d + 1))^{L_{\block}}
    \end{align*}
    Since $\delta \ge p_{\scrH}$, for $S^+$, it holds that
    \begin{align*}
        \E{(S^+)^p} = \frac{\E{S^p \id{\BlockMark}}}{\delta} \le L_\block (d + 1)^{L_\block} b_p^{p L_\block}.
    \end{align*}
    Together with the choice of $b_p$, $\norm{S^+}_p \le b_p^{2L_\block}$.
    For $S^-$, since every internal node has at most $D - 3$ children,
    \begin{align*}
        S^- \le \sum_{j = 0}^{L_{\block}} (D - 3)^j \le (1 + D - 3)^{L_{\block}} \le b_p^{L_\block}
    \end{align*}
    and thus $\norm{S^-}_p \le b_p^{L_\block}$.
    By Minkowski's inequality,
    \begin{align*}
        \norm{1 + Z}_p \le 1 + \norm{S^+}_p + \norm{S^-}_p \le b_p^{3L_\block}.
    \end{align*}
    Hence we conclude that $\ln \E{(1 + Z)^p} \le C_{\eta} p L_{\block} \ln(d + p)$, which completes the proof.
\end{proof}

A block is said to be bad if it is a marked block itself or there is a marked block within two block distance of it, and otherwise it is good.
By $L_{\block} \ge 3r$, every bad node in the original tree belongs to a bad block.
Set
\begin{align*}
    \beta \defeq \alpha^{L_{\block}}, \quad a \defeq \frac{2(1 + D^R)}{1 - \alpha}.
\end{align*}
Note that a good block has internal contribution at most $a$ and transmits at most $\beta$ times the largest frontier value, while the influence of a bad block is bounded by summing all internal costs and frontier influences.
Therefore, the message $\X$ is dominated by a message on the block tree equipped with $\cost(\omega) = a(1 + Z_{\omega})$ passed by the good rule $\beta \max$ and the bad rule sum.
Let $\X^{\block}$ be the message of the root on the block tree, and let $\Y^{\block}$ be the message with an additional marked parent outside the tree.
Therefore
\begin{align*}
    \X \le \X^{\block}, \quad \Y \le \Y^{\block}.
\end{align*}
We provide the following proposition to $\X^{\block}$ and $\Y^{\block}$.
\begin{proposition} \label{prop:block-message-bound}
    Assume all definitions and notations above.
    Let $m = \E{Z}$, $N_h$ be the generation-$h$ population, $C_h \defeq \sum_{\depth(\omega) < h} a(1 + Z_\omega)$ and $b_h \defeq \E{N_h \cdot \id{\mathcal E}}$.
    Let $\mathcal E$ be the event that there is a mark at depth zero, one or two.
    Then
    \begin{align*}
        p_{\mathcal E} \defeq \Pr{\mathcal E} \le \delta(1 + m + m^2).
    \end{align*}
    Suppose that there is an integer $P \ge 2$ satisfying that
    \begin{align}
        \beta m^{1/P} &\le \frac{1}{4}, \label{eq:block-message-bound-moment-1} \\
        p_{\mathcal E}^{1/(2P)} \norm{N_3}_{2P} &\le \frac{1}{4}, \label{eq:block-message-bound-moment-2} \\
        p_{\mathcal E}^{1/(2P)} \norm{N_2}_P \norm{N_3}_{2P} & \le \frac{1}{16}. \label{eq:block-message-bound-moment-3}
    \end{align}
    Then the matrix
    \begin{align*}
        B_P \defeq \begin{pmatrix}
            \beta^{P + 1} m & b_3 \\
            m^2 & b_3
        \end{pmatrix}
    \end{align*}
    has spectral radius less than $1$ and for some $s > 0$, $\E{\exp(s \X^{\block})} < \infty$ and $\E{\exp(s \Y^{\block})} < \infty$.
\end{proposition}
\begin{proof}
    To show $p_{\mathcal E} \le \delta(1 + m + m^2)$, note that conditional on the tree, the probability of $\mathcal E$ is at most $\delta(N_0 + N_1 + N_2)$.
    Taking expectations provides the desired upper bound.

    For $h \ge 1$, let $X_h$ and $Y_h$ denote the messages restricted to depths $0, \ldots, h - 1$ with zero frontier values. For notational convenience, set
    \begin{align*}
        X_h = Y_h = 0 \qquad \text{for every integer $h \le 0$}.
    \end{align*}
    Write
    \begin{align*}
        M_h^X(s)&\defeq \E{\e^{sX_h}}, &
        M_h^Y(s)&\defeq \E{\e^{sY_h}}.
    \end{align*}
    Consequently, $M_h^X(s) = M_h^Y(s) = 1$ for $h \le 0$, and the non-negativity of all costs and update rules gives
    \begin{align*}
        X_h \uparrow \X^{\block}, \qquad
        Y_h \uparrow \Y^{\block}
        \qquad \text{as $h \to \infty$}.
    \end{align*}
    We claim that $X_h$ and $Y_h$ have positive exponential moments.
    In fact, by induction, it is direct to see $X_h \le C_h$ and $Y_h \le C_h$.
    An argument similar to the proof of \Cref{lem:population-moment} shows that, for every finite $h \ge 0$, there exists some $s_h>0$ such that $\E{\exp(s_hX_h)}$ and $\E{\exp(s_hY_h)}$ are finite.

    Work on one realization of the original marked block tree.
    For a node $w$, let $X_j^{(w)}$ be the depth-$j$ message computed using only the marked descendant tree rooted at $w$, and let $Y_j^{(w)}$ be the analogous message with one additional marked parent attached above $w$.
    As above, both messages are zero for $j\le 0$.

    For $j\ge 1$, let $\mathcal F_j$ be the sigma-field generated by the offspring counts and marks of the nodes at depths strictly less than $j$.
    This information determines generation $j$, while, conditional on $\mathcal F_j$, the marked descendant trees rooted at its nodes are independent copies of the original marked block tree.
    In particular, $N_j$ and $C_j$ are $\mathcal F_j$-measurable, and $\mathcal E\in\mathcal F_3$.

    Put $A=\overline{\mathcal E}$.
    On $A$, the root is good and no mark outside a first-generation descendant tree can affect a node in that tree through distance two.
    Hence, point-wise,
    \begin{align} \label{eq:block-message-good-event-X}
        X_h=a(1+Z)+\beta\max_{\depth(w)=1}X_{h-1}^{(w)}.
    \end{align}
    For $Y_h$, the additional marked parent makes the root and every first-generation node bad.
    After expanding these two generations, the event $A$ excludes all remaining external marks that could affect a second-generation descendant tree.
    Therefore
    \begin{align} \label{eq:block-message-good-event-Y}
        Y_h
        \le C_2+\sum_{\depth(w)=2}X_{h-2}^{(w)}
        \qquad\text{on }A.
    \end{align}
    The inequality also covers $h=1$, when $C_2$ contains unused non-negative costs.

    On $\mathcal E$, replace every update in the first three generations by the sum rule.
    Adding a marked parent above each generation-three descendant tree covers every possible influence from marks outside that tree.
    Since replacing a $\beta\max$ update by a sum can only increase a non-negative message, point-wise
    \begin{align} \label{eq:block-message-bad-event}
        \max\;\set{X_h,Y_h}
        \le C_3+\sum_{\depth(w)=3}Y_{h-3}^{(w)}
        \qquad\text{on }\mathcal E.
    \end{align}

    We now take expectations on the original probability space.
    In the contribution from $A$, first drop its indicator on the right-hand sides of \eqref{eq:block-message-good-event-X} and \eqref{eq:block-message-good-event-Y}, using non-negativity of every Taylor coefficient.
    Then condition on $\mathcal F_1$ for the maximum in the $X$ bound and on $\mathcal F_2$ for the sum in the $Y$ bound.
    For non-negative $x_1,\ldots,x_N$,
    \begin{align*}
        \e^{s\beta\max_i x_i}
        \preceq 1+\sum_{i=1}^N\ab(\e^{s\beta x_i}-1)
    \end{align*}
    coefficient-wise.
    In the contribution from $\mathcal E$, keep its indicator and condition on $\mathcal F_3$ in \eqref{eq:block-message-bad-event}.
    The generation-three descendant trees then have their original independent laws.
    Consequently, coefficient-wise,
    \begin{align*}
        M_h^X(s)
        &\preceq
        \E{\e^{sa(1+Z)}
        \left(1+Z\left(M_{h-1}^X(\beta s)-1\right)\right)}
        +\E{\id{\mathcal E}\e^{sC_3}
        \left(M_{h-3}^Y(s)\right)^{N_3}},\\
        M_h^Y(s)
        &\preceq
        \E{\e^{sC_2}\left(M_{h-2}^X(s)\right)^{N_2}}
        +\E{\id{\mathcal E}\e^{sC_3}
        \left(M_{h-3}^Y(s)\right)^{N_3}}.
    \end{align*}

    For every finite integer $M \ge 0$, define
    \begin{align*}
        H_M \defeq \max_{h = 0, \ldots, M} \norm{X_h}_P, \quad K_M \defeq \max_{h = 0, \ldots, M} \norm{Y_h}_P.
    \end{align*}
    We upper bound $H_M$ and $K_M$ by the above decomposition.
    It is clear to see
    \begin{align*}
        \E{\ab(\max_{i = 1, \ldots, Z}\; X_{h - 1}^{(i)})^P} \le \E{\sum_{i = 1}^Z \ab(X_{h - 1}^{(i)})^P} \le m H_M^P.
    \end{align*}
    For the sum rule, convexity provides that
    \begin{align*}
        \E{\ab(\sum_{i = 1}^{N} W_i)^P} \le N^P \max_{i = 1, \ldots, N} \E{W_i^P}.
    \end{align*}
    Together with the decomposition, it tells that
    \begin{align*}
        \norm{\id{\mathcal{E}} \sum_{i = 1}^{N_3} Y_{h - 3}^{(i)}}_P^P &= \E{\id{\mathcal{E}} \E{\ab(\sum_{i = 1}^{N_3} Y_{h - 3}^{(i)})^P \mid \text{revealed depth-three nodes}}} \\
        &\le K_M^P \E{\id{\mathcal{E}} N_3^P} \\
        &= K_M^P \norm{\id{\mathcal{E}} N_3}_P^P.
    \end{align*}
    By Cauchy-Schwarz inequality,
    \begin{align*}
        \norm{\id{\mathcal{E}} N_3}_P \le p_{\mathcal{E}}^{1/(2P)} \E{N_3^{2P}}^{1/(2P)}.
    \end{align*}
    Similarly,
    \begin{align*}
        \norm{\sum_{i = 1}^{N_2} X_{h - 2}^{(i)}}_P \le H_M \norm{N_2}_P.
    \end{align*}
    Minkowski's inequality gives that
    \begin{align*}
        H_M &\le c_0 + \beta m^{1/P} H_M + K_M \cdot p_{\mathcal{E}}^{1/(2P)} \norm{N_3}_{2P}, \\
        K_M &\le c_1 + \norm{N_2}_P H_M + K_M \cdot p_{\mathcal{E}}^{1/(2P)} \norm{N_3}_{2P}
    \end{align*}
    where $c_0 = \norm{a(1 + Z)}_P + \norm{C_3}_P$ and $c_1 = \norm{C_2}_P + \norm{C_3}_P$ are both finite positive terms only depending on fixed-height weights.
    By assumption,
    \begin{align*}
        K_M \le \frac{4}{3}\ab(c_1 + \norm{N_2}_P H_M) \quad \text{and thus} \quad H_M \le c_0 + \frac{4}{3}c_1 \cdot p_{\mathcal{E}}^{1/(2P)} \norm{N_3}_{2P} + \frac{1}{3} H_M
    \end{align*}
    meaning
    \begin{align*}
        \sup_{M \ge 0} H_M < \infty, \quad \sup_{M \ge 0} K_M < \infty.
    \end{align*}
    By monotonicity of $\ell_p$-norm, all moments mentioned in this proof of orders $0, \ldots, P$ are also uniformly bounded.

    For higher moments, choose polynomials $f_0, g_0$ of degree at most $P$ with constant coefficient $1$ and non-negative coefficients satisfying
    \begin{align*}
        [s^i] f_0(s) \ge \frac{\E{X_h^i}}{i!}, \quad [s^i] g_0(s) \ge \frac{\E{Y_h^i}}{i!}
    \end{align*}
    for every $h \ge 0$.
    The existence of $f_0$ and $g_0$ comes from the uniform upper bound for $H_M$ and $K_M$.
    Our aim is to find two convergent series $u, v$ with non-negative coefficients and no terms of degree at most $P$ such that
    \begin{align*}
        \E{\e^{sX_h}} \preceq f_0(s) + u(s), \quad \E{\e^{sY_h}} \preceq g_0(s) + v(s)
    \end{align*}
    where $\preceq$ means coefficient-wise comparison.

    We first state two useful observations.
    \begin{itemize}
        \item For every $x_1, \ldots, x_N \ge 0$,
        \begin{align*}
            \e^{s \beta \max_{i = 1}^N x_i} \preceq 1 + \sum_{i = 1}^{N} \ab(\e^{s \beta x_i} - 1).
        \end{align*}

        \item Since $u$ and $v$ start at degree $P + 1$,
        \begin{align*}
            u(\beta s) \preceq \beta^{P + 1} u(s), \quad v(\beta s) \preceq \beta^{P + 1} v(s).
        \end{align*}
    \end{itemize}
    For scalar variables $s, u$ and $v$, define
    \begin{align*}
        \Psi_1(s, u, v) &\defeq \E{\e^{sa(1 + Z)} \ab(1 + Z\ab(f_0(\beta s) - 1 + \beta^{P + 1}u))} + \E{\id{\mathcal{E}} \e^{sC_3} (g_0(s) + v)^{N_3}}, \\
        \Psi_2(s, u, v) &\defeq \E{\e^{sC_2} \ab(f_0(s) + u)^{N_2}} + \E{\id{\mathcal{E}} \e^{sC_3} (g_0(s) + v)^{N_3}}.
    \end{align*}
    We first verify that $\Psi_1, \Psi_2$ are analytic around the origin $(0, 0, 0)$.
    By \Cref{prop:moments-block-random-variable}, there exists some $t_0 > 0$ such that $\E{\e^{t_0 Z}} < \infty$.
    Consequently, the population through any fixed depth has a positive exponential moment.
    Indeed, for every $h \ge 0$, define
    \begin{align*}
        M_h(t) \defeq \E{\exp\ab(t\ab(N_0 + \cdots + N_h))}.
    \end{align*}
    The branching recursion for the Galton--Watson tree provides that
    \begin{align*}
        M_0(t) = \e^t, \quad M_h(t) = \e^t \E{M_{h - 1}(t)^Z}.
    \end{align*}
    If $M_{h - 1}$ is finite near $0$, then $M_{h - 1}(t) \to 1$ as $t \downarrow 0$.
    We pick $t$ sufficiently small such that $\ln{M_{h - 1}(t)} \le t_0$ which makes $M_h(t)$ finite.
    It is enough to pick $t$ depending only on $h$ and the offspring distribution $Z$.
    Specifically, let $T \defeq N_0 + N_1 + N_2 + N_3$ be the population from depth zero to three.
    There exists a constant $\theta > 0$ such that
    \begin{align*}
        M_3(\theta) = \E{\e^{\theta T}} < \infty.
    \end{align*}
    Observe that $C_2 = a(N_0 + 2N_1 + N_2)$ and $C_3 = a(N_0 + 2N_1 + 2N_2 + N_3)$.
    Then it is clear that $a(1 + Z) \le aT$ and $C_2, C_3 \le 2aT$.
    Choose $L \ge 1$ such that for complex $\abs{s} \le 1$,
    \begin{align*}
        \abs{f_0(s) - 1} \le L\abs{s}, \qquad \abs{g_0(s) - 1} \le L\abs{s}.
    \end{align*}
    Pick the radius $R$ as
    \begin{align*}
        R \defeq \min\;\set{\frac{1}{2}, \frac{\theta}{2(2a + L + 1)}}.
    \end{align*}
    Write
    \begin{align*}
        G_1(s, u, v) &\defeq \e^{sa(1 + Z)} \ab(1 + Z\ab(f_0(\beta s) - 1 + \beta^{P + 1}u)), \\
        G_2(s, u, v) &\defeq \e^{s C_2} \ab(f_0(s) + u)^{N_2}, \\
        G_3(s, u, v) &\defeq \id{\mathcal{E}} \e^{s C_3} \ab(g_0(s) + v)^{N_3}
    \end{align*}
    and thus $\Psi_1 = \E{G_1 + G_3}$ and $\Psi_2 = \E{G_2 + G_3}$.
    We show that for every $s, u, v \in \mathbb{C}$ such that $\abs{s}, \abs{u}, \abs{v} \le R < 1$, $G_1, G_2$ and $G_3$ are upper bounded by $\e^{(2a + L + 1)RT}$.
    
    For $G_1(s, u, v)$, since $\beta \in (0, 1)$ and $\abs{s}, \abs{u} \le R \le 1$,
    \begin{align*}
        \abs{f_0(\beta s) - 1 + \beta^{P + 1}u} &\le \abs{f_0(\beta s) - 1} + \beta^{P + 1} \abs{u} \le L\abs{\beta s} + \abs{u} \le (L + 1)R
    \end{align*}
    By the inequality $1 + x \le \e^x$ on $[0, \infty)$,
    \begin{align*}
        \abs{G_1(s, u, v)} \le \e^{aR(1 + Z)} \ab(1 + Z(L + 1) R) \le \e^{(a + L + 1)R(1+Z)} \le \e^{(2a + L + 1) RT}.
    \end{align*}

    For $G_2(s, u, v)$, by triangle inequality,
    \begin{align*}
        \abs{f_0(s) + u} \le 1 + \abs{f_0(s) - 1} + \abs{u} \le 1 + (L + 1)R \le \e^{(L + 1)R}.
    \end{align*}
    Therefore,
    \begin{align*}
        \abs{G_2(s, u, v)} &\le \e^{RC_2} \abs{f_0(s) + u}^{N_2} \le \e^{2aRT} \e^{(L + 1)R N_2} \le \e^{(2a + L + 1) RT}.
    \end{align*}
    A similar argument also holds for $G_3(s, u, v)$ and therefore $\abs{G_3(s, u, v)} \le \e^{(2a + L + 1)RT}$.

    Now we are ready to show analyticity of $\Psi_1$ and $\Psi_2$.
    Fix any one of $G_1, G_2$ and $G_3$ and denote it by $G$.
    Observe that for every realization of the tree $\omega$, $G$ is an entire function consisting of the product of the exponential and polynomials.
    Expand its Taylor series:
    \begin{align*}
        G(\omega; s, u, v) = \sum_{i, j, k \ge 0} c_{i, j, k}(\omega) s^i u^j v^k.
    \end{align*}
    Cauchy's coefficient formula gives
    \begin{align*}
        c_{i, j, k}(\omega) = \frac{1}{(2\pi i)^3} \int_{\abs{\zeta_1} = R} \int_{\abs{\zeta_2} = R} \int_{\abs{\zeta_3} = R} \frac{G(\omega; \zeta_1, \zeta_2, \zeta_3)}{\zeta_1^{i + 1} \zeta_2^{j + 1} \zeta_3^{k + 1}} \d{\zeta_3}\d{\zeta_2}\d{\zeta_1}.
    \end{align*}
    The above modular bound for $G$ shows that
    \begin{align*}
        \abs{c_{i, j, k}(\omega)} \le \e^{\theta T(\omega)/2} R^{-(i + j + k)}
    \end{align*}
    and each coefficient is integrable.
    Therefore,
    \begin{align*}
        \E{G(s, u, v)} = \sum_{i, j, k \ge 0} \E{c_{i, j, k}} s^i u^j v^k.
    \end{align*}
    This power series converges absolutely and uniformly over every smaller polydisc
    \[
        \set{(s, u, v) \midc \abs{s} \le r, \abs{u} \le r, \abs{v} \le r} \qquad 0 \le r < R.
    \]
    Therefore $\E{G}$ is analytic around the origin.
    This leads to the claim that $\Psi_1$ and $\Psi_2$ are analytic around $(0, 0, 0)$.
    Moreover, since $f_0(\beta s) - 1$ has non-negative coefficients, every multivariate Taylor coefficient of $\Psi_1$ and $\Psi_2$ is non-negative.

    For a power series $H$, let $\mathsf{T}_P$ be the Taylor truncation of $H$ through degree $P$:
    \begin{align*}
        \mathsf{T}_P H(s) = \sum_{j = 0}^P [s^j] H(s) s^j.
    \end{align*}
    Define
    \begin{align*}
        F(s, u, v) \defeq \begin{pmatrix}
            \Psi_1(s, u, v) \\
            \Psi_2(s, u, v)
        \end{pmatrix} - \mathsf{T}_P\begin{pmatrix}
            \Psi_1(s, 0, 0) \\
            \Psi_2(s, 0, 0)
        \end{pmatrix}\;.
    \end{align*}
    Differentiate $F$ on $(u, v)$ at the origin, and we obtain that
    \begin{align*}
        D_{(u, v)} F|_{(0, 0, 0)} = B_P = \begin{pmatrix}
            \beta^{P + 1}m & b_3 \\
            m^2 & b_3
        \end{pmatrix}\;
    \end{align*}
    By monotonicity of $\ell_p$-norms and Cauchy-Schwarz inequality,
    \begin{align*}
        b_3 = \norm{N_3 \cdot \id{\mathcal{E}}}_1 \le \norm{N_3 \cdot \id{\mathcal{E}}}_P \le p_{\mathcal{E}}^{1/(2P)} \norm{N_3}_{2P} \le \frac{1}{4}.
    \end{align*}
    For $z = (x,y)^{\mathsf{T}} \in \mathbb{C}^2$, define the weighted norm
    \begin{align*}
        \norm{z}_w = \max\;\set{\abs{x}, \frac{\abs{y}}{1 + 4m^2}}.
    \end{align*}
    The operator norm of $B_P$ is
    \begin{align*}
        \norm{B_P}_w = \max\;\set{\beta^{P + 1}m + b_3(1 + 4m^2), \frac{m^2}{1 + 4m^2} + b_3}.
    \end{align*}
    By assumption, $\norm{B_P}_w \le \frac{9}{16} < 1$.
    Since $F$ is analytic and its derivative is continuous, we can choose $r_0, \eps > 0$ such that $F$ is analytic on a neighborhood of
    $\set{(s, z) \midc \abs{s} \le r_0, \norm{z}_w \le \eps}$
    and $\norm{D_z F(s, z)}_{w} \le \frac{3}{4}$
    through this set.
    Integrating the derivative along the segment between $z$ and $z'$ gives that
    \begin{align*}
        \norm{F(s, z) - F(s, z')}_{w} \le \frac{3}{4} \norm{z - z'}_w.
    \end{align*}
    On the other hand, $F(s, 0) = O(s^{P + 1})$.
    Thus after reducing $r_0$ if necessary, there is a constant $K \ge 1$ such that
    \begin{align*}
        \norm{F(s, 0)}_w \le K \abs{s}^{P + 1}, \qquad \forall \abs{s} \le r_0.
    \end{align*}
    Choose $0 < r \le r_0$ small such that $K r^{P + 1} \le \eps/4$.
    Therefore, for every $\abs{s} \le r$ and $\norm{z}_w \le \eps$,
    \begin{align*}
        \norm{F(s, z)}_{w} &\le \norm{F(s, z) - F(s, 0)}_w + \norm{F(s, 0)}_w \\ 
        &\le \frac34\|z\|_w+K|s|^{P+1} \le  \frac34\varepsilon+\frac14\varepsilon
        =\eps.
    \end{align*}
    Hence, for every $\abs{s} \le r$, $F(s, \cdot)$ maps a closed ball into itself and is a contraction.
    Define
    \[
        U_0(s) = (0, 0), \qquad U_{k + 1}(s) = F(s, U_k(s)).
    \]
    By Banach's fixed-point theorem, the iteration goes to a limit $(u, v)$.
    Uniform convergence on compact subsets preserves analyticity.
    Therefore, $u, v$ are analytic on $\abs{s} < r$.

    Now we verify the required properties of their Taylor coefficients.
    Write
    \begin{align*}
        \Psi_i(s, u, v) = \sum_{j, k, \ell \ge 0} c_{j, k, \ell}^{(i)} s^j u^k v^{\ell}, \quad c_{j, k, \ell}^{(i)} \ge 0.
    \end{align*}
    Therefore
    \begin{align*}
        F_i(s, u, v) = \sum_{j \ge P + 1} c_{j, 0, 0}^{(i)} s^j + \sum_{j, k, \ell \ge 0, k + \ell \ge 1} c_{j, k, \ell}^{(i)} s^j u^k v^{\ell}.
    \end{align*}
    Starting from $U_0 = (0, 0)$, induction shows that every component of $U_k$ has non-negative coefficients and vanishes through degree $P$.
    These properties pass to the limit coefficient by coefficient.
    In fact, for any $0 < t < r$, Cauchy's coefficient formula and uniform convergence on $\abs{s} = t$ provides that
    \begin{align*}
        [s^n] U(s) = \frac{1}{2\pi i} \int_{\abs{s} = t} \frac{U(s)}{s^{n + 1}} \d{s} = \lim_{k \to \infty} [s^n]U_k(s).
    \end{align*}
    Thus
    \begin{align*}
        u(s) = \sum_{n = P + 1}^{\infty} u_n s^n, \qquad v(s) = \sum_{n = P + 1}^{\infty} v_n s^n
    \end{align*}
    and both series converge on $\abs{s} < r$.

    Since $u(s)$ and $v(s)$ vanish through degree $P$, every monomial
    in $\Psi_i(s,u(s),v(s))$ containing at least one factor of $u(s)$
    or $v(s)$ has degree at least $P+1$ in $s$.  Therefore,
    \begin{equation}\label{eq:Psi-low-degree-independence}
        \mathsf T_P\Psi_i(s,u(s),v(s))
        =\mathsf T_P\Psi_i(s,0,0),
        \qquad i=1,2.
    \end{equation}
    Consequently, the fixed-point equation and
    \eqref{eq:Psi-low-degree-independence} imply
    \begin{align}
        u(s)=(I-\mathsf T_P)\Psi_1(s,u(s),v(s)),
        \qquad
        v(s)=(I-\mathsf T_P)\Psi_2(s,u(s),v(s)).
        \label{eq:Psi-high-degree-fixed-point}
    \end{align}

    We now prove by strong induction on $h\in\mathbb Z$ that
    \begin{align}
        M_h^X(s)\preceq f_0(s)+u(s), \qquad
        M_h^Y(s)\preceq g_0(s)+v(s).
        \label{eq:mgf-majorant}
    \end{align}
    For $h\le0$, both moment generating functions equal $1$, so the
    claim follows from the non-negativity of the coefficients of
    $f_0,g_0,u$, and $v$.

    Fix $h\ge1$ and suppose that the claim holds for every integer
    $k<h$.  Using the coefficient-wise recursions above, the induction
    hypothesis, and
    $u(\beta s)\preceq\beta^{P+1}u(s)$, we obtain
    \begin{align*}
        M_h^X(s)\preceq\Psi_1(s,u(s),v(s)),\qquad
        M_h^Y(s)\preceq\Psi_2(s,u(s),v(s)).
    \end{align*}
    By the definition of $f_0$ and $g_0$,
    \begin{align*}
        \mathsf T_P M_h^X(s)\preceq f_0(s),\qquad
        \mathsf T_P M_h^Y(s)\preceq g_0(s).
    \end{align*}
    For the higher-order coefficients,
    \eqref{eq:Psi-high-degree-fixed-point} gives
    \begin{align*}
        (I-\mathsf T_P)M_h^X(s)
        &\preceq
        (I-\mathsf T_P)\Psi_1(s,u(s),v(s))=u(s),\\
        (I-\mathsf T_P)M_h^Y(s)
        &\preceq
        (I-\mathsf T_P)\Psi_2(s,u(s),v(s))=v(s).
    \end{align*}
    Adding the low- and high-degree parts proves
    \eqref{eq:mgf-majorant}.

    For every fixed $s\in(0,r)$, the two majorants are finite.
    Since $X_h\uparrow\X^{\block}$ and
    $Y_h\uparrow\Y^{\block}$, the monotone convergence theorem gives
    \begin{align*}
        \E{\e^{s\X^{\block}}}
        &=\lim_{h\to\infty}M_h^X(s)
        \le f_0(s)+u(s)<\infty,\\
        \E{\e^{s\Y^{\block}}}
        &=\lim_{h\to\infty}M_h^Y(s)
        \le g_0(s)+v(s)<\infty. \qedhere
    \end{align*}
\end{proof}

\begin{proof}[Proof of \Cref{prop:GWT-message-well-definedness-finiteness}]
    We verify \eqref{eq:block-message-bound-moment-1}, \eqref{eq:block-message-bound-moment-2} and \eqref{eq:block-message-bound-moment-3} to show that $\X^{\block}$ and $\Y^{\block}$ have finite positive exponential moments and therefore
    \begin{align*}
        \E{\e^{t_0\X}} < \infty, \quad \E{\e^{t_0\Y}} < \infty.
    \end{align*}
    for some constant $t_0 > 0$.
    Note that $\ln{m} = O_{\eta}\ab(L_{\block} \ln d)$ and by convexity,
    \begin{align*}
        \norm{N_j}_p \le \norm{Z}_p^j.
    \end{align*}
    Therefore, for fixed $j \ge 0$, $\ln\norm{N_j}_p = O_{\eta}(L_{\block} \ln(d + p))$.
    Choose $P = \lceil C_{\eta, \alpha} \ln(d + 2) \rceil$ for a sufficiently large constant.
    It is easy to verify
    \begin{align*}
        \beta m^{1/P} \le \exp\ab(-\frac{1}{2} \ln{(1/\alpha)} L_{\block}) \le \frac{1}{4}.
    \end{align*}
    Meanwhile,
    \begin{align*}
        \ln p_{\mathcal{E}} \le -c_{\eta} d + O_{\eta}(L_{\block} \ln d).
    \end{align*}
    This gives that
    \begin{align*}
        \max\,\{p_{\mathcal{E}}^{1/(2P)} \norm{N_3}_{2P}, p_{\mathcal{E}}^{1/(2P)} \norm{N_2}_P \norm{N_3}_{2P}\} \le \exp\ab( -\Omega_{\eta, \alpha}\ab(\frac{d}{\ln d}) + O_{\eta, \alpha}\ab((\ln d)^3))
    \end{align*}
    which is bounded by $1/16$ for sufficiently large $d$.
    Then by \Cref{prop:block-message-bound}, the finite positive exponential moment bound is concluded for some constant $t_0 > 0$.
\end{proof}

%% file: refs.bib
@inproceedings{AL20,
    author = {Alev, Vedat Levi and Lau, Lap Chi},
    title = {Improved analysis of higher order random walks and applications},
    year = {2020},
    isbn = {9781450369794},
    publisher = {Association for Computing Machinery},
    address = {New York, NY, USA},
    url = {https://doi.org/10.1145/3357713.3384317},
    doi = {10.1145/3357713.3384317},
    booktitle = {Proceedings of the 52nd Annual ACM SIGACT Symposium on Theory of Computing},
    pages = {1198-1211},
    numpages = {14},
    location = {Chicago, IL, USA},
    series = {STOC 2020},
}

@article{ALOG24,
    author = {Anari, Nima and Liu, Kuikui and Oveis Gharan, Shayan},
    title = {Spectral Independence in High-Dimensional Expanders and Applications to the Hardcore Model},
    journal = {SIAM Journal on Computing},
    volume = {53},
    number = {6},
    pages = {FOCS20-1-FOCS20-37},
    year = {2024},
    doi = {10.1137/20M1367696},
    URL = {https://doi.org/10.1137/20M1367696},
    eprint = {https://doi.org/10.1137/20M1367696},
}

@article{AN05,
    ISSN = {0003486X},
    URL = {http://www.jstor.org/stable/20159944},
    author = {Achlioptas, Dimitris and Naor, Assaf},
    journal = {Annals of Mathematics},
    number = {3},
    pages = {1335--1351},
    publisher = {Annals of Mathematics},
    title = {The Two Possible Values of the Chromatic Number of a Random Graph},
    urldate = {2026-10-16},
    volume = {162},
    year = {2005},
}

@inproceedings{BCCPSV22,
    author = {Blanca, Antonio and Caputo, Pietro and Chen, Zongchen and Parisi, Daniel and {\v{S}}tefankovi{\v{c}}, Daniel and Vigoda, Eric},
    title = {On Mixing of {Markov} Chains: Coupling, Spectral Independence, and Entropy Factorization},
    booktitle = {Proceedings of the 2022 Annual ACM-SIAM Symposium on Discrete Algorithms (SODA)},
    year = {2022},
    chapter = {},
    pages = {3670-3692},
    doi = {10.1137/1.9781611977073.145},
    URL = {https://epubs.siam.org/doi/abs/10.1137/1.9781611977073.145},
    eprint = {https://epubs.siam.org/doi/pdf/10.1137/1.9781611977073.145},
}

@article{BGGS24,
    author = {Bez{\'a}kov{\'a}, Ivona and Galanis, Andreas and Goldberg, Leslie Ann and {\v S}tefankovi{\v c}, Daniel},
    title = {Fast Sampling via Spectral Independence Beyond Bounded-degree Graphs},
    year = {2024},
    issue_date = {January 2024},
    publisher = {Association for Computing Machinery},
    address = {New York, NY, USA},
    volume = {20},
    number = {1},
    issn = {1549-6325},
    url = {https://doi.org/10.1145/3631354},
    doi = {10.1145/3631354},
    journal = {ACM Trans. Algorithms},
    month = Jan,
    articleno = {7},
    numpages = {26},
}

@article{BS20,
    author = {Budzynski, Louise and Semerjian, Guilhem},
    journal = {Journal of Statistical Physics},
    number = {5},
    pages = {1490--1522},
    title = {The Asymptotics of the Clustering Transition for Random Constraint Satisfaction Problems},
    volume = {181},
	year = {2020},
}

@inproceedings{CE22,
    author = {Chen, Yuansi and Eldan, Ronen},
    booktitle = {2022 IEEE 63rd Annual Symposium on Foundations of Computer Science (FOCS)},
    title = {Localization Schemes: A Framework for Proving Mixing Bounds for {Markov} Chains (extended abstract)},
    year = {2022},
    volume = {},
    number = {},
    pages = {110-122},
    doi={10.1109/FOCS54457.2022.00018},
}

@inproceedings{CF25,
    author = {Chen, Xiaoyu and Feng, Weiming},
    title = {Rapid Mixing via Coupling Independence for Spin Systems with Unbounded Degree},
    booktitle = {Approximation, Randomization, and Combinatorial Optimization. Algorithms and Techniques (APPROX/RANDOM 2025)},
    pages = {68:1--68:17},
    series = {Leibniz International Proceedings in Informatics (LIPIcs)},
    ISBN = {978-3-95977-397-3},
    ISSN = {1868-8969},
    year = {2025},
    volume = {353},
    editor = {Ene, Alina and Chattopadhyay, Eshan},
    publisher = {Schloss Dagstuhl -- Leibniz-Zentrum f{\"u}r Informatik},
    address = {Dagstuhl, Germany},
    URL = {https://drops.dagstuhl.de/entities/document/10.4230/LIPIcs.APPROX/RANDOM.2025.68},
    URN = {urn:nbn:de:0030-drops-244345},
    doi = {10.4230/LIPIcs.APPROX/RANDOM.2025.68},
}

@inproceedings{CFYZ22,
    author = {Chen, Xiaoyu and Feng, Weiming and Yin, Yitong and Zhang, Xinyuan},
    booktitle = {2022 IEEE 63rd Annual Symposium on Foundations of Computer Science (FOCS)}, 
    title = {Optimal mixing for two-state anti-ferromagnetic spin systems},
    year = {2022},
    volume = {},
    number = {},
    pages = {588-599},
    doi={10.1109/FOCS54457.2022.00062},
}

@inproceedings{CGSV21,
    author = {Chen, Zongchen and Galanis, Andreas and {\v{S}}tefankovi{\v{c}}, Daniel and Vigoda, Eric},
    title = {Rapid Mixing for Colorings via Spectral Independence},
    booktitle = {Proceedings of the 2021 ACM-SIAM Symposium on Discrete Algorithms (SODA)},
    chapter = {},
    pages = {1548-1557},
    doi = {10.1137/1.9781611976465.94},
    URL = {https://epubs.siam.org/doi/abs/10.1137/1.9781611976465.94},
    eprint = {https://epubs.siam.org/doi/pdf/10.1137/1.9781611976465.94},
    year = {2021},
}

@article{CLMM26,
    author = {Chen, Zongchen and Liu, Kuikui and Mani, Nitya and Moitra, Ankur},
    title = {Strong Spatial Mixing for Colorings on Trees and its Algorithmic Applications},
    year = {2026},
    publisher = {Association for Computing Machinery},
    address = {New York, NY, USA},
    issn = {0004-5411},
    url = {https://doi.org/10.1145/3837755},
    doi = {10.1145/3837755},
    note = {Just Accepted},
    journal = {J. ACM},
    month = aug,
}

@article{CLV23,
    author = {Chen, Zongchen and Liu, Kuikui and Vigoda, Eric},
    title = {Rapid Mixing of {Glauber} Dynamics up to Uniqueness via Contraction},
    journal = {SIAM Journal on Computing},
    volume = {52},
    number = {1},
    pages = {196-237},
    year = {2023},
    doi = {10.1137/20M136685X},
    URL = {https://doi.org/10.1137/20M136685X},
    eprint = {https://doi.org/10.1137/20M136685X},
}

@article{CLV26,
    author = {Chen, Zongchen and Liu, Kuikui and Vigoda, Eric},
    title = {Optimal Mixing of {Glauber} Dynamics: {Entropy} Factorization via High-Dimensional Expansion},
    journal = {SIAM Journal on Computing},
    volume = {55},
    number = {4},
    pages = {STOC21-104-STOC21-153},
    year = {2026},
    doi = {10.1137/21M1443340},
    URL = {https://doi.org/10.1137/21M1443340},
    eprint = {https://doi.org/10.1137/21M1443340},
}

@inproceedings{CZ23,
    author = {Chen, Xiaoyu and Zhang, Xinyuan},
    title = {A Near-Linear Time Sampler for the {Ising} Model with External Field},
    booktitle = {Proceedings of the 2023 Annual ACM-SIAM Symposium on Discrete Algorithms (SODA)},
    chapter = {},
    pages = {4478-4503},
    year = {2023},
    doi = {10.1137/1.9781611977554.ch170},
    URL = {https://epubs.siam.org/doi/abs/10.1137/1.9781611977554.ch170},
    eprint = {https://epubs.siam.org/doi/pdf/10.1137/1.9781611977554.ch170},
}

@inproceedings{EF23,
    author = {Efthymiou, Charilaos and Feng, Weiming},
    title = {On the Mixing Time of {Glauber} Dynamics for the Hard-Core and Related Models on {$G(n,d/n)$}},
    booktitle = {50th International Colloquium on Automata, Languages, and Programming (ICALP 2023)},
    pages = {54:1--54:17},
    series = {Leibniz International Proceedings in Informatics (LIPIcs)},
    ISBN = {978-3-95977-278-5},
    ISSN = {1868-8969},
    year = {2023},
    volume = {261},
    editor = {Etessami, Kousha and Feige, Uriel and Puppis, Gabriele},
    publisher = {Schloss Dagstuhl -- Leibniz-Zentrum f{\"u}r Informatik},
    address = {Dagstuhl, Germany},
    URL = {https://drops.dagstuhl.de/entities/document/10.4230/LIPIcs.ICALP.2023.54},
    URN = {urn:nbn:de:0030-drops-181064},
    doi = {10.4230/LIPIcs.ICALP.2023.54}
}

@inproceedings{Efthymiou14,
    author = {Efthymiou, Charilaos},
    title = {{MCMC} sampling colourings and independent sets of {$G(n, d/n)$} near uniqueness threshold},
    year = {2014},
    isbn = {9781611973389},
    publisher = {Society for Industrial and Applied Mathematics},
    address = {USA},
    booktitle = {Proceedings of the Twenty-Fifth Annual ACM-SIAM Symposium on Discrete Algorithms (SODA 2014)},
    pages = {305-316},
    numpages = {12},
    location = {Portland, Oregon},
}

@article{Efthymiou16,
    author = {Efthymiou, Charilaos},
    title = {A Simple Algorithm for Sampling Colorings of {$G(n,d/n)$} Up to The {Gibbs} Uniqueness Threshold},
    journal = {SIAM Journal on Computing},
    volume = {45},
    number = {6},
    pages = {2087-2116},
    year = {2016},
    doi = {10.1137/140977643},
    URL = {https://doi.org/10.1137/140977643},
    eprint = {https://doi.org/10.1137/140977643},
}

@article{Efthymiou20,
    title = {Deterministic counting of graph colourings using sequences of subgraphs},
    volume = {29},
    DOI = {10.1017/S0963548320000255},
    number = {4},
    journal = {Combinatorics, Probability and Computing},
    author = {Efthymiou, Charilaos},
    year = {2020},
    pages = {555-586},
}

@inproceedings{EHSV18,
    author = {Efthymiou, Charilaos and Hayes, Thomas P. and {\v{S}}tefankovi{\v{c}}, Daniel and Vigoda, Eric},
    title = {Sampling Random Colorings of Sparse Random Graphs},
    booktitle = {Proceedings of the 2018 Annual ACM-SIAM Symposium on Discrete Algorithms (SODA)},
    year = {2018},
    chapter = {},
    pages = {1759-1771},
    doi = {10.1137/1.9781611975031.115},
    URL = {https://epubs.siam.org/doi/abs/10.1137/1.9781611975031.115},
    eprint = {https://epubs.siam.org/doi/pdf/10.1137/1.9781611975031.115},
}

@article{FGYZ22,
    author = {Feng, Weiming and Guo, Heng and Yin, Yitong and Zhang, Chihao},
    title = {Rapid Mixing from Spectral Independence beyond the {Boolean} Domain},
    year = {2022},
    issue_date = {July 2022},
    publisher = {Association for Computing Machinery},
    address = {New York, NY, USA},
    volume = {18},
    number = {3},
    issn = {1549-6325},
    url = {https://doi.org/10.1145/3531008},
    doi = {10.1145/3531008},
    journal = {ACM Trans. Algorithms},
    month = Oct,
    articleno = {28},
    numpages = {32},
}

@book{FK15,
    place = {Cambridge},
    title = {Introduction to Random Graphs},
    publisher = {Cambridge University Press},
    author = {Frieze, Alan and Karo{\'{n}}ski, Micha{\l}},
    year = {2015},
}

@article{GSV15,
    author = {Galanis, Andreas and {\v S}tefankovi{\v c}, Daniel and Vigoda, Eric},
    title = {Inapproximability for Antiferromagnetic Spin Systems in the Tree Nonuniqueness Region},
    year = {2015},
    issue_date = {December 2015},
    publisher = {Association for Computing Machinery},
    address = {New York, NY, USA},
    volume = {62},
    number = {6},
    issn = {0004-5411},
    url = {https://doi.org/10.1145/2785964},
    doi = {10.1145/2785964},
    journal = {J. ACM},
    month = dec,
    articleno = {50},
    numpages = {60},
}

@article{Jonasson02,
    title = {Uniqueness of uniform random colorings of regular trees},
    journal = {Statistics \& Probability Letters},
    volume = {57},
    number = {3},
    pages = {243-248},
    year = {2002},
    issn = {0167-7152},
    doi = {https://doi.org/10.1016/S0167-7152(02)00054-8},
    url = {https://www.sciencedirect.com/science/article/pii/S0167715202000548},
    author = {Jonasson, Johan},
}

@article{JPV22,
    title = {Spectral independence, coupling, and the spectral gap of the {Glauber} dynamics},
    journal = {Information Processing Letters},
    volume = {177},
    pages = {106268},
    year = {2022},
    issn = {0020-0190},
    doi = {https://doi.org/10.1016/j.ipl.2022.106268},
    url = {https://www.sciencedirect.com/science/article/pii/S0020019022000254},
    author = {Jain, Vishesh and Pham, Huy Tuan and Vuong, Thuy-Duong},
}

@article{LM11,
    author = {Lucier, Brendan and Molloy, Michael},
    title = {The {Glauber} Dynamics for Colorings of Bounded Degree Trees},
    journal = {SIAM Journal on Discrete Mathematics},
    volume = {25},
    number = {2},
    pages = {827-853},
    year = {2011},
    doi = {10.1137/090779516},
    URL = {https://doi.org/10.1137/090779516},
    eprint = {https://doi.org/10.1137/090779516},
}

@book{LP17,
    author = {Levin, David A. and Peres, Yuval},
    title = {Markov Chains and Mixing Times},
    edition = {2},
    publisher = {American Mathematical Society},
    address = {Providence, RI},
    year = {2017},
    note = {With contributions by Elizabeth L. Wilmer},
    doi = {10.1090/mbk/107},
}

@article{LY26,
    title = {Spatial Mixing and Deterministic Approximate Counting of Multi-spin Systems beyond Bounded Degree Graphs},
    author = {Li, Zhidan and Yang, Kuan},
    year = {2026},
    eprint = {2609.12352},
    archivePrefix = {arXiv},
    primaryClass = {cs.DS},
    journal = {arXiv: 2609.12352},
}

@inproceedings{Martinelli99,
    author = {Martinelli, Fabio},
    editor = {Bernard, Pierre},
    title = {Lectures on {Glauber} Dynamics for Discrete Spin Models},
    booktitle = {Lectures on Probability Theory and Statistics: {\'E}cole d'{\'E}t{\'e} de Probabilit{\'e}s de Saint-Flour XXVII - 1997},
    year = {1999},
    publisher = {Springer Berlin Heidelberg},
    address = {Berlin, Heidelberg},
    pages = {93-191},
    isbn = {978-3-540-48115-7},
    doi = {10.1007/978-3-540-48115-7_2},
    url = {https://doi.org/10.1007/978-3-540-48115-7_2}
}

@article{MS10,
	author = {Mossel, Elchanan and Sly, Allan},
	journal = {Probability Theory and Related Fields},
	number = {1},
	pages = {37--69},
	title = {Gibbs rapidly samples colorings of {$G(n, d/n)$}},
	volume = {148},
	year = {2010},
}

@article{ZK07,
    title = {Phase transitions in the coloring of random graphs},
    author = {Zdeborov{\'a}, Lenka and Krzakala, Florent},
    journal = {Phys. Rev. E},
    volume = {76},
    issue = {3},
    pages = {031131},
    numpages = {29},
    year = {2007},
    month = {Sep},
    publisher = {American Physical Society},
    doi = {10.1103/PhysRevE.76.031131},
    url = {https://link.aps.org/doi/10.1103/PhysRevE.76.031131},
}
